\documentclass[12pt,oneside,reqno]{amsart}

\usepackage[a4paper,left=25mm,top=32mm,right=25mm,bottom=32mm]{geometry}

\usepackage{mathtools}

\mathtoolsset{showonlyrefs}

\usepackage{hyperref} 
\hypersetup{
	colorlinks=true,
	linkcolor=blue,
	filecolor=blue,
	citecolor = blue,      
	urlcolor=cyan,
}

\usepackage{amssymb,amsfonts}
\usepackage[all,arc]{xy}
\usepackage{enumerate}
\allowdisplaybreaks
\usepackage{siunitx}
\usepackage{enumitem}
\usepackage{mathtools}
\usepackage{mathrsfs}
\usepackage{color}
\usepackage{setspace}
\usepackage{dsfont}
\usepackage{amsfonts}
\usepackage{amssymb}
\usepackage{graphicx}
\usepackage{ragged2e}
\usepackage{amsthm}
\usepackage[normalem]{ulem}
\usepackage{amsmath}
\usepackage[mathscr]{euscript}
 \let\mathscr\relax
\usepackage[scr]{rsfso}
\usepackage{graphicx}
\usepackage{color}
\usepackage{float}
\usepackage{caption}
\usepackage{comment}
\usepackage{tikz}

\usetikzlibrary{shapes,arrows.meta,calc,fit,backgrounds,shapes.multipart,positioning}

\tikzset{box/.style={draw, rectangle, rounded corners, thick, node 
		distance=7em, 
		text width=6em, text centered, minimum height=3.5em}}
\tikzset{every node/.style={font=\scriptsize}}

\definecolor {processblue}{cmyk}{0.96,0,0,0}

\DeclareMathOperator{\Div}{Div}

\DeclareMathOperator{\vol}{Vol}

\DeclareMathOperator{\re}{Re}

\DeclareMathOperator{\diam}{Diam}

\newtheorem*{rep@theorem}{\rep@title}
\newcommand{\newreptheorem}[2]{%
	\newenvironment{rep#1}[1]{%
		\def\rep@title{#2 \ref{##1}}%
		\begin{rep@theorem}}%
		{\end{rep@theorem}}}
\makeatother

\newtheorem{thm}{Theorem}[section]
\newreptheorem{thm}{Theorem}
\newtheorem{cor}[thm]{Corollary}

\newreptheorem{cor}{Corollary}

\newtheorem{prop}[thm]{Proposition}
\newreptheorem{prop}{Proposition}

\newtheorem{lem}[thm]{Lemma}
\newreptheorem{lem}{Lemma}
\newtheorem{conj}[thm]{Conjecture}

\newtheorem{claim}[thm]{Claim}

\theoremstyle{definition}
\newtheorem{defn}[thm]{Definition}

\newtheorem{exmp}[thm]{Example}

\theoremstyle{remark}
\newtheorem{rem}[thm]{Remark}

\numberwithin{equation}{section}

\newcommand{\thalf}{\tfrac{1}{2}}

\newcommand{\spc}{\\[0.3em]}

\newcommand{\ra}{\rightarrow}
\newcommand{\rootn}{\sqrt{n}}

\newcommand{\real}{\mathbb{R}}
\newcommand{\nat}{\mathbb{N}}
\newcommand{\ep}{\varepsilon}
\newcommand{\varep}{\varepsilon}

\newcommand{\lb}{\left(}
\newcommand{\rb}{\right)}

\newcommand{\col}{\mathscr{C}}

\newcommand{\N}{\mathcal{N}}

\newcommand{\tun}{\mathcal{T}}

\begin{document}
	\onehalfspacing
	\title[Effective lower bound]{Almost sharp lower bound for the nodal volume of harmonic functions.}
	
	\author{Alexander Logunov, Lakshmi Priya M.E., Andrea Sartori }
	
	\address[Andrea Sartori]{Departement of Mathematics, Tel Aviv University, Tel Aviv, Israel, IL}
	\email[Andrea Sartori]{sartori.andrea.math@gmail.com}
	\address[Lakshmi Priya]{Departement of Mathematics, Tel Aviv University, Tel Aviv, Israel, IL}
	\email[Lakshmi Priya]{lpriyame@gmail.com}
	\address[Alexander Logunov]{Université de Genève, Section de mathématiques, rue du Conseil-Général 7-9, 1205 Genève, Switzerland and Department of Mathematics, Princeton University, Princeton, NJ, USA}
	\email[Alexander Logunov]{log239@yandex.ru}
\begin{abstract}
This paper focuses on a relation between the growth of harmonic functions and the Hausdorff measure of their zero sets.		
	Let $u$ be a real-valued harmonic function in $\mathbb{R}^n$ with $u(0)=0$ and $n\geq 3$. We prove
	$$\mathcal{H}^{n-1}(\{u=0\} \cap B(0,2)) \gtrsim_{\varep} N^{1-\varep},$$
	where the doubling index $N$ is a notion of growth defined by
	$$ \sup_{B(0, 1)}|u| = 2^N \sup_{B(0,\frac{1}{2})}|u|.$$
	This gives an almost sharp lower bound for the Hausdorff measure of the zero set of $u$, which is conjectured to be linear in $N$. The new ingredients of the article are the notion
	of \emph{stable growth}, and a multi-scale induction technique for a lower bound for
	the distribution of the doubling index of harmonic functions. It gives a significant improvement over the previous best-known bound
	$\mathcal{H}^{n-1}\left(\{u=0\} \cap 2B\right)\geq \exp (c \log N/\log\log N )$,
	which implied Nadirashvili's conjecture.

\end{abstract}

	\maketitle 
	
	\section{Introduction}\label{sec1}

Given a non-constant harmonic function $u$ in $\mathbb{R}^n$, its zero set $\{u=0\}$ is non-empty and has Hausdorff dimension $(n-1)$. Locally the geometry and complexity of the zero set of $u$ is controlled in terms of its growth. In particular, one can bound the $(n-1)$ dimensional Hausdorff measure of  the set $\{u=0\}$  in every Euclidean ball $B$ in terms of growth of $u$ around $B$. 

\vspace{2mm}

One of the multiple ways to quantify the growth of a function is given by the notion of the doubling index. Given a positive number $k$ and an Euclidean ball $B=B(x,r)$ with center at $x$ and radius $r>0$, we denote by $kB$ the scaled ball $B(x,kr)$. 
For any function $h$ in a ball $2B\subset \mathbb{R}^n$, the \textit{doubling index} of $h$ in $B$ is defined by
$$\mathcal{N}_h(B):= \log_2 \frac{\sup_{2B}|h|}{\sup_B|h|}.$$
We often write $\mathcal{N}(x,r)$ instead of $\mathcal{N}_h(B(x,r))$ and often omit the dependence on $h$ in the notation and simply write 
$ \mathcal{N}(B).$	 

\begin{thm}[\cite{DF88}, \cite{H07}] \label{Thm.Han} 
	Let $B\subset \mathbb{R}^n$ be a unit ball. There exists a constant $C=C(n)>1$ such that 
	$$\mathcal{H}^{n-1}\left(\{u=0\} \cap B\right)\leq C \mathcal{N}\left(B\right),$$
	for all harmonic functions $u: 2B \rightarrow \mathbb{R}$.
\end{thm}
The main result of this paper states that there is also a lower bound in terms of growth.
\begin{thm}[Main theorem]
	\label{thm: main}
	Let $B\subset \mathbb{R}^n$ be a unit ball and $n\geq 3$. For every $\varepsilon>0$, there exists a constant $c=c(n,\varepsilon)>0$ such that for every harmonic function  $u: 4B \rightarrow \mathbb{R}$ with $u(0)=0$, we have 
	$$\mathcal{H}^{n-1}\left(\{u=0\} \cap 2B\right)\geq c \mathcal{N}\left(\frac{1}{2}B\right)^{1-\varepsilon}.$$
\end{thm}

Let us make a few remarks. Given a unit ball $B\subset \mathbb{R}^n$, Harnack's inequality for harmonic functions implies that if $\mathcal{N}_u(\frac{1}{2}B)$ is sufficiently large,  then $2B$ must contain a zero point of $u$. If the harmonic function $u$ is zero at the center of the ball $B$, then 
\begin{align}
\nonumber	\mathcal{N}(B)\geq c_n
\end{align}
for some numerical constant $c_n > 0$ (see Claim \ref{claim: lower bound doubling index} in Appendix \ref{sec: appendix} for the proof). Hence, in Theorem \ref{thm: main} there is a uniform lower bound, which was conjectured by Nadirashvili.

\vspace{2mm}

\textbf{Nadirashvili's conjecture}. {\it If the harmonic function $u$ is zero at the center of the unit ball $B\subset \mathbb{R}^3$, then the area of the zero set of $u$ in $B$ is bounded from below by a positive numerical constant.}

\vspace{2mm}

Nadirashvili's conjecture was recently proved in \cite{Lproof}. The fact sounds elementary, but there is no simple proof known, and there are applications including the lower bound in Yau's conjecture for nodal sets of Laplace eigenfunctions. We refer to \cite{Lmreview,Ybook}  for the statement of Yau's conjecture and related results.

\vspace{2mm}

This article does not pursue the purpose of generality. We restrict our attention only to the Euclidean case.   There isn't a methodological obstacle preventing to extend the main result to the case of smooth Riemannian manifolds, though we don't do this job. We hope that this decision makes the article more accessible for the reader.

\vspace{2mm}

One can compare this paper to the approach for Nadirashvili's conjecture in \cite{Lproof}, which gave a slowly growing bound
$$\mathcal{H}^{n-1}\left(\{u=0\} \cap 2B\right)\geq \exp (c \log N/\log\log N ),$$
where $N= \mathcal{N}(\frac{1}{2}B)$ is assumed to be bigger than $10$ just to make the right hand side well-defined and $u$ vanishes at the center of $B$. Compared to \cite{Lproof}, this article introduces the new notion of \emph{stable growth} and a new multi-scale induction technique, which gives significantly improved lower bounds and allows to come closer to the following folklore conjecture.
\begin{conj}
	\label{Conj1}
	Given a unit ball $B\subset \mathbb{R}^n$, $n\geq 3$, and a harmonic function $u$ in $4B$, which is zero at the center of $B$, the following lower bound holds:
	\begin{align}
		\nonumber
		\mathcal{H}^{n-1}\left(\{u=0\} \cap 2B\right) \geq c \mathcal{N}\left(\frac{1}{2}B\right),
	\end{align}
	for some numerical constant $c=c(n)>0$. 
\end{conj}
Before discussing the conjecture let us mention a few properties of the doubling index.	The doubling index is scale invariant. 
A non-trival fact states that the doubling index for a harmonic function $u$ is almost monotone function of a set in the sense that $$\mathcal{N}_u(b) \leq C \mathcal{N}_u(B)$$
if a ball $B$ contains the ball $2b$. Here $C>1$ is a dimensional constant, but one can find a sharper statement in \S \ref{sec3} about the doubling index. The monotonicity property is a powerful tool often used in geometrical analysis. However,
besides almost monotonicity it is unclear how the doubling index behaves as a function of a set. Rescaling Conjecture \ref{Conj1} and Theorem \ref{Thm.Han}, one can
see that {\it scaled doubling index} 
$$\mathcal{SN}(B):=\mathcal N(B) \textup{radius}(B)^{n-1}$$ is comparable to the Hausdorff measure of the zero set of $u$ in the following sense.

\vspace{2mm}

\textbf{Rescaled Theorem \ref{Thm.Han}.}
\begin{align}
	\mathcal{H}^{n-1}\left(\{u=0\} \cap B\right)\leq C \cdot  \mathcal{SN}(B).
\end{align}	

\textbf{Rescaled Conjecture \ref{Conj1}}.
\begin{align}
	\label{conjecture additivity} 
	\mathcal{SN}\left(\frac{1}{2}B\right)\leq C \cdot  \mathcal{H}^{n-1}\left(\{u=0\} \cap 2B\right) + C \cdot \textup{radius}(B)^{n-1}.
\end{align}	
The extra additive term $C \cdot \textup{radius}(B)^{n-1}$ is required ($u$ can be positive and have no zeroes in $4B$), but as a rule it plays no role. 

\vspace{2mm}

A challenge in nodal geometry is to understand the distribution of the  doubling index as a function of a set. If Conjecture \ref{Conj1} is true and the scaled doubling index for harmonic functions is comparable to the Hausdorff measure of the zero set, then one can conclude that the scaled doubling index is almost additive as a function of a set. The main Theorem \ref{thm: main} is likely to have applications in nodal geometry as it gives a weak version of almost additivity for scaled doubling index with subpolynomial error term.

\vspace{2mm}

In this paper we introduce a new notion of \textit{stable growth}. At some moment in this paper we will choose a dimensional constant $C=C(n)>1$ (its choice will be presented in Remark \ref{rem: choice of constant}).  A harmonic function $u:\mathbb{R}^n\ra \real$ is said to have \textit{stable growth}  in a ball $B\subseteq \real^n$ if 
\begin{align}
	C^{-1}\N(\tfrac{1}{2}B) \leq \N(2B) \leq   C \N(\tfrac{1}{2}B).
\end{align}
We note that the first inequality holds due to almost monotonicity of the doubling index (see Lemma \ref{lem: almost monotonicity} below) and can be omitted in the definition of stable growth. The assumption of stable growth is usually earned for free or for a small price of a subpolynomial error term, but it simplifies many of the proofs.  

\vspace{2mm}

Conjecture \eqref{Conj1} together with Theorem \ref{Thm.Han} imply that if a harmonic function $u$ has stable growth  in a unit ball $B$ and $A$ is a constant such that  $\mathcal{N}(B) \gg A>1$, then on the scale $1/A$ the following property of the doubling index holds:
\begin{align}
	\label{eq: additivity}	\int_B \mathcal{N}\left(x, A^{-1}\right)dx \asymp \frac{\mathcal{N}(B)}{A}.
\end{align}
In other words, it is anticipated that the average doubling index drops proportionally to the scale factor.

\vspace{2mm}

\section*{Acknowledgment.}
The authors are indebted to Mikhail Sodin for his enormous patience and numerous fruitful discussions that motivated this work. The authors  are grateful to Tel Aviv University, Princeton University and University of Geneva, where this work was conducted. 
A.L. was supported in part by the Packard Fellowship, Swiss NSF grant 200020-200400 and by NCCR SwissMAP (Swiss NSF grant 205607).  L.P.M.E. was supported by ERC Advanced Grant 692616, ERC consolidator grant 101001124 (UniversalMap) and ISF grant 1294/19.  A.S. was supported by the ISF Grant 1903/18 and the BSF Start up Grant no. 20183.

\subsection{Notation}
As mentioned above, given a ball $B=B(x,r)\subset \mathbb{R}^n$ with center at $x\in \mathbb{R}^n$ and radius $r>0$, we denote by $\ell B$ the ball $\ell B= B(x,\ell\cdot r)$ for a given scaling parameter $\ell>0$.  Sometimes we will not need to explicate the center of the ball and we will simply write $B(x,r)= B(r)$. We also write $\overline{B}$ to denote a closed ball.  Similarly, given a cube $Q\subset \mathbb{R}^n$ we write $\ell Q$ for the cube which is the homothetic copy of $Q$ with same center and with homothety coefficient $\ell>0$. We define the \textit{maximal doubling index} of a cube $Q$ by 
\begin{align}\label{defn_di_cube}\mathcal{N}^{*}(Q) := \sup_{\substack{x \in Q \\ 0<\rho\leq \diam(Q) }} \mathcal{N}(B(x,\rho)).\end{align}
We denote by $c,c_1,c_2,\ldots$ (small) positive constants and by $C,C_1,C_2,\ldots $ (large) positive constants,  which may change from line to line. Given two quantities $A,B$, we write $A\lesssim B$ and $A\gtrsim B$ to designate the existence of two constants $c,C>0$ such that $A\leq CB$ and $A\geq cB$, respectively. Here the constants $c, C$ are allowed to depend on the dimension $n$. 
If these constants depend on some auxiliary parameter $\gamma$, we write $A\lesssim_{\gamma} B$ and $A\gtrsim_{\gamma} B$. If $A\lesssim B$ and $B\lesssim A$, we write $A\asymp B$ and $A\asymp_{\gamma} B$ if one (or both) constants implied in the notation depend on the auxiliary parameter $\gamma$.  Finally, we write $A\ll B$ to mean that, for a sufficiently small constant $c>0$, we have $A\leq cB$ and $A\gg B$ if, for a sufficiently large constant $C\geq1$, we have   $A\geq CB$. Again, we write $A\ll_{\gamma} B$ and $A\gg_{\gamma} B$ if the constants implied in the notation depend on the parameter $\gamma$. Sometimes we will just write $\asymp$, $\lesssim$ and $\ll$ instead of $\asymp_n$, $\lesssim_n$ and $\ll_n$.

Given some $x\in \mathbb{R}$, we write $\lfloor x\rfloor$ for the largest integer smaller than $x$ and $\lceil x \rceil$ for the smallest integer larger than $x$.

Given a finite set $F$, we denote by $\# F$ the number of elements in $F$. For a set $S \subset \real^n$ and $\ell >0$, define the $\ell$-neighbourhood of $S$ by
\begin{align*}
S_{+\ell} := \{x \in \real^n : x= s+b\text{, where }s\in S\text{ and }b\in B(0,\ell) \}. 
\end{align*}
  The zero sets of harmonic functions are often called \textit{nodal} sets. We will sometimes refer to the Hausdorff measure of the nodal set as \textit{nodal volume}. When we say that a harmonic function $u$ has \textit{bounded doubling index} in the ball $B$, we mean that there exists a numerical constant $C>0$ such that 
  $$\mathcal{N}_u(B)\leq C.$$

\section{First thoughts and observations}\label{sec_2}

This section discusses elementary ideas and some examples for building intuition about the zero sets and the doubling index of harmonic functions. 


\subsection{Bounded doubling index } \label{sec2.1}
The following simple observation gives a lower bound for the nodal volume of a harmonic function in terms of its doubling index. We note that this bound is useful only when the doubling index is not large. 
\begin{lem}
	\label{claim: bounded doubling index} 
	Let $B=B(0,1)\subset \mathbb{R}^n$ be the unit ball and let $u$  be a harmonic function in  $2B$. Suppose that $u(0)=0$, then  there exists a constant $c=c(n)>0$ such that 
	$$\mathcal{H}^{n-1}\left(\{u=0\} \cap 2B \right)\geq c \mathcal{N}(B)^{1-n}.$$
\end{lem}
 Before starting the proof, we formulate a consequence of Harnack's inequality. 
\begin{claim}\label{claim:harnack_consequence} Let $B=B(x,r) \subset \real^n$ be any ball. Then  for any non-constant harmonic function $u$ defined in $2B$ satisfying $u(x) \geq 0$, we have
	\begin{align}\label{eq:conclusion}
		\sup_{\tfrac{4}{3}B} u \gtrsim_{n} \sup_{B} |u|. 
	\end{align}
\end{claim}
\begin{proof}[Proof of Claim \ref{claim:harnack_consequence}] Let $M$ and $m$ denote the supremum and the infimum of $u$ in $B$. If $m\geq 0$, then the conclusion \eqref{eq:conclusion} holds trivially. Let ${M'}$ denote the supremum of $u$ in $\tfrac{4}{3}B$. It suffices to consider the case $m <0$, and show that $-m \lesssim_n M'$. 
	
	 Let $x_m$, $x_M \in \partial B$ be such that $u(x_m) = m$ and $u(x_M) = M$. Consider the  positive harmonic function  $h := M' -u$  in $\tfrac{4}{3}B$. By Harnack's inequality, there is a constant $C=C(n) >1$ such that 
	$$(M' - m)	= h(x_m) \leq C h(x_M) = C(M' - M) \leq CM'.$$
	Hence,
	$$-m \lesssim_n M',$$	
	and the claim follows.
\end{proof}


\begin{proof}[Proof of Lemma \ref{claim: bounded doubling index}]
	We will estimate the Hausdorff measure of the zero set of $u$ in  the spherical layer $\{{4}/{3}\leq |x|\leq {5}/{3}\}$. Let $N:= {20}\lceil \mathcal{N}(B)\rceil$. Consider the behavior of $u$ in the concentric spheres
	\begin{align}
		&	S_i= \partial ( r_i B )= \{x\in 2B : |x|=r_i\},\text{ where } r_i= \frac{4}{3} + \frac{i}{3N} \nonumber
	\end{align}
	for $i=0,\ldots,N$ and define
	$$m_{i}^+ := \max_{S_i}u,  \text{ } \text{ ~and~ } \text{ } m_{i}^- := \min_{S_i}u.$$
We may assume that $u$ is not identically zero.	Since $u(0)=0$, the maximum principle  implies 
	\begin{align}
		m_{i}^+>0,  \hspace{7mm} m_{i}^-<0, \hspace{7mm} \nonumber m_i^+ <  m_{i+1}^+, \hspace{3.5mm}\text{ and } \hspace{3.5mm} |m_i^-| < |m_{i+1}^-|. 
	\end{align}
	We aim to show that there are two balls of radius $c_0 N^{-1}$ in $B$, for some $c_0=c_0(n)>0$, such that  $u$ is positive in one ball and negative in the other ball. The following geometrical fact shows that existence of two such balls guarantees a lower bound for the nodal volume.

	\begin{claim}\label{euclidean_geom_claim}
		Let $f:D \ra \real$ be a continuous function in a convex set $D\subset \mathbb{R}^n$.  Assume that  there are two balls $B^{+}$, $B^{-} \subset D$ both of radius $r>0$ such that $f>0$ in $B^{+}$ and $f<0$ in $B^{-}$. Then
		\begin{align*}
			\mathcal{H}^{n-1} (\{f=0\}) \gtrsim_n r^{n-1}. 
		\end{align*}
	\end{claim}
\begin{proof}[Proof of Claim \ref{euclidean_geom_claim}]
		 Consider line segments starting in $B^{+}$ and ending in $B^{-}$, which are parallel to the line connecting the centers of the balls $B^{+}$ and $B^{-}$. Since $f$ is positive at one end and negative at the other end, every such line segment contains a zero of $f$. Consider the orthogonal projection of the zero set $\{f=0\}$ onto a hyperplane orthogonal to these line segments. The zeros of $f$ in the line segments project onto a $(n-1)$ dimensional ball of radius $r$ in this hyperplane. Since the orthogonal projection does not increase the distances, the  Hausdorff measure of $\{f=0\}$ is at least the Hausdorff measure of its projection, which contains a $(n-1)$ dimensional ball of radius $r$.  This completes the proof of Claim \ref{euclidean_geom_claim}.
	\end{proof}
	\begin{figure}[H]
		\includegraphics[scale=0.8]{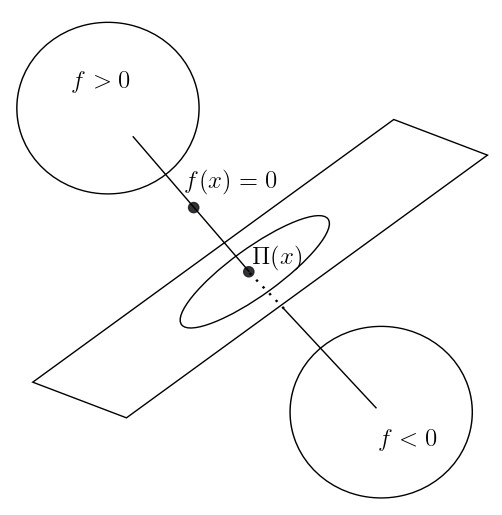}
		\centering
		\caption{Pictorial proof of Claim \ref{euclidean_geom_claim}}
	\end{figure}

	We are ready to proceed with the proof of Lemma \ref{claim: bounded doubling index}. 
	We will show that there are at least $N/10$ disjoint  spherical layers $\{x\in 2B: r_{i-1}\leq |x|\leq r_i\}$ such that in each of them, we can find two balls of radius $c_0 N^{-1}$, for some $c_0=c_0(n)>0$, so that $u$ is positive in one of the balls and negative in the other.

	The increment of $u$ from $\tfrac{4}{3}B$ to $\tfrac{5}{3}B$ can be written as a telescopic product
	$$\frac{m_{1}^+}{m_{0}^+} \cdot \frac{m_2^+}{m_1^+} \cdots \frac{m_{N}^+}{m_{N-1}^+} = \frac{\sup_{\frac{5}{3}B}u}{\sup_{\frac{4}{3} B}u}.$$
	 By Claim \ref{claim:harnack_consequence} and recalling that $N= {20}\lceil \mathcal{N}(B)\rceil$, we have  
	\begin{align*} \frac{\sup_{\frac{5}{3}B}u}{\sup_{\frac{4}{3} B}u} \lesssim_n \frac{\sup_{2B}|u|}{\sup_{B}|u|} =  2^{\mathcal{N}(B)} \leq e^{N}.
	\end{align*}
	 Similarly,
	\begin{align*}
		\frac{m_{1}^-}{m_{0}^-} \cdot \frac{m_2^-}{m_1^-} \cdots \frac{m_{N}^-}{m_{N-1}^-} \lesssim_n e^{ N}.
	\end{align*}
	Hence we conclude that there exists $C_1=C_1(n) >0$ such that $|m_{i+1}^-|\leq C_1 |m_i^-|$ holds for at least $\frac{3}{4}N$ values of $i$ (and similarly for $m_i^+$). Thus, we have
	\begin{align}
		\label{eq: bound needed}	m_{i+1}^+\leq C_1 m_i^+ \text{ } \text{ and } \text{ }  |m_{i+1}^-|\leq C_1 |m_i^-|,
	\end{align}
for at least $N/2$  values of $i$.		From now on, we will consider only $i$ satisfying \eqref{eq: bound needed}. 	Let $x_{i}^{+}$ and $x_{i}^{-} \in S_i$ be points where $u$ attains its maximum and minimum respectively, i.e., $u(x_i^+) = m_i^+$ and $u(x_i^-) = m_i^-$.  Define 
	\begin{align*}
		B_{i}^{+} := B(x_{i}^{+}, \tfrac{1}{10N})\text{ }\text{ and }\text{ }B_{i}^{-} := B(x_{i}^{-}, \tfrac{1}{10N}).
	\end{align*}
	We will now show that there is a small constant $c_0>0$ such that  $u >0$ in $c_0 B_{i}^{+}$, and $u <0$ in $c_0 B_{i}^{-}$, by estimating the gradient of $u$ as follows:
	\begin{align*}
		\sup_{B_{i}^{+}} |\nabla u| \lesssim_{n} N \sup_{2B_{i}^{+}} |u| \lesssim_n {N} \sup_{\tfrac{8}{3}B_{i}^{+}}u \leq {N} \sup_{S_{i+1}} u = {N} m_{i+1}^{+} \mathop{\leq}\limits_{\eqref{eq: bound needed}} C_2 {N}  m_{i}^{+},
	\end{align*}
	where the first inequality follows from the standard gradient estimates (Claim \ref{claim:grad_estimates}), the second inequality follows from  Claim \ref{claim:harnack_consequence}, and the third inequality follows from the inclusion $\tfrac{8}{3}B_{i}^{+} \subset r_{i+1}B$.
	All in all, we have  
	\begin{align*}
		u(x_{i}^{+}) = m_{i}^{+} \text{ } \text{ and } \text{ }\sup_{B_{i}^{+}}|\nabla u| \lesssim_n {N}m_{i}^{+}, \text{ } \text{ where } \text{ }	B_{i}^{+} := B(x_{i}^{+}, \tfrac{1}{10N}).
	\end{align*}
	We may conclude that there is a small constant $c_0 \in (0,1)$ such that $u >0$ in $c_0 B_{i}^{+}$. A similar argument implies that (for a small $c_0\in (0,1)$) we have $u <0$ in $c_0B_{i}^{-}$. Finally, the Euclidean Claim \ref{euclidean_geom_claim} implies the lower bound for the nodal volume. 
\end{proof}
\begin{rem} One can modify the argument in  Lemma 2.1 to prove a slightly better estimate for the nodal volume:
	\begin{align*}
		\mathcal{H}^{n-1}(\{u=0\}\cap 2B) \gtrsim \N(B)^{2-n},
	\end{align*}
	by showing that 
	\begin{align*}
		\mathcal{H}^{n-1}\left( \{u=0\} \cap \{x\in 2B: r_{i-1}\leq |x|\leq r_{i+1}\}\right)\gtrsim  \frac{1}{N^{n-1}} 
	\end{align*}
	for at least $N/2$ indices $i$. 
\end{rem}

We recall from \S\ref{sec1}  the known upper bound  and the conjectured lower bound for the nodal volume of a harmonic $u$ vanishing at the center of a ball $B \subset \real^n$
\begin{align*}
	\N(\tfrac{1}{2}B)r^{n-1} \stackrel{\text{Conj.\ref{Conj1}} }{\lesssim_{n}} \mathcal{H}^{n-1}(\{u=0\} \cap B) \stackrel{\text{Thm.\ref{Thm.Han}} }{\lesssim_{n}} \N(B) r^{n-1}. 
\end{align*}
In the case when $u$ has bounded doubling index in $B$ 
$$\mathcal{N}\left(\frac{1}{2}B\right)\asymp \mathcal{N}(B)\asymp 1,$$
Theorem \ref{Thm.Han} and Lemma \ref{claim: bounded doubling index} imply that the nodal volume of $u$ is comparable to the surface area of $B$. 

\begin{claim}\label{claim_sec2_2.4} 
	Let $B \subset \real^n$ be a ball and let $C>1$ be some constant. Then for all harmonic functions $u$ in $2B$, vanishing at the center of $B$, such that $\N_u(B) \leq C$, we have
	\begin{align*}
		\mathcal{H}^{n-1}(\{u=0\}\cap B) \asymp_{C,n} r^{n-1}.
	\end{align*}
	The nodal volume is also comparable to the scaled doubling index:
	\begin{align*}
	\mathcal{H}^{n-1}(\{u=0\}\cap B) \asymp_{C,n} \mathcal{SN}(B).
	\end{align*}
\end{claim}


\subsection{Chopping} \label{sec2.2}

Given a harmonic function $u$ vanishing at the center of a ball $B\subset \real^n$, how can we estimate the volume of the zero set  when the doubling index is large? 

\vspace{2mm}

When we start looking at the zero set at small scales, the doubling index in small balls around the zero drops down asymptotically in the following way:
\begin{align*}
	\lim_{\ep \ra 0} \N(B(x,\ep)) = \text{ order of vanishing of $u$ at $x$.}
\end{align*}
The latter limit is $1$ for all $x$ in the zero set except the singular set, that is the set where $u$ and $|\nabla u|$ vanish simultaneously. The singular set does not contribute to the nodal volume as it has Hausdorff dimension $(n-2)$, see \cite{HHL98,HHHN99}, and also \cite{NV17} for a quantitative estimate. The non-singular zero set can be covered by countably many balls  with bounded doubling index. 

\vspace{2mm}

In order to prove a lower bound for the nodal volume, a naive idea is to find small disjoint balls $B_k = B(x_k,r_k)$ in $B$ with doubling index $\N(B_k) \leq C$ such that the harmonic function has a zero at the center of each $B_k$. Then in each of the $B_k$ we have a lower bound for the nodal volume by Claim \ref{claim_sec2_2.4}, and therefore
\begin{align}\label{eq:chop.2}
	\mathcal{H}^{n-1}(\{u=0\}\cap B) \geq  \sum_{B_k} \mathcal{H}^{n-1}(\{u=0\}\cap B_k) \asymp \sum_{B_k} r_{k}^{n-1}.
\end{align}
To find balls with bounded doubling index, one can attempt to use the following idea of \textit{chopping}. 
Take a cube $Q$ (which is easier to chop than balls) and repeatedly divide it into smaller subcubes $\{Q_k\}$. At any point in this process if we have a subcube with bounded doubling index, we stop chopping, and if a subcube has large doubling index we continue chopping. Finally, in subcubes with small doubling index and a zero of $u$, one can use a version of Claim \ref{claim_sec2_2.4} for cubes. 

\vspace{2mm}

There are two problems to implement this idea. The first one is to show that for most of the subcubes the doubling index drops down. The second problem is to get some lower bound for the sum of scaled doubling indices 
\begin{align}\label{total_sdi}
	\sum \N(Q_k) (\text{diam $Q_k$})^{n-1},
\end{align}
where the sum is over subcubes $Q_k$ containing a zero of $u$. We choose to work with \eqref{total_sdi} as it imitates the nodal volume.

\vspace{2mm}

In order to obtain quantitative bounds for the nodal volume, one has to understand the distribution of the doubling index on many scales. More precisely, given a partition of cube $Q \subset \real^n$ into equal subcubes and a harmonic function in $2Q$, what can we say about the doubling indices in these subcubes compared to the doubling indices of $Q$ and $\frac{1}{4}Q$?  

\vspace{2mm}

In the next section we discuss two simple examples of the distribution of doubling index.  Old and new results about the distribution of doubling index of harmonic functions will be presented in \S \ref{sec4}

\subsection{Examples} 

\begin{exmp}
	We consider the harmonic function $u(x,y)= \exp (Nx)\sin (Ny)$, for some large positive integer $N$.  Partition $Q=[0,1]^2$ into $A^2$ many equal subcubes $\{Q_i\}_i$ of side length $A^{-1}$. If $1 \ll A \ll N$, then 
	\begin{align*}
		\mathcal{N}^{*}(Q)\asymp N \text{ and }\mathcal{N}^{*}(Q_i)\asymp \frac{N}{A}.
	\end{align*}
	In this case, the doubling index is a linearly decreasing function of the size of the subcubes. The zero set of $u$ consists of parallel lines, see Figure \ref{figure1} below. 
\end{exmp}

\begin{figure}
	\includegraphics[scale=0.4]{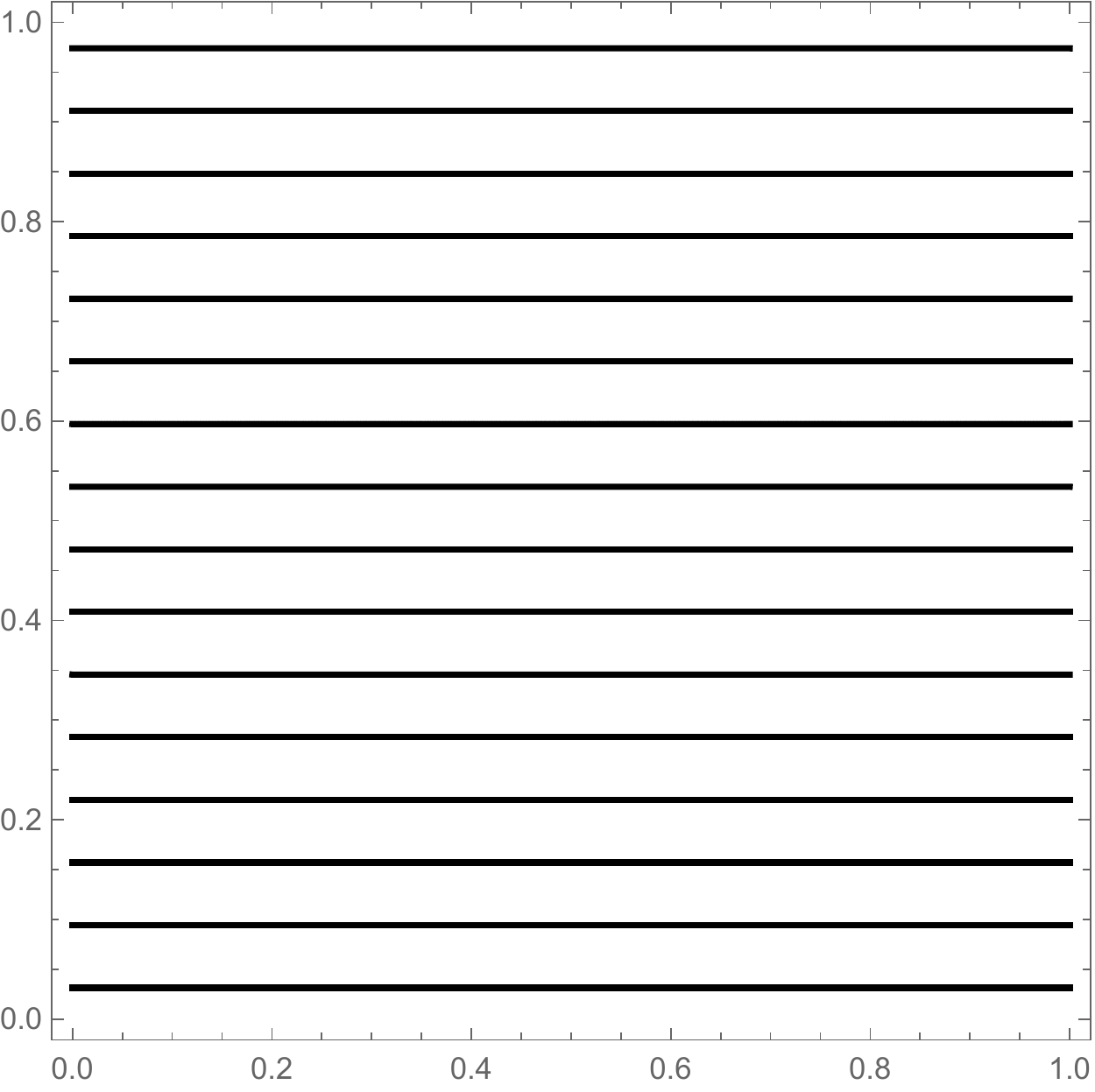}
	\centering
	\caption{Nodal set of  $u(x,y)= \exp(100x)\cos(100y)$.}\label{figure1}
\end{figure}

\begin{exmp}
	Consider $u(x,y)= \re((x+iy)^N)$ for some large integer $N>1$. 
	The zero set of the function $u$ is the union of $2N$ rays emanating from the origin, see Figure \ref{figure2} below. We consider $Q=[-2,2]^2$ and partition it into $A^2$ equal subcubes $\{Q_i\}_i$. 
	
	We have $\mathcal{N}^{*}(Q)\asymp N$. 
	Unlike the previous example, the doubling index of the subcubes depend on the proximity of the subcube to the origin, which is a singular zero of $u$. Consider $z\in Q$ with $|z| = r>0$, let us now get an estimate for $\N(z,1/A)$. In polar coordinates, $u$ can be written as $u(r,\theta) = r^N \cos N\theta$. Hence, for $1 \ll A\ll N$, we have 
	\begin{align*}
		\log	\frac{\sup_{B(z,2/A)}|u|}{\sup_{B(z,1/A)}|u|} \asymp \log \frac{(r+\tfrac{2}{A})^N}{(r+\tfrac{1}{A})^N} = \log \frac{(1+\tfrac{2}{rA})^N}{(1+\tfrac{1}{rA})^N}.
	\end{align*}
	If $(1/rA) \ll 1$, or equivalently $r\gg 1/A$, then we have
	\begin{align*}
		\N(z,1/A) \asymp \frac{N}{rA}. 
	\end{align*}
	Thus, for subcubes $Q_i$ close to the origin, we have $\N^{*}(Q_i) \asymp N$. And for subcubes away from the origin, e.g. the ones in $B(0,1)\backslash B(0,1/2)$, we have $\N^{*}(Q_i) \asymp N/A$.

\end{exmp}

\begin{figure}[H]
	\includegraphics[scale=0.4]{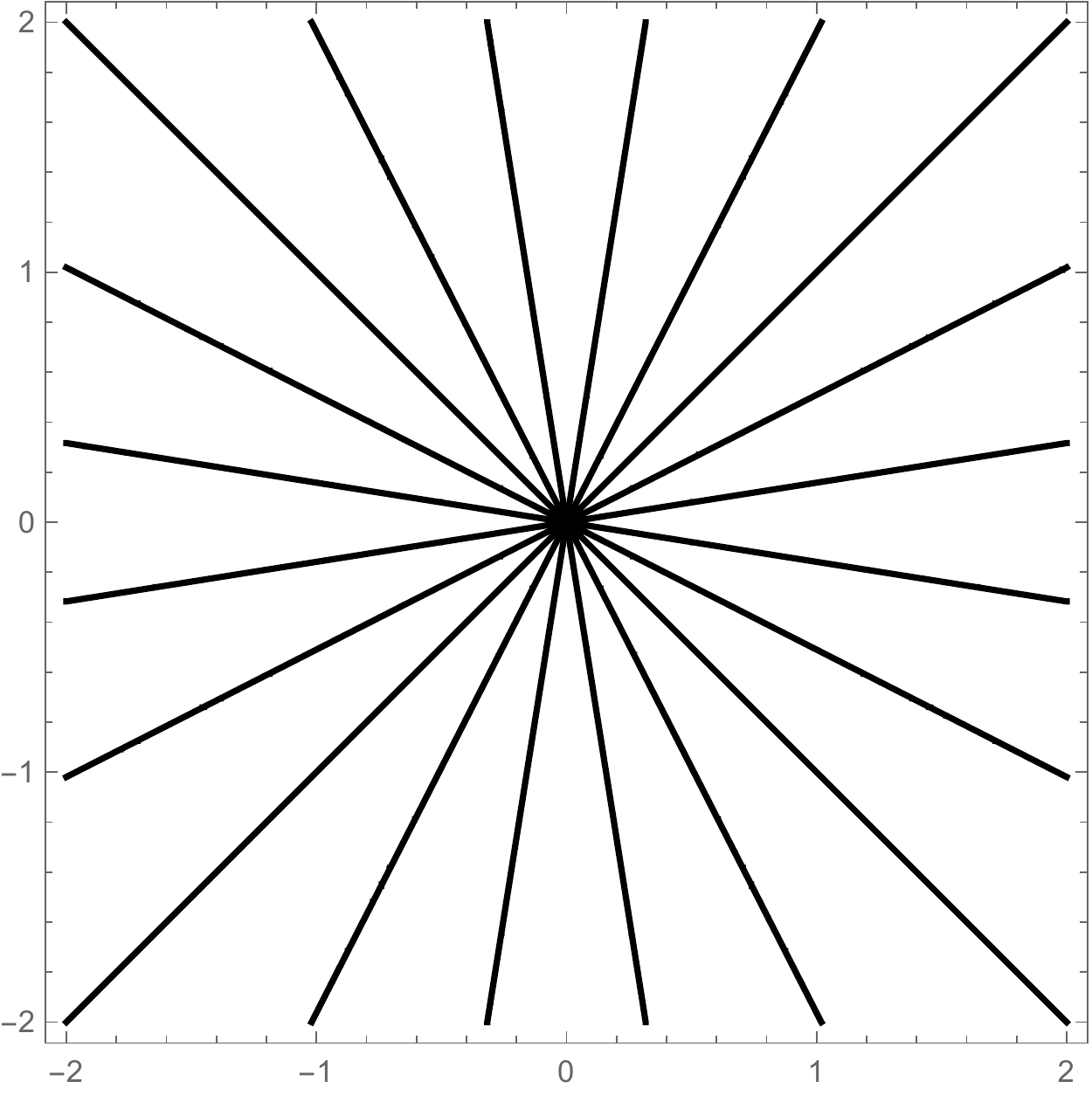}
	\centering
	\caption{Nodal set of $u(x,y)= \text{Re}(x+iy)^{10}$.}	\label{figure2}
\end{figure}

\section{Doubling index and the frequency function }
\label{sec3}
The frequency of a harmonic function is a well-studied close cousin of the doubling index. Due to its analytic nature, it is sometimes easier to study the frequency than the doubling index. In this section, we first define the frequency of a harmonic function and then present some of its well-known properties. Later, we discuss the relation between the frequency and the doubling index. All the results presented in this section are well-known; for completeness, we give their proofs in Appendix \ref{sec: appendix}.

\subsection{Frequency of a harmonic function}\label{sec_3.1}

Given  a harmonic function $u$ in $\mathbb{R}^n$, $x\in \mathbb{R}^n$  and $r>0$,  we define 
\begin{align}
	\label{def: harmonic average}
	H(x,r):= \int_{\partial B(x,r)} u^2 dS,
\end{align}
where $dS$ is the surface measure on $\partial B(x,r)$. Moreover, we also let 
 \begin{align}
 	\label{def: harmonic average derivative}
 	G(x,r):= \int_{B(x,r)} |\nabla u|^2 d\vol.
 \end{align}
The frequency function is defined by 
\begin{align}
	\label{def: frequency function}
	\beta(x,r):= \frac{r G(x,r)}{H(x,r)}.
\end{align} 
 Agmon \cite{AG66} and  Almgren \cite{AL79} proved that $\beta(\cdot)$ is a monotonic function. 

\begin{thm}[Monotonicity of the frequency]
	\label{thm: monotonicity frequency}
	For any harmonic function in $\mathbb{R}^n$, its frequency function $\beta(r):=\beta(0,r)$ is a non-decreasing function of $r$. 
\end{thm}

Garofalo and Lin  \cite{GL86} generalized it to elliptic operators with variable coefficients, see also the work of Kukavica \cite{K98} and Mangoubi \cite{M13}. We would like to mention an elegant work of Colding and Minicozzi \cite{CM22}, which proves monotonicity of the frequency function for parabolic operators in a suprisingly elementary way and in great generality.

\vspace{2mm}

An equivalent way to define the frequency function is 
$$\beta(x,r):= \frac{rH'(x,r)}{2H(x,r)} - \frac{n-1}{2},$$
where the derivative is taken with respect to $r$, see Appendix \ref{sec: appendix}. That is the frequency is the logarithmic derivative of $H$, up to the factor of $r$ and constants. The following is an immediate corollary of this definition and the monotonicity of $\beta$.
 \begin{cor}
 \label{cor: formulas for the growth}
	For any harmonic function in $\mathbb{R}^n$, let $\beta(r)= \beta(0,r)$ be its frequency function. For any $0<r_1<r_2$,  we have 
 $$\frac{H(r_2)}{r_2^{n-1}}= \frac{H(r_1)}{r_1^{n-1}}\exp \left( 2 \int_{r_1}^{r_2}\frac{\beta(r)}{r}dr\right).$$
 Moreover, we also have 
 $$\left(\frac{r_2}{r_1}\right)^{2\beta(r_1) + n-1}\leq \frac{H(r_2)}{H(r_1)}\leq \left(\frac{r_2}{r_1}\right)^{2\beta(r_2) + n-1}.$$
 \end{cor}

\begin{rem} 
	To define the frequency, it is not necessary for the harmonic function to be defined in the whole of $\real^n$. For any $R>1$ and any harmonic function defined in a neighbourhood of the closed ball $\overline{B(x,R)} \subset \real^n$, the frequency is well defined by \eqref{def: frequency function} for $r\in [0,R]$. All the results about the frequency (namely Theorem \ref{thm: monotonicity frequency}, Corollary \ref{cor: formulas for the growth}, and Lemmas \ref{lem: comparability di and frequency} $\&$ \ref{lem: monotonicity doubling index}) hold in any ball where the harmonic function is defined. 
\end{rem}

 \subsection{Relation between the frequency and the doubling index}\label{sec_3.2} Thanks to the standard elliptic estimates (see the proofs in Appendix \ref{sec: appendix}), we have the following comparison between the frequency in a ball and the doubling index in a slightly larger ball. 
 
\begin{lem}
	\label{lem: comparability di and frequency}  Let $\delta>0$ be a sufficiently small parameter and let $\mathcal{N}(r):= \mathcal{N}(B(0,r))$, for $r>0$. There exists a constant $C=C(\delta,n)\geq 1$ such that 
	$$\beta (r(1+\delta))(1-100\delta) -C\leq \mathcal{N}(r) \leq \beta (2r(1+\delta))(1+100\delta) +C$$
uniformly	for all $0<r\leq 4^{-1}$ 	and all harmonic functions $u$ defined in $B(0,2)$. 
\end{lem}
The above lemma together with the monotonicity of the frequency (Theorem \ref{thm: monotonicity frequency}) immediately imply the following {almost} monotonicity property of the doubling index.  

\begin{lem}[Almost monotonicity of the doubling index]
	\label{lem: monotonicity doubling index}
		Let $\delta>0$ be a sufficiently small parameter and let $u$ be a harmonic function in $B(0,2)$. {Let $\mathcal{N}(r):= \mathcal{N}(B(0,r))$, for $r>0$.} There exists a constant $C=C(\delta,n)>0$, independent of $u$, such that 
	$$\left(\frac{r_2}{r_1}\right)^{\mathcal{N}(r_1)(1-\delta)-C}\leq \frac{\sup \limits_{B(0,r_2)}|u|}{\sup \limits_{B(0,r_1)}|u|}\leq \left(\frac{r_2}{r_1}\right)^{\mathcal{N}(r_2)(1+\delta)+C},$$
	 for all $r_1\leq r_2/2\leq 4^{-1}$. In particular, we have 
	$$\mathcal{N}(r_1)(1-\delta) - C\leq \mathcal{N}(r_2)(1+\delta) + C.$$
\end{lem}

\begin{lem}[Almost monotonicity for non-concentric balls]
\label{lem: almost monotonicity} 
Let $B\subset \mathbb{R}^n$ be any ball and let $u$ be a harmonic function in $2B$. There exists some constant $C=C(n)>1$ such that 
$$\mathcal{N}_u(b) \leq C \mathcal{N}_u(B)$$
for all balls $b$ with  $2b\subset B$.  
\end{lem}

\section{Distribution of the doubling index and how to use it}\label{sec4}
In this section we review the known results from \cite{DF88}, \cite{Llemmas} about the distribution of the doubling index and formulate a new bound (Proposition \ref{prop: lemma 2}) , which is a crucial ingredient of the paper for proving a lower bound for the nodal volume. 
\subsection{Upper bounds for the distribution of doubling index}
\label{sec: distribution}

The following result is borrowed from \cite[Theorem 5.3]{Lproof}. A sketch of the proof is provided in Appendix \ref{sec: proof upper bound distribution}. The proposition below states that if we chop a big cube into smaller subcubes, then most of the subcubes have smaller doubling index than that of the big cube.

\begin{prop}
	\label{prop: lemma 1}
	Let a cube $Q$ in $\mathbb{R}^n$ be partitioned into $A^n$ equal subcubes $Q_i$. Let $u$ be a harmonic function in $4\sqrt{n} \cdot Q$. There exist numerical (depending only on the dimension $n$) constants $A_0,C > 1$ and $c>0$  such that  if $A>A_0$, then the number of $Q_i$ with 
	$$\N^*(Q_i)>\max(\N^*(Q)\exp(- c \log A/ \log \log A ), C)$$
	is smaller than $A^{n-1-c}$. 
\end{prop}


The work of Donnelly--Fefferman \cite{DF88} gives a different kind of information about the distribution of the doubling index.

\begin{prop}
	\label{lem: half of the doubling index} Let $Q$ be a unit cube in $\real^n$ and let $u$ be a harmonic function in $4\sqrt n \cdot Q$. Let $\varepsilon>0$ and 	let  $Q$  be partitioned into  equal subcubes $\{Q_i\}_i$ of side length  $\asymp_n\tfrac{c}{\N_u^*(Q)}$ for some sufficiently small $c=c(n,\varepsilon)>0$.  There exists a constant $C = C(n,\varepsilon) > 1$ such that 
	$$\N^*(Q_i)\leq C$$
	for at least $(1-\varepsilon)$ portion of the $Q_i$. 
\end{prop}
Although Proposition \ref{lem: half of the doubling index} is known to experts and is attributed to Donnelly and Fefferman, we could not find a precise reference. A version of it appeared in the study of the doubling index for Laplace eigenfunctions \cite{DF88}. Roughly speaking, if we consider a Laplace eigenfunction on a compact Riemannian manifold with eigenvalue $\lambda$, and cover the manifold with balls of radius $\asymp \lambda^{-1/2}$, the doubling index for most of these balls is bounded. 

\vspace{2mm}

We will not use Proposition \ref{lem: half of the doubling index}  by Donnelly and Fefferman in 
our present work. Nevertheless,  it follows in a straightforward way from their results in \cite{DF88} about holomorphic functions and the method of holomorphic extension. A formal deduction of Proposition \ref{lem: half of the doubling index} is contained in Appendix \ref{sec: Appendix DF}. The result is formulated here for the sake of completeness and in connection to Conjecture \ref{Conj1} (almost additivity). 
To explain the connection we have to formulate another conjecture.

\begin{conj}
	\label{conj2}
	Assume that a harmonic function $u$ has stable growth in a unit ball $B$ and $\mathcal{N}_u(B)=N \gg 1$.
	Let us cover $2B$ by a lattice of cubes $Q_i$ of size $1/N$. Then the number of $Q_i$ in $2B$, which contain a zero of $u$, is at least 
	$cN^n$ for some $c>0$ depending on the dimension $n$.
\end{conj}  

Let us make a few remarks and speculations. 

The constant $c$ in Conjecture \ref{conj2} might become worse as we change the constant in the definition of stable growth. 


\vspace{2mm}

Conjecture \ref{conj2} about many separated zeroes and Proposition \ref{lem: half of the doubling index} imply that Conjecture \ref{Conj1} is true under the assumption of stable growth. 
Roughly speaking they imply that one can find $\asymp N^n$ cubes $Q_i$ of size $1/N$  ($N=\mathcal{N}_u(B)$), which contain zeroes of u and have a bounded doubling index. Claim \ref{claim_sec2_2.4} implies that $2Q_i$ must have nodal volume $\gtrsim 1/N^{n-1}$. So 
\begin{align*}
	\mathcal{H}^{n-1}(\{u=0\}\cap 2B) \gtrsim N 
\end{align*}
in the case of stable growth (this additional assumption can be earned for free).


\subsection{ New result on the distribution of doubling index} 

Let us remind the definition of \textit{scaled doubling index}.

\begin{defn} \label{defn_nod_cap} For a harmonic function $u$ in a ball $2B \subset \real^n$, we define the \textit{scaled doubling index} (s.d.i.) of $B$, denoted by $\mathcal{SN}(B)$, to be 
	\begin{align*}
		\mathcal{SN}(B) := \N(B) r^{n-1}, 
	\end{align*} 
	where $r$ is the radius of $B$.
\end{defn}

For a harmonic function vanishing at the center of a ball $B$,  $\mathcal{SN}(\thalf B)$ is conjectured (Conjecture \ref{Conj1}) to give a lower bound for its nodal volume in $2B$. 



\vspace{2mm}

Complementary to the known results (\S \ref{sec: distribution}), our main new result also gives a lower bound for the distribution of the doubling index.  Given a ball $B$ and a harmonic function on $4B$, we show that there is a  collection of disjoint balls with doubling index smaller than $\N(2B)$, but whose total s.d.i. is almost as large as $\N(\thalf B)$.
\begin{prop}
	\label{prop: lemma 2}
	Let $B\subset \mathbb{R}^n$ be a unit ball, $u: 4B \ra \real$ a harmonic function, and $A>1$  a sufficiently large parameter. Let $N_1:= \N(\frac{1}{2}B)$ and $N_2:= \mathcal{N}(2B)$. Then there 
	exists a constant $C_A>1$ such that whenever $N_1\geq C_A$,  there exists a collection of disjoint balls $\{B_i = B(x_i,r_i)\}_{i \in\mathcal{I}}$ such that 
	\begin{enumerate}
		\item \label{pnt1}We have an upper bound on the doubling index:
		$$\max \N(2B_i) \lesssim_n N_2/A.$$ 
		\item \label{pnt2} We have a lower bound on the sum of the s.d.i.'s:
		$$ \sum 	\mathcal{SN}(\tfrac{1}{2}B_i)=\sum \N\left(\tfrac{1}{2}B_i\right) r_{i}^{n-1} \gtrsim_n \frac{N_1}{(\log A\log\log A)^{n-1}}. $$
		
		\item \label{pnt3}Vanishing and non-intersection properties:
		$$u(x_i) = 0,\text{ } \text{ } 4B_i \cap 4 B_j=\emptyset,\text{ } \text{ and } \text{ }4B_i\subset 2B,$$
		for all $i\neq j$. 
	\end{enumerate}
\end{prop}

We now briefly sketch how one can use Proposition \ref{prop: lemma 2} to implement the naive idea discussed in \S\ref{sec2.2}  in order to prove a lower bound for the zero set. 
One can recursively apply Proposition \ref{prop: lemma 2} in order to obtain a collection of balls with bounded doubling indices. Roughly speaking, in each step we apply Proposition \ref{prop: lemma 2} to each of the balls with large doubling index and replace it by a collection of sub-balls with smaller doubling index. Conditions \eqref{pnt1} and \eqref{pnt3} guarantee that in a finite number of steps all balls will have bounded doubling index  and  $u$ must vanish at the center of each ball. The number of steps can be estimated in terms of $N_2$ and $A$.

\vspace{2mm}

Condition \eqref{pnt2} should be interpreted as a condition ensuring that after recursive applications of Proposition \ref{prop: lemma 2}, we don't lose too much of the nodal set. Here is an informal explanation. The sum  $\sum \mathcal{SN}(\thalf B_i)$ imitates the nodal volume. If we believe in Conjecture \ref{Conj1}, then 
\begin{align*}
\sum \mathcal{H}^{n-1}(\{u=0\} \cap 2B_i) \gtrsim \sum \mathcal{SN}(\thalf B_i) \gtrsim \frac{N_1}{(\log A\log\log A)^{n-1}},
\end{align*}
and the loss of the nodal volume in each step is subpolynomial in $A$. 

In the last step all balls have bounded doubling index. Recall that if $u$ has a bounded doubling index in a ball $B$ and vanishes at its  center, then Claim \ref{claim_sec2_2.4} gives the following estimate for the nodal volume 
\begin{align*}
\mathcal{SN}(\thalf B) \asymp \mathcal{H}^{n-1}(\{u=0\} \cap B) \asymp (\text{radius}(B))^{n-1}.
\end{align*}
We can apply Claim \ref{claim_sec2_2.4} to each of the balls in the last step and the control \eqref{pnt2} of the sum of s.d.i. provides a lower bound for the nodal volume. In the next section we formalize the idea and quantify what kind of a lower bound we get for the  nodal volume.


	\subsection{How to use Proposition \ref{prop: lemma 2}} We will start proving 
	Proposition \ref{prop: lemma 2} only in \S \ref{sec: reduction}. Now, we show how Proposition \ref{prop: lemma 2} implies the following lower bound for the nodal volume. 
	
	\begin{lem}
		\label{lem: induction on scales}
		Let $B\subset \mathbb{R}^n$ be a unit ball, and $u: 4B \ra \real$ a harmonic function. Assume that $\varepsilon>0$,  $N_1  := \mathcal{N}(\frac{1}{2}B) \gg_{n,\varepsilon} 1$, and $N_2 := \mathcal{N}(2B)$. Then we have  
		$$\mathcal{H}^{n-1}\left( \{u=0\}\cap 2B\right)\gtrsim_{n,\varepsilon}  N_1^{1-\varepsilon} \left(\frac{N_1}{N_2}\right)^{\varepsilon}.$$ 
	\end{lem}
	We observe that if $u$ has stable growth in a unit ball $B$, i.e. the ratio $N_2/N_1$ is bounded by some numerical constant (possibly depending on the dimension $n$), then Lemma \ref{lem: induction on scales} immediately implies Theorem \ref{thm: main} {under the assumption of stable growth}. Unfortunately, the bound in Lemma \ref{lem: induction on scales} gets worse as the ratio $N_2/N_1$ gets larger. However, in \S\ref{sec: reduction to stable growth} we will show that Lemma \ref{lem: induction on scales} implies Theorem \ref{thm: main} by reducing the general case to the stable growth case.
	
	\begin{proof}[Proof of Lemma \ref{lem: induction on scales}] Given $\varepsilon>0$, we will choose a large constant $A=A(\varepsilon)\geq 2$ later.  Let $C_A\geq 1$ be given by Proposition \ref{prop: lemma 2}, and assume that $C_A\gg_n 1$. Starting with $S_0 := \{B\}$, for every $k \in \nat$, we will recursively use Proposition \ref{prop: lemma 2} to construct a collection of balls $S_k = \{B_{\alpha}\}$. This sequence will eventually stabilize, that is, there exists $K\geq 0$ such that $S_k = S_K$ for every $k \geq K$. Moreover, the final collection $S_K$ will be such that  for every ball $B_{\alpha} \in S_K$,  $u$ vanishes at its center and $\N(\tfrac{1}{2}B_{\alpha}) \leq C_A$.

		Let $S_0 := \{B\}$. Given $S_k = \{B_{\alpha}\}$, we will define $S_{k+1}$. Suppose $B_{\alpha} \in S_k$ is such that $\N(\frac{1}{2} B_{\alpha}) \leq C_A$, then we retain $B_{\alpha}$ in $S_{k+1}$. If $\N(\frac{1}{2} B_{\alpha}) > C_A$, then we apply Proposition \ref{prop: lemma 2} to $B_{\alpha}$ to obtain a collection of balls $\{B_{\alpha,i}\}_{i}$ satisfying the three conditions of  Proposition \ref{prop: lemma 2}, and this collection will replace $B_{\alpha}$ in $S_{k+1}$, i.e.,
		\begin{align*}
			S_{k+1} := \{B_{\alpha}: B_{\alpha} \in S_k\text{ and }\N\left(\tfrac{1}{2} B_{\alpha}\right) \leq C_A\} \cup \{B_{\alpha,i}: B_{\alpha} \in S_k\text{ and }\N\left(\tfrac{1}{2} B_{\alpha}\right) > C_A\}.
		\end{align*}
		
		For $B_{\alpha } \in S_k$, let $x_{\alpha}$ and $r_{\alpha}$ be its center and radius respectively.  We claim that the collection $S_k$ satisfies the following properties:
		\begin{enumerate}[label=(\roman*)]
			\item\label{concon1} For some (large) $C=C(n)\geq 1$, we have for every $B_{\alpha} \in S_k$: 
			\begin{align*} 
				\N(2B_{\alpha}) \leq  C^k N_2/A^k \text{ or }\N
				\left(\tfrac{1}{2} B_{\alpha}\right) \leq C_A.
			\end{align*}
			\item \label{concon2} For some (small) $c=c(n)>0$, we have 
			$$	\sum_{B_{\alpha} \in S_k} \N(\tfrac{1}{2}B_\alpha) r_{\alpha}^{n-1} \geq c^k \frac{N_1}{(\log A\log\log A)^{k(n-1)}}. $$
			
			\item \label{concon3}Vanishing  and non-intersection properties: 
			$$u(x_{\alpha}) = 0 \text{, }\text{ }4B_{\alpha} \cap 4B_{\gamma}=\emptyset \text{, }\text{ and }\text{ }4B_{\alpha}\subset 2B,$$
			for all $B_\alpha,B_\gamma\in S_{k} $ with $B_\alpha \neq B_\gamma.$
		\end{enumerate}
		We prove by induction that the above properties \ref{concon1}-\ref{concon3} hold for every $S_k$. Properties \ref{concon1}-\ref{concon3} hold for $S_1$ by the statement of Proposition \ref{prop: lemma 2}. Assuming that all the above properties hold for $S_k$, let us show that they hold for $S_{k+1}$ as well. It is an immediate consequence of Proposition \ref{prop: lemma 2} and the way $S_k$ is constructed that \ref{concon1} and \ref{concon3} hold. We now show that \ref{concon2} also holds. We partition $S_k$ into $S_k^{(1)}$ and $S_{k}^{(2)}$ where
		\begin{align*}
			S_{k}^{(1)} = \{B_\alpha \in S_k: \N(\tfrac{1}{2} B_{\alpha}) \leq C_A\}\text{ and }S_{k}^{(2)} = S_k \setminus S_{k}^{(1)}. 
		\end{align*}
		For $B_\alpha \in S_k^{(2)}$, we denote by $\{B_{\alpha,i}\}_i$ the collection of balls obtained by applying Proposition \ref{prop: lemma 2} to $B_{\alpha}$. Let $r_{\alpha,i}$ be the radius of $B_{\alpha,i}$, then we have
		\begin{align*}
			\sum_{B_j \in S_{k+1}} \N(\tfrac{1}{2} B_j) r_{j}^{n-1} = \sum_{B_\alpha \in  S_k^{(1)}} \N(\tfrac{1}{2} B_{\alpha}) r_{\alpha}^{n-1} + \sum_{B_\alpha \in  S_k^{(2)}} \sum_{i} \N(\tfrac{1}{2} B_{\alpha,i}) r_{\alpha,i}^{n-1}. 
		\end{align*}
		It follows from the second condition of Proposition \ref{prop: lemma 2} that 
		\begin{align*}
			\sum_{i} \N(\tfrac{1}{2} B_{\alpha,i}) r_{\alpha,i}^{n-1} \geq \frac{c}{(\log A \log \log A)^{n-1}} \cdot \N(\tfrac{1}{2} B_{\alpha}) r_{\alpha}^{n-1}. 
		\end{align*}
		We thus have
		\begin{align*}
			\sum_{B_j \in S_{k+1}} \N(\tfrac{1}{2} B_j) r_{j}^{n-1} &\geq  \sum_{B_\alpha \in  S_k^{(1)}} \N(\tfrac{1}{2} B_{\alpha}) r_{\alpha}^{n-1} + \frac{c}{(\log A \log \log A)^{n-1}} \sum_{B_\alpha \in  S_k^{(2)}} \N(\tfrac{1}{2} B_{\alpha}) r_{\alpha}^{n-1},\\
			& \geq \frac{c}{(\log A \log \log A)^{n-1}} \sum_{B_\alpha \in S_k} \N(\tfrac{1}{2} B_{\alpha}) r_{\alpha}^{n-1},\\
			& \geq \frac{c^{k+1}}{(\log A \log \log A)^{(k+1)(n-1)}} \cdot N_1, 
		\end{align*}
		where the last inequality follows by our induction hypothesis that \ref{concon2} holds for $S_k$.
		
		We now show that $S_k$ stabilizes  in $\asymp \log N_2/\log A$ steps. Because of \ref{concon1}, at every step and for every ball $B_{\alpha} \in S_k$ either $\N(\tfrac{1}{2} B_{\alpha}) \leq C_A$ (and then $B_{\alpha}$ is frozen forever) or the doubling index of $2B_{\alpha}$ drops by a factor of $A/C$, which can be crudely estimated by $\sqrt{A}$ (assuming that $A$ is sufficiently large). The latter scenario can happen only $\asymp \log N_2/\log A$ times as we start with $N(2B)=N_2$. Indeed, otherwise doing $k \geq 2 \log N_2/\log A$ steps, we have a ball $B_{\alpha} \in S_k$ with $$\N(2B_{\alpha}) \leq N_2/ A^{k/2} \leq 1 $$ 
		and by almost monotonicity of the doubling index it implies
		$$ \N(\thalf B_{\alpha}) \leq C \N(2B_{\alpha}) \leq C \leq C_A.$$
		So after $K := \lceil 2 \log N_2/\log A \rceil$ steps all balls have bounded doubling index. 
		By Lemma \ref{claim: bounded doubling index} we have
		\begin{align}
			\mathcal{H}^{n-1}(\{u=0\}\cap 2B) &\geq \sum_{B_\alpha \in S_K} \mathcal{H}^{n-1}(\{u=0\}\cap B_{\alpha})  \nonumber \\
			& \geq c \sum_{B_\alpha \in S_K} \N\left(\tfrac{1}{2} B_{\alpha}\right)^{1-n} r_{\alpha}^{n-1} \nonumber\\
			& \geq c \cdot C_A^{1-n} \sum_{B_{\alpha} \in S_k}  r_{\alpha}^{n-1}. \label{eq:bound.on.nod.vol}
		\end{align}
		It is time to use \ref{concon2} to get a lower bound \eqref{eq:bound.on.nod.vol}:
		\begin{align}
			C_A \sum_{B_{\alpha} \in S_k}  r_{\alpha}^{n-1}  \geq \sum_{B_{\alpha} \in S_k} \N(\tfrac{1}{2} B_{\alpha}) r_{\alpha}^{n-1} \geq \frac{c^K}{(\log A \log \log A)^{K(n-1)}} N_1, \nonumber
		\end{align}
		thus,
		\begin{align}
			\mathcal{H}^{n-1}(\{u=0\}\cap 2B) \gtrsim C_A^{-n} \frac{c^{K}}{(\log A \log \log A)^{K(n-1)}} N_1. \nonumber
		\end{align}
		We may assume that $N_2\gg 1$ because $N_1 \gg 1$.	Since $K = \lceil 2 \log N_2/\log A \rceil$, we can always choose $A=A(\varepsilon)\geq 2$  so that
		\begin{align*} 
			c^K\geq  N_2^{-\varepsilon/2}\hspace{3mm}\text{ and } \hspace{3mm} (\log A\log\log A)^{K(n-1)} \leq  N_2^{\varepsilon/2}.
		\end{align*}
		Hence, absorbing the factor $C_A^{-n}$, which now depends only on $\varepsilon$, in the $\lesssim_{\varepsilon}$ notation, we have 
		$$ C_A^{-n}\cdot  c^K \frac{N_1}{(\log A\log\log A)^{K(n-1)}}\gtrsim_{\varepsilon}  N_1 \cdot N_2^{-\varepsilon }= N_1^{1-\varepsilon} \left(\frac{N_1}{N_2}\right)^{\varepsilon}, $$
		concluding the proof of Lemma \ref{lem: induction on scales}. 
	\end{proof}

\section{Structure of the proofs of Theorem \ref{thm: main} and Proposition \ref{prop: lemma 2}}

\tikzset{block/.style={rectangle split, draw, rectangle split parts=2,
		text width=14em, text centered, rounded corners, minimum height=4em},
	grnblock/.style={rectangle, draw, fill=white, text width=14em, text centered, rounded corners, minimum height=4em}, 
	sgrnblock/.style={rectangle, draw, fill=white, text width=7em, text centered, rounded corners, minimum height=2em}, 
newsgrnblock/.style={rectangle, draw, fill=white, text width=9em, text centered, rounded corners, minimum height=2em},
slgrnblock/.style={rectangle, draw, fill=white, text width=13em, text centered, rounded corners, minimum height=2em}, 
whtblock/.style={rectangle, draw, fill=white!20, text width=14em, text centered, minimum height=4em},    
line/.style={draw, -{Latex[length=2mm,width=1mm]}},
cloud/.style={draw, ellipse,fill=white!20, node distance=3cm,    minimum height=4em},  
container/.style={draw, rectangle,dashed,inner sep=0.28cm, rounded
	corners,fill=white,minimum height=4cm}}

\vspace{2cm}
\begin{figure}
\centering
\begin{tikzpicture}[node distance = .8cm, auto]
	
	\node [grnblock,font=\fontsize{10}{0}\selectfont] (lem81) {\\[0.4em] In a neighbourhood of a \textit{good tunnel}, \\[0.3em] we show existence  of many balls  with \spc zeros of $u$ and large total sum of d.i.};

	\node at (lem81.north) [fill=white,draw,font=\fontsize{11}{0}\selectfont] (name) {\textbf{Lemma \ref{lem: general lemma}}};

	\node [slgrnblock,node distance=7cm,font=\fontsize{10}{0}\selectfont] [below left=0.5cm and 0.5cm of lem81] (prop41) { \\[0.7em] Provides an estimate for the \spc distribution of the d.i. of a \spc harmonic function in a cube. \\[0.8em]  Follows from monotonicity of \spc  the d.i. and Cauchy-uniqueness \spc property of harmonic functions.};   
	
	\node at (prop41.north)  [fill=white,draw,font=\fontsize{11}{0}\selectfont] {\textbf{Proposition \ref{prop: lemma 1}}};
	
	
	\node [newsgrnblock = Lemma,node distance=7cm,font=\fontsize{10}{0}\selectfont] [below right= 1.2cm and -0.8cm of lem81] (lem64)  {\\[.5em] Gives growth estimates \spc of a harmonic function \spc in a spherical layer.};    
	
	\node at (lem64.north)  [fill=white,draw,font=\fontsize{10}{0}\selectfont] {\textbf{Lemma \ref{lemma: behaviour_around_max 2.0}}};
	

	\node [slgrnblock, node distance=4cm,font=\fontsize{10}{0}\selectfont] [below left=1.1cm and -0.5cm of lem64] (lem71) {\\[0.6em] Shows existence of many good \spc tunnels, and many balls  with \spc zeros of $u$ and large total s.d.i.};

	\node at (lem71.north) [fill=white,draw,font=\fontsize{11}{0}\selectfont] (name) {\textbf{Lemma \ref{lem: extra assumption lemma 2}}};
	\node [slgrnblock, node distance=2.5cm,font=\fontsize{10}{0}\selectfont] [below=2cm of lem71] (prop43) { \\[0.6em]Analogous version of  Lemma  \ref{lem: extra assumption lemma 2}, \spc without the stable  growth \spc assumption.};
	
	\node at (prop43.north) [fill=white,draw,font=\fontsize{11}{0}\selectfont] (name2) {\textbf{Proposition \ref{prop: lemma 2}}};
	
	\node [slgrnblock, node distance=2.5cm,font=\fontsize{10}{0}\selectfont] [below=1.5cm of prop43] (lem44) { \\[0.6em] Involves repeated \spc applications of Proposition \ref{prop: lemma 2} \spc at multiple scales to get a lower \spc bound for the nodal volume.};
	
	\node at (lem44.north) [fill=white,draw,font=\fontsize{11}{0}\selectfont] (name3) {\textbf{Lemma \ref{lem: induction on scales}}};

	\node [newsgrnblock, node distance=1cm,font=\fontsize{10}{0}\selectfont]  [below left=1cm and 0cm of lem44] (stable) { \\[0.6em] under stable \spc growth assumption. };
	
	\node at (stable.north) [fill=white,draw,font=\fontsize{11}{0}\selectfont] (namestable) {\textbf{Theorem \ref{thm: main}}};
	
	\node [sgrnblock, node distance=1cm,font=\fontsize{10}{0}\selectfont]  [below right=1cm and 0cm  of lem44] (general) {\\[0.6em] General case.  };
	
	\node at (general.north) [fill=white,draw,font=\fontsize{11}{0}\selectfont] (namegeneral) {\textbf{Theorem  \ref{thm: main}}};

	\path [line] (lem81) -- (name);
	\path [line] (prop43) -- (name3);
	\path [line] (lem44) -- (namestable) node[midway,above left] {{Easy consequence}};
	\path [line] (lem44) -- (namegeneral) node[midway] {$\boxed{\text{\textbf{Claim \ref{claim: existence of stable region}}}}$};
	
	\path [line] (lem71) -- (name2) node[midway] {$\boxed{\text{\textbf{Claim \ref{claim: existence of stable region}}}}$};
	\begin{scope}[on background layer]
		\coordinate (aux1) at ([xshift=0mm, yshift=-10mm]lem81.south);
		\node [container,fit=(aux1) (lem64)(lem71)] (MICRO) {};
		\node [below right=1.6cm and -3cm of lem64] [fill=white,draw,font=\fontsize{12}{0}\selectfont] {\textbf{Stable growth}};
		
	\end{scope}
	
	\path [line] (prop41) -- ($(lem81)!0.5!(lem71)$);
	\path [line] (lem64) -- ($(lem81)!0.5!(lem71)$);
\end{tikzpicture}\\[.5cm]
\caption{A schematic outline of the main results used to prove Theorem \ref{thm: main}.}
\label{fig:schematic_outline}
\end{figure}
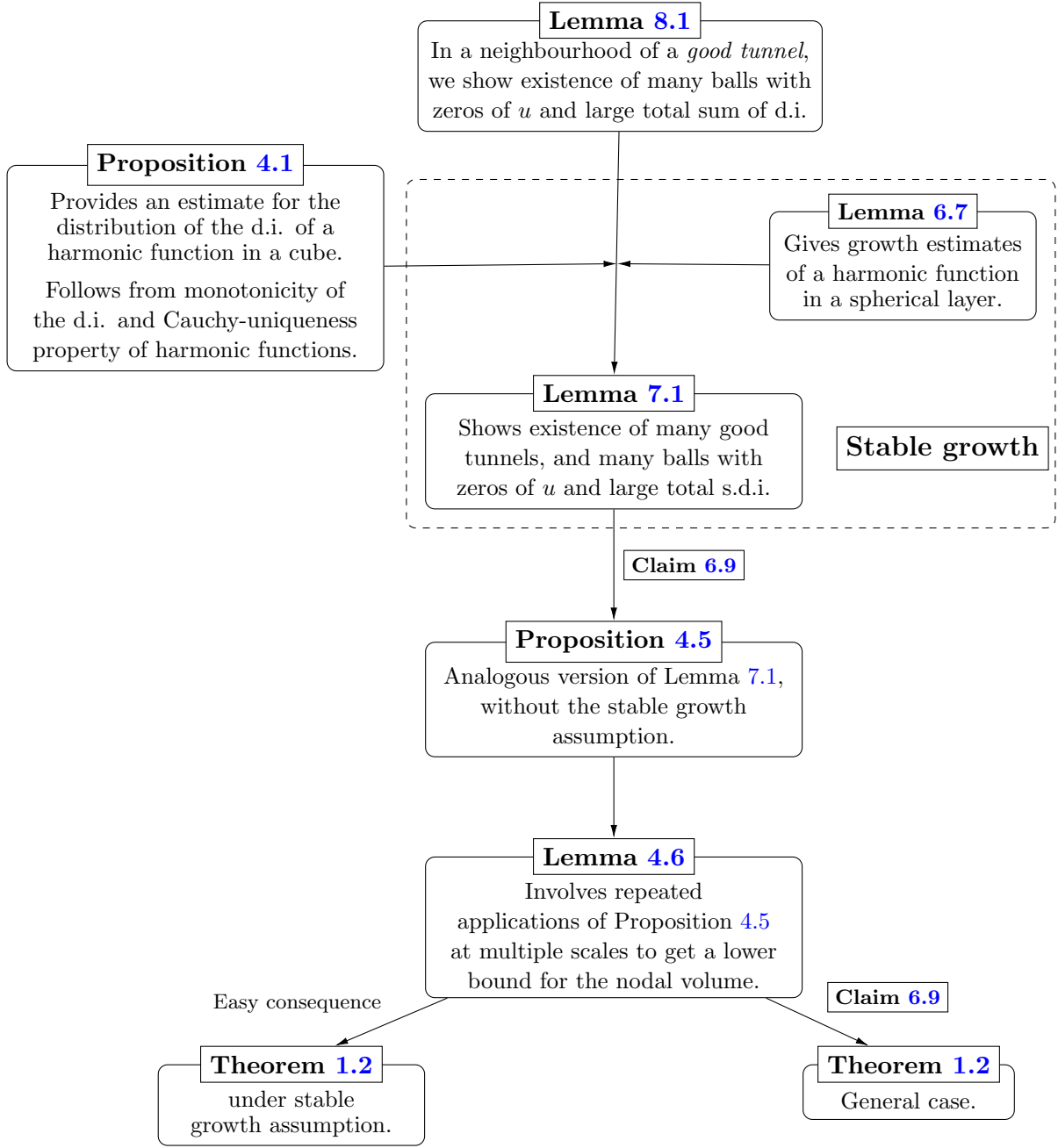

In Figure \ref{fig:schematic_outline}, we present a logical progression of ideas/results leading to our main result Theorem \ref{thm: main}. We  now indicate where to find these results, and also what we do in the remaining sections. 
We advise the reader to keep Figure \ref{fig:schematic_outline} and the following as a guide while navigating through the rest of the paper.

\vspace{2mm}

\textbf{\S \ref{sec4} Distribution of the doubling index and how to use it:} In this section, we present old and new results on the distribution of the doubling index of a harmonic function. Given a cube $Q \subset \real^n$ which is partitioned into equal subcubes $Q_i$, Proposition \ref{prop: lemma 1}, which was established in \cite{Lproof}, gives an upper bound on the number of cubes with large doubling index.

\vspace{2mm} 

Proposition \ref{prop: lemma 2} is a new result about the distribution of doubling index. For a harmonic function $u$ defined in a ball $B$, Proposition \ref{prop: lemma 2} gives a collection of disjoint balls in $B$ satisfying the following properties: $u$ vanishes at their centers, their total s.d.i. is large, and their doubling index is much smaller than $\N(B)$. We explain how to use Proposition \ref{prop: lemma 2} in Lemma \ref{lem: induction on scales}, which gives a lower bound for the nodal volume.


\vspace{2mm}

\textbf{\S \ref{sec: reduction to stable growth} Stable growth: }
In this section, we study the  behaviour of a harmonic function in a ball with stable growth (Definition \ref{def: stable growth}). In \S \ref{sec6.1}, we prove Lemma \ref{lemma: behaviour_around_max 2.0},
which gives estimates for the growth of a harmonic function in a spherical layer where the frequency is almost constant. 

\vspace{2mm}

In some instances, we will prove  results about the nodal set of a harmonic function in a ball $B$, under the simplified assumption of {stable growth}. Claim \ref{claim: existence of stable region}, proved in \S\ref{sec6.2}, serves as a bridge between the stable growth case and the general case by implying that for any harmonic function defined in ball $B$, there are several balls $B_k \subset B$ of \textit{significant} size with stable growth.  Lemma \ref{lem: induction on scales} implies Theorem \ref{thm: main} under the additional assumption of stable growth. We show that this along with Claim \ref{claim: existence of stable region} implies the general case of Theorem \ref{thm: main}.

\vspace{2mm}

\textbf{\S \ref{sec: reduction} Reducing Proposition \ref{prop: lemma 2} to the stable growth case:} We state Lemma \ref{lem: extra assumption lemma 2}, which is a  version of Proposition \ref{prop: lemma 2} with the additional  assumption of stable growth. Using Claim  \ref{claim: existence of stable region} again, we reduce Proposition \ref{prop: lemma 2} to Lemma \ref{lem: extra assumption lemma 2}.

\vspace{2mm}

\textbf{\S \ref{sec: general lemma} Zeros of harmonic functions in a tunnel:} In this section, we study harmonic functions defined in a \textit{tunnel}, that is a hyper-rectangle in $\real^n$ with all sides but one of equal length, while the length of the remaining side is much larger than the others. In Lemma \ref{lem: general lemma}, we show that for a harmonic function $u$ defined in a \textit{good tunnel} (which is a tunnel where $u$ has significant growth {along the longer side of the tunnel}, and there is some control on the doubling index at all points in the tunnel), there are many disjoint balls where $u$ vanishes and the sum of their s.d.i.'s is large. 

\vspace{2mm}

\textbf{\S \ref{sec: conencting the dots} Connecting the dots: Proof of Lemma \ref{lem: extra assumption lemma 2}:} In this final section, equipped with Proposition \ref{prop: lemma 1}, Lemma \ref{lemma: behaviour_around_max 2.0} and Lemma \ref{lem: general lemma}, we finish the proof of Lemma \ref{lem: extra assumption lemma 2}. 
Recall that Lemma \ref{lem: extra assumption lemma 2} is Proposition \ref{prop: lemma 2} with the additional assumption of stable growth. 

\vspace{2mm}

We first use Lemma \ref{lemma: behaviour_around_max 2.0} to get several tunnels where the harmonic function has significant growth {along their longer sides}. We then use Proposition \ref{prop: lemma 1} to control the doubling index at all points in at least half of these tunnels, thus giving us several good tunnels. 
Finally in these good tunnels, we apply Lemma \ref{lem: general lemma} to conclude that there are several disjoint balls satisfying the conclusions of Lemma \ref{lem: extra assumption lemma 2}.

\section{Stable growth}\label{sec: reduction to stable growth}
The notion of stable growth has already appeared a few times in the earlier sections; we are finally ready to formally introduce it.

\subsection{Definition of stable growth and some consequences}\label{sec6.1}

\begin{defn}[Stable growth]
	\label{def: stable growth}
	
	Let $C>1$ be a  constant depending only on the dimension $n$ (which will be chosen in Remark \ref{rem: choice of constant}).
A harmonic function $u:\mathbb{R}^n\ra \real$ is said to have \textit{stable growth}  in a ball $B\subseteq \real^n$ if 
		 \begin{align}
	\label{stablegrowthdefn2}
	C^{-1}\N(\tfrac{1}{2}B) \leq \N(2B) \leq   C \N(\tfrac{1}{2}B).
	\end{align}
\end{defn}
In \eqref{stablegrowthdefn2}, the first inequality $C^{-1} \N(\thalf B) \leq \N(2B)$ automatically holds as a consequence of the almost monotonicity of the doubling index (Lemma \ref{lem: almost monotonicity}). The main point of the stable growth condition \eqref{stablegrowthdefn2} is the other inequality, namely, $\N(2B) \leq C\N(\thalf B)$.
\begin{exmp}  Any homogeneous harmonic polynomial $p$ of degree $N \geq 1$ has stable growth in every ball $B(0,r) \subseteq \real^n$. Such $p$ can be written as
		\begin{align*}
		p(r,\theta) = r^{N} q(\theta),\text{ for }r\geq 0, \text{ }\theta \in \mathbb{S}^{n-1},
		\end{align*}
	where $q: \mathbb{S}^{n-1} \ra \real$ is a spherical harmonic of degree $N$. Then 
	\begin{align*}
\sup_{B(0,r)}|p| = r^{N} \sup_{\mathbb{S}^{n-1}}|q|,
\end{align*}
and the doubling index $\N(0,r) \equiv N$, for every $r>0$. 
\end{exmp}
\begin{exmp}
	Any harmonic polynomial $P$ has the following representation:
\begin{align*}
P(x) = \sum_{k=m}^{M} p_k(x),
\end{align*} 
where $p_k$ is a homogeneous harmonic polynomial in $\real^n$, of degree $k$. For $0<r\ll 1$ and $R \gg 1$, we have 
\begin{align*}
P(r,\theta) & = p_{m}(r,\theta) + O(r^{m +1}),\\
P(R,\theta) & = p_{M}(R,\theta) + O(R^{M-1}). 
\end{align*}
Hence for small $r$, we have $\N(0,r) \asymp m$; for  large  $R$, we have $\N(0,R) \asymp M$. Hence $P$ has stable growth in balls centered at the origin, whose radius is either small or large.   
	
In the case $M=2m$, for the polynomial $P(x) = \sum_{k=m}^{2m} p_k(x) $ we have $\N(0,r) \asymp m$ for small $r$, and $\N(0,R) \asymp M \asymp m$ for large $R$. So by the monotonicity of the doubling index, $\N(0,t) \asymp m$ for all $t>0$. Hence $P$ has stable growth in every ball $B(0,r)$. 

\end{exmp}

\begin{rem}
	\label{rem: important remark}
We observe that if a harmonic function $u$ has stable growth in a ball $B=B(x,r) \subset \real^n$ with $\mathcal{N}(\frac{1}{2}B)=: N\geq 1$, then because of the close relation between the doubling index and the frequency (Lemma \ref{lem: comparability di and frequency}), the frequency $\beta$ satisfies 
$$
\beta_u(x,t)\in \left[c_1 N, C_1 N\right], \hspace{5mm}\text{ }\forall t\in [1.1r,1.9r]
$$
where constants $c_1,C_1>0$ depend only on $n$. 
\end{rem}

\subsection*{Choosing the stability constant}
\begin{lem} \label{lem:stability_constant} Let $B=B(0,1) \subset \real^n$ be the unit ball and let $u$ be a harmonic function in $4B$.
	Consider the spherical layer
	$$ \mathcal{S}_{\rho,w}:=\{x\in \mathbb{R}^n : \rho-10w\leq |x|\leq \rho + 10 w\}\subset B(2)\backslash B(1),$$ with width $20w \in (0,\rho)$.	 Suppose that the frequency  of $u$  in this layer satisfies 
	\begin{align*}
	\beta_u(0,t) \in [N,10N] \text{, for $t\in [\rho- 10w, \rho + 10w]$}.
	\end{align*}
	Also assume that the layer is not too narrow in terms of $N$, in the sense that
	$$N w(\log \tfrac{1}{w})^{-1} \underset{n}{\gg} 1.$$
	Then there exists a constant $\mathrm{C} >0$ and a ball $D$ of radius $w$ such that $4D \subset \mathcal{S}_{\rho,w}$ satisfying 
	\begin{align*}
	\mathrm{C}^{-1}\N(\thalf D) \leq \N(2D) \leq \mathrm{C} \N(\thalf D), 
	\end{align*}
	and 
	\begin{align*}
	\N(2D) \asymp_n wN. 
	\end{align*}
\end{lem}
\begin{rem}
	\label{rem: choice of constant}
	From now on, we fix the stability constant in  Definition \ref{def: stable growth} to be $\mathrm{C}$, as given in the statement of Lemma \ref{lem:stability_constant}. 
\end{rem}

 The following result captures the growth of a harmonic function with stable growth.

\begin{lem}
	\label{lemma: behaviour_around_max 2.0}Let $B=B(0,1) \subset \real^n$ be the unit ball and let $u$ be a harmonic function in $4B$.
 Consider the spherical layer
	$$ \mathcal{S}_{\rho,w}:=\{x\in \mathbb{R}^n : \rho-10w\leq |x|\leq \rho + 10 w\}\subset B(2)\backslash B(1),$$ with width $20w \in (0,\rho)$.	 Suppose that the frequency  of $u$ is comparable to $N>1$ in this layer, that is, $\beta_u(0,t) \asymp_n N$  for $t\in [\rho- 10w, \rho + 10w]$. Also assume that the layer is not too narrow in terms of $N$, in the sense that
	 $$N w(\log \tfrac{1}{w})^{-1} \underset{n}{\gg} 1$$  (the larger $N$, the smaller $w$ can be).   Define the function 
	$$M(t):= \sup_{tB}|u|.$$
 	Then, for all  $t_1<t_2\in [\rho- 3w,\rho+ 3w]$ and $|t_2-t_1|\geq w/2$, we have  
\begin{align}
 	\label{eqn:stab.growth.max.comp.}
 	\exp(\lambda(t_2-t_1)N)<\frac{M(t_2)}{M(t_1)}<\exp (\Lambda (t_2-t_1)N),
 \end{align}
 for some constants $\lambda=\lambda(n)>0$ and $\Lambda=\Lambda(n)>1$.
\end{lem}


Lemma \ref{lemma: behaviour_around_max 2.0} will be used to prove Lemma \ref{lem:stability_constant} and also Lemma \ref{lem: scaling around the maximum}.
\begin{lem}
	\label{lem: scaling around the maximum}

	Under the assumptions of Lemma \ref{lemma: behaviour_around_max 2.0} and maintaining the same notation, we have the following. Let $x_0 \in \partial(\rho B)$ be a point where $|u|$ attains its maximum on $\rho B$, that is, $|u(x_0)| = M(\rho)$. If ${D}$ is any ball of radius $w$ such that $x_0 \in \partial {D}$, then 
	\begin{align*}
	\N({D}) \lesssim_n wN.
	\end{align*}
\end{lem}

We are now ready to prove Lemma \ref{lemma: behaviour_around_max 2.0}.
\begin{proof}[Proof of Lemma \ref{lemma: behaviour_around_max 2.0}]
	To simplify the notation, we write
		\begin{align*}
	H(r) = \int_{\partial(rB)}u^2 dS,
	\end{align*}
	and  $\beta(r)$ instead of the  frequency $\beta(0,r)$. 
	We know that $H$ is a non-decreasing function, and thanks to Corollary \ref{cor: formulas for the growth}, we also know a great deal about the growth of $H$ in an interval where $\beta$ is nearly constant. For harmonic functions, we can compare the $L^2$ norm $\sqrt{H(\cdot)}$ with the $L^{\infty}$ norm $M(\cdot)$. Hence we should also be able to get information about the growth of $M(\cdot)$ in an interval where $\beta$ is nearly constant. 

	We now establish the upper bound in \eqref{eqn:stab.growth.max.comp.}, the lower bound can be shown using a similar argument. Let $x_2 \in \partial (t_2 B)$ be such that 
$$M(t_2)= |u(x_2)|,$$
and let us write $s= (t_2-t_1)$. We have 
\begin{align}
	\label{eq: 4.1}
	u(x_2)^2 \lesssim s^{-2n} \left(\int_{B(x_2,s)}|u| \right)^2 \lesssim s^{-n} \int_{B(x_2,s)}u^2\leq s^{-n} \int_{B(0,t_2+s)}u^2.
\end{align}
Since  $H(\cdot)$ is non-decreasing, we have  
\begin{align}\label{eq_sec6_eq1}
\int_{rB} u^2= \int_{0}^{r} \int_{\partial(tB)} u^2(x) dS(x) dt = \int_{0}^r H(t)dt\leq r H(r).
\end{align}
Since $(t_2 + s)<4$, we have the following comparison between $H$ and $M$:
\begin{align}\label{eq_hm_comparison}
	M(t_2)^2 \mathop{\lesssim} \limits_{\eqref{eq: 4.1}} s^{-n}\int_{B(0,t_2 +s)}u^2 \mathop{\lesssim} \limits_{\eqref{eq_sec6_eq1}} s^{-n}  H(t_2 +s). 
\end{align}
Since $(t_2+s)  \in (\rho-10w, \rho + 10w)$,  we have by our  assumption on $\beta$ that  $\beta(t_2+s) \asymp N$.   Corollary \ref{cor: formulas for the growth} implies that for some $C=C(n)>0$, we have
\begin{align*}
	\frac{H(t_2+s)}{H(t_1)}\leq \left(\frac{t_2+s}{t_1}\right)^{C N}= \lb \frac{t_1 +2s}{t_1} \rb^{C N} \leq (1+2s)^{C N} \leq  \exp(2C sN),
\end{align*}
since $t_1 >1$. Hence 
\begin{align}\label{hm_eq_comparison}
	H(t_1) \geq H(t_2 +s) \exp({-2Cs N}) \mathop{\gtrsim} \limits_{\eqref{eq_hm_comparison}} s^n \exp({-2Cs N}) M(t_2)^2.
\end{align}
From the trivial comparison $H(t_1) \lesssim M(t_1)^2$ and \eqref{hm_eq_comparison}, we have
\begin{align*}
	M(t_2)^2 \lesssim M(t_1)^2 \exp(2Cs N + n \log\tfrac{1}{s}). 
\end{align*}
Since $s \asymp w$, and $Nw \gg\log \tfrac{1}{w}$, we also have $Ns \gg \log \tfrac{1}{s}$. Hence there is a constant $\Lambda = \Lambda(n)$ such that 
\begin{align*}
	2Cs N + n \log \tfrac{1}{s} \leq 2\Lambda s N = 2\Lambda (t_2 - t_1)N, 
\end{align*}
and this establishes the upper bound in \eqref{eqn:stab.growth.max.comp.}. 
\end{proof}

We now prove Lemma \ref{lem:stability_constant} and Lemma \ref{lem: scaling around the maximum}.

\begin{proof}[Proof of Lemma \ref{lem:stability_constant}] Let $x_0$ be a point where $|u|$ attains its maximum in $\rho B$. Let $D \subset \rho B$ be a ball of radius $w$ such that $x_0\in \partial D$. Then we have $4D \subset (\rho+3w)B$ and by Lemma \ref{lemma: behaviour_around_max 2.0}
	\begin{align*}
	\sup_{4D}|u|\leq M(\rho+3w) \leq  M(\rho) \exp(3\Lambda w N).
	\end{align*}
	Since $\thalf D \subset (\rho-\thalf w)B$, by Lemma \ref{lemma: behaviour_around_max 2.0} we have
	\begin{align*}
	\sup_{\frac{1}{2}D} |u| \leq M(\rho - \thalf w) \leq M(\rho) \exp(-\lambda w N/2).
	\end{align*}
	Using $\sup_{D}|u|=M(\rho)$ and $\sup_{2D}|u|\geq M(\rho)$, we have
	\begin{align*}
	\frac{\lambda w N}{2} \leq \log \frac{\sup_{D} |u|}{\sup_{\frac{1}{2}D}|u|}  = \N(\tfrac{1}{2}D)\text{ }\text{ and }\text{ } \N(2D) = \log \frac{\sup_{4D}|u|}{\sup_{2D}|u|} \leq 3\Lambda w N,
	\end{align*}
	which concludes the proof of Lemma \ref{lem:stability_constant}.
	
	\end{proof}

\begin{proof}[Proof of Lemma \ref{lem: scaling around the maximum}]
  The proof is similar to the proof of Lemma \ref{lem:stability_constant} and  follows  by observing that 
\begin{align*}
\sup_{{D}}|u| \geq M(\rho)  \text{ } \text{ and } \text{ } \sup_{2{D}}|u| \leq M(\rho) \exp(3\Lambda wN),
\end{align*} 
where the second inequality holds because  $2{D} \subset (\rho+3w)B$.  This completes the proof of Lemma \ref{lem: scaling around the maximum}. 
	\end{proof}

\subsection{Proof of Theorem \ref{thm: main} assuming Proposition \ref{prop: lemma 2}.} \label{sec6.2}

We now state a useful claim about the frequency. Although we state it specifically for the frequency, 
an analogous statement is true for more general monotone functions $\varphi:[a,b] \ra \real$.
 It essentially says that for any  monotone function, it is possible to find a subinterval of \textit{significant} length where the function is almost constant. 
\begin{claim}
\label{claim: existence of stable region} 
   Let $u$ be a harmonic function defined in $4B$, with $B =B(0,1) \subset \mathbb{R}^n$. Suppose that $\mathcal{N}(\frac{1}{2}B)=N_1\gg_{n} 1$ and let $\mathcal{N}(2B)=N_2$. Then there exists some $N$ satisfying  $ N_1 \lesssim_{n} N \lesssim_{n} N_2$ such that 
   \begin{align*}
   \beta(0,t)\in [N, eN]\text{,  }\text{$\forall t \in \mathcal{I} $}
   \end{align*}
   where $\mathcal{I} \subset [1.1,1.9]$ is some interval of length $|\mathcal{I}|\gtrsim_n (\log N)^{-2}$.
\end{claim}
\begin{proof}[Proof of Claim \ref{claim: existence of stable region}] In this proof, we write $\beta(r)$ instead of $\beta(0,r)$. We first observe that as a consequence of Lemma \ref{lem: comparability di and frequency} and assumption $N_1 \gg_{n} 1$, we have
	\begin{align*}
	1 \ll_{n} N_1 \lesssim_n \beta(1.1) \leq \beta(1.9) \lesssim_n N_2.
	\end{align*}
	Depending on the growth of the frequency $\beta$, we partition the interval $[1.1,1.9]$ as follows. Let $k \in \nat$ be such that 
	\begin{align*}
	k:=\lfloor \log \tfrac{\beta(1.9)}{\beta(1.1)} \rfloor.
	\end{align*}
	For $0 \leq j \leq k$, let $a_j \in [1.1,1.9]$ be defined by 
	\begin{align*} 
	a_j :=  \inf\{t\in [1.1,1.9]: \beta(t)\geq 10^{j}\beta(1.1)\}.
	\end{align*}
	If there is some $j$ such that $0\leq j \leq k-1$ and
	\begin{align}\label{eq:diff_of_aj}
	a_{j+1} - a_{j} \geq \frac{1}{100} (\log \beta(a_j))^{-2},
	\end{align}
then we can take $\mathcal{I} = [a_j,a_{j+1}]$. 
If \eqref{eq:diff_of_aj} does not hold for any $j$, then we have
\begin{align*}
\sum_{j=0}^{k-1} a_{j+1} - a_{j} &\leq \frac{1}{100} \sum_{j=0}^{k-1}(\log \beta(a_j))^{-2} \leq \frac{1}{100} \cdot \sum_{j=0}^{k-1} \frac{1}{(\log \beta(1.1) + j)^2},\\
a_k - 1.1 & \leq \frac{1}{100} \cdot \sum_{j=0}^{\infty} \frac{1}{j^2} \leq 0.02,
\end{align*}
	and hence $a_k \leq 1.12$. By our definition of $k$, we also have $1\leq \beta(1.9)/\beta(a_k) \leq e$. Hence 
	\begin{align*}
	\beta(t) \in [\beta(a_k),e\beta(a_k)],\text{ }\forall t\in [a_k,1.9].
	\end{align*}
	Also notice that $1.9 - a_k \geq 0.7 \gg (\log \beta(1.1))^{-2} \geq (\log \beta(a_k))^{-2}$. This shows that at least one of the  intervals $[a_j,a_{j+1}]$ or $[a_k,1.9]$ satisfies the required condition, and this completes the proof of Claim \ref{claim: existence of stable region}. 
\end{proof}
Lemma \ref{lem: induction on scales}  implies Theorem \ref{thm: main}, under the stable growth assumption. We now show that Theorem \ref{thm: main} also follows in the general case.
Let us recall the setting of Theorem \ref{thm: main}. We have the unit ball $B=B(0,1)$, and a harmonic function $u$ in $4B$, which vanishes at the center of $B$. We wish to get a lower bound for the nodal volume of $u$ in $2B$, in terms of $\N(\thalf B)$.

\begin{proof}[Proof of Theorem \ref{thm: main}]
Define $N_1:= \N(\thalf B)$ and fix $\varepsilon >0$. We may assume that $N_1 \gg_{\varepsilon, n} 1$, otherwise we may employ the elementary Lemma \ref{claim: bounded doubling index}.

 We first use Claim \ref{claim: existence of stable region} to conclude that there is $N \gtrsim_n N_1$, and an interval $\mathcal{I}$ of length $|\mathcal{I}| \asymp (\log N)^{-2}$ on which $\beta \in [N,eN]$.  Then  we apply Lemma \ref{lem:stability_constant} to $\mathcal{I}$ and conclude that there is a ball $D$ of radius $\asymp (\log N)^{-2}$ with stable growth satisfying: $\N(D) \asymp N(\log N)^{-2}$ and $4D \subset 2B$. Finally, we use Lemma \ref{lem: induction on scales} for this ball to get the following lower bound: 
\begin{align*}
\mathcal{H}^{n-1}(\{u=0\} \cap 2D) &\gtrsim_{n,\ep} ({N}(\log N)^{-2})^{1-\ep} \cdot (\log N)^{-2(n-1)}
 \geq N^{1-\ep}(\log N)^{-2n}
 \geq N^{1-2\ep},
\end{align*}
where the last inequality follows because $N \gg 1$. Now since $4D \subset 2B$, we can also conculde that 
\begin{align*}
\mathcal{H}^{n-1}(\{u=0\} \cap 2B) &\gtrsim_{n,\ep} N^{1-2\varep},
\end{align*}
and this completes the reduction of Theorem \ref{thm: main} to Lemma \ref{lem: induction on scales}. 
\end{proof}
Lemma \ref{lem: induction on scales} was already reduced to Proposition \ref{prop: lemma 2}, which is the only thing left to prove. 
\section{Reducing Proposition \ref{prop: lemma 2} to the stable growth case}

\label{sec: reduction}

The aim of this section is to show that it suffices to prove Proposition \ref{prop: lemma 2} under the additional assumption that $u$ has stable growth in $B$. The following result is exactly Proposition \ref{prop: lemma 2} with the extra assumption of stable growth.

\begin{lem}
	\label{lem: extra assumption lemma 2}  Let $B$ be a unit ball, $u: 4B\subset \mathbb{R}^n \ra \real$ be a harmonic function and $A>1$ be a sufficiently large parameter. There exists a constant $C_A>1$ such that the following holds. Assume that $u$ has stable growth in $B$ and $\N(\frac{1}{2}B)=N\geq C_A$. Then there exists a collection of disjoint balls  $B_i = B(x_i,r_i)$ such that:
	\begin{enumerate}
	\item We have an upper bound on the doubling index:
$$\max_i \N(2B_i) \lesssim_n N/A.$$ 
\item  We have a lower bound on the total s.d.i.'s:
$$	\sum \N (\tfrac{1}{2}B_i) r_{i}^{n-1} \gtrsim_n \frac{N_1}{(\log A\log\log A)^{n-1}}. $$

\item Vanishing at the center and non-intersection properties:
$$u(x_i) = 0, \text{ }\text{ }4B_i \cap 4 B_j=\emptyset,\text{ } \text{ and } \text{ }4B_i\subset 2B,\text{ }\forall i\neq j.$$
\end{enumerate}
\end{lem}
\begin{proof}[Proof of Proposition \ref{prop: lemma 2} assuming Lemma \ref{lem: extra assumption lemma 2}] Let $x$ be the center of $B$; we will write $\beta(\cdot)$ for $\beta(x,\cdot)$.
	Depending on the growth of the frequency $\beta$, we divide the interval $[1.1,1.9]$ as follows. Let $k \in \nat$ be such that $k:=\lfloor \log \tfrac{\beta(1.9)}{\beta(1.1)} \rfloor$. For $0 \leq j \leq k$, let $a_j \in [1.1,1.9]$ be such that $\beta(a_j) = 10^{j} \beta(1.1)$. Define the interval $I_j:= [a_j, a_{j+1}]$ of length $20w_j := a_{j+1} - a_j$. Observe that on each of the intervals $I_j$, the frequency $\beta$ does not grow more than by a factor of $10$. If $w_j$ satisfies 
	\begin{align}
	\label{goodj}
	10^j \beta(1.1) \geq C_n \cdot  \tfrac{1}{w_j} \log \tfrac{1}{w_j},
	\end{align}
	for some sufficiently large constant $C_n>0$, then we can use Lemma \ref{lem:stability_constant} to get a ball $B_j$ of stable growth in $B(a_{j+1})\setminus B(a_j)$.
	But some of these intervals $I_j$ can be too short and \eqref{goodj} may not hold for these intervals, and we will discard them. 
We will show that the sum of the $w_j$ satisfying \eqref{goodj} is at least $1/2$ if $\beta(1.1)$ is large enough (which holds because $\beta(1.1)\gtrsim \mathcal{N}(\frac{1}{2}B) \geq C_A \gg 1$).

Since $w_j \in (0,1)$, and $x\geq \log x$ for $x>1$,  we have 
\begin{align*}
\frac{1}{w_{j}^{2}} \geq \frac{1}{w_j} \log \frac{1}{w_j}.
\end{align*}
 Hence if $w_{j}^{2} \cdot 10^j \beta(1.1) \geq C_n$, then \eqref{goodj} holds. Thus, letting $\mathcal{J}$ be the collection of indices $j$ for which 
$w_{j}^{2} \cdot 10^j \beta(1.1) \leq C_n$, we have
\begin{align*}
	\sum_{j \in \mathcal{J}} w_j \leq \sqrt{\frac{C_n}{\beta(1.1)}}~\lb  \sum_{j \in \nat} \frac{1}{{10}^{j/2}} \rb \leq \sqrt{\frac{C_n}{\beta(1.1)}}. 
\end{align*}
Hence the sum $\sum_{j \in \mathcal{J}} w_j < 1/3$ for $\beta(1.1)\gg 1$. Define 
$${G} = \{0,1,\ldots,k\} \setminus \mathcal{J}.$$
We have
\begin{align}\label{sum_lbound}
	\sum_{j \in {G}} w_j \geq 1/2,
\end{align}
and for each $I_j$, $j \in G$, we can apply Lemma \ref{lem:stability_constant} to find a ball $B_j$ such that:
\begin{itemize}
	\item $u$ has stable growth in $B_j$, and $\N(B_j) \asymp w_j 10^j \beta(1.1)$.
	\item The radius of $B_j$ is $w_j$, and $4B_j \subset B(0,a_{j+1})\setminus B(0,a_j)$.
	\end{itemize}
The latter property implies that $4B_j$ are disjoint. 


In order to simplify the notation, let us write $\Phi(A)= (\log A \log \log A)^{n-1}$ until the end of the proof. For each $B_j$, $j \in {G}$, we use Lemma \ref{lem: extra assumption lemma 2} to get a collection of disjoint sub-balls $\{B_{i,j}\}_{i}$, with radii $r_{i,j}$ and  centers $x_{i,j}$, satisfying the following conditions:
\begin{enumerate}[label=(\roman*),align=left,leftmargin=*,widest={7},itemsep=.2cm]
	\item \label{one11}  Upper bound for the doubling index:
	$$\max_i \N(2B_{i,j}) \lesssim_n \N(B_j)/A \asymp_n w_j 10^j \beta(1.1)/A.$$ 
	\item \label{two22} Lower bound for the total s.d.i.:
	\begin{align}\label{eqtwo22}
		\sum_{i } \N\left(\tfrac{1}{2}B_{i,j}\right) r_{i,j}^{n-1} \gtrsim_n \frac{\N(B_j)}{\Phi(A)} \cdot w_{j}^{n-1} \asymp \frac{\beta(1.1)}{\Phi(A)} \cdot 10^j w_{j}^{n}. 
	\end{align}
	\item \label{three33} Vanishing at the centers and disjointness: $$u(x_{i,j}) = 0, \text{ } \text{ } 4B_{i,j} \cap 4 B_{k,j}=\emptyset,  \text{ } \text{and} \text{ } 4B_{i,j}\subset 2B_j,\text{ }\forall i\neq k.$$
\end{enumerate}
Since $w_j\leq 1$, \ref{one11} and Lemma \ref{lem: comparability di and frequency}  imply that
$$\max_i \N(2B_{i,j}) \lesssim_n w_j 10^j \beta(1.1)/A  \lesssim_n \beta(1.9)/A\lesssim_n \mathcal{N}(2B)/A. $$
We observe that \ref{three33} implies $(3)$ in Proposition \ref{prop: lemma 2}. Therefore, we are left to show $(2)$.   It suffices to show that, for $j \in G$, we have 
\begin{align}
	\label{eq: left to prove}
	\sum_{i,j} \N(\tfrac{1}{2}B_{i,j})~ r_{i,j}^{n-1} \gtrsim_n  \frac{\N(\thalf B)}{\Phi(A)}.
\end{align}
Indeed, summing \eqref{eqtwo22} over all $j \in G$ and using the bound $\beta(1.1)\gtrsim \mathcal{N}(\frac{1}{2}B)$, which follows from Lemma \ref{lem: comparability di and frequency}, we get
\begin{align}
	\label{eq: 4.3}
	\sum_{j \in G}\sum_{i} \N(\tfrac{1}{2}B_{i,j})~ r_{i,j}^{n-1} \gtrsim_n  \frac{\N(\thalf B)}{\Phi(A)} \sum_{j \in G} 10^j w_{j}^{n}. 
\end{align}
We will get a lower bound for $\sum_{j \in G} 10^j w_{j}^{n}$ by using  Holder's inequality which says that, for $\alpha_m, \gamma_m >0$, and $p,q \geq 1$ such that $p^{-1} + q^{-1} = 1$, we have
\begin{align}\label{holder}
	\sum_{m} \alpha_m \gamma_m \leq \lb \sum_m \alpha_{m}^p\rb^{1/p} \lb \sum_m \gamma_{m}^q\rb^{1/q}. 
\end{align}
Choosing $p=n$, $q=\tfrac{n}{n-1}$, $\alpha_j = w_{j} 10^{j/n}$, $\gamma_{j} = 10^{-j/n}$ in \eqref{holder}, we get
\begin{align}\label{ineq100}
	\sum_{j \in G} w_j \leq \left(\sum_{j \in G} 10^j w_{j}^{n}\right)^{\frac{1}{n}} \cdot  \left( \sum_{j \in \nat} 10^{-j/n-1}\right)^{\frac{n-1}{n}}.
\end{align}
By  \eqref{sum_lbound},  $\sum_{j \in G}w_j \geq 1/2$. We deduce that there exists constant $c_n>0$ such that 
$$ \sum_{j \in G} 10^j w_{j}^{n} \geq c_n.$$
Using this in \eqref{eq: 4.3}, we have shown that \eqref{eq: left to prove} holds. This concludes the proof of Proposition \ref{prop: lemma 2} (assuming Lemma \ref{lem: extra assumption lemma 2}). 
	\end{proof}

\section{Zeros of a harmonic function in a tunnel}

\label{sec: general lemma}

In this section we introduce the notion of a \textit{tunnel}. A \textit{tunnel} $\mathcal{T} \subset \real^n$ is a closed hyper-rectangle with all sides  but one of length $h>0$, while  the length $\ell$ of the remaining side is much larger than $h$. That is, $\ell \gg_{n} h$. We call $h$ and $\ell$ the \textit{width} and the \textit{length} of the tunnel respectively. We will always tacitly assume that $m := \ell/ h \in \nat$. A tunnel $\tun$ has two faces whose sides are all of length $h$. We designate one of these to be the \textit{beginning face} and the other one to be the \textit{end face}, and we  will denote them by  $F_{\text{beg}}$ and $F_{\text{end}}$ respectively. We can find a unit vector $v \in \real^n$ (which is parallel to the longer side of $\tun$) and $a\in \real$ such that 
\begin{align*}
F_{\text{beg}} = \tun \cap \{x \in \real^n: v \cdot x = a\}\text{ and }F_{\text{end}} = \tun \cap \{x\in \real^n: v \cdot x = a+\ell\},
\end{align*}
where $\cdot$ denotes the usual inner product in $\real^n$. We partition $\tun$ into equal hyper-cubes of side length $h$, and denote this collection by $\{q_i\}_{i=1}^{m}$. More precisely, for $1\leq k \leq m$: 
\begin{align*}
q_k := \tun \cap \{x\in \real^n : a+ (k-1)h \leq v\cdot x \leq a + kh\}. 
\end{align*}
Note that $F_{\text{beg}} \subset q_1$ and $F_{\text{end}} \subset q_m$. We will call subcubes $q_1$ and $q_m$ the \textit{beginning} and the \textit{end} of $\tun$ respectively. Henceforth, whenever we talk about a tunnel, we implicitly assume that it comes with a choice of the beginning and the end.  

With $\tun$ as above, a harmonic function $u$ defined in a neighbourhood of $\tun$ is said to have \textit{logarithmic multiplicative increment} (l.m.i.) $Z>0$ in $\tun$ if 
\begin{align}
\label{def: multiplicative increment}
\log\frac{\sup_{q_m}|u|}{\sup_{q_1}|u|}= Z.
\end{align}

\vspace{-3mm}

\begin{figure}[h]
	\includegraphics[scale=0.3]{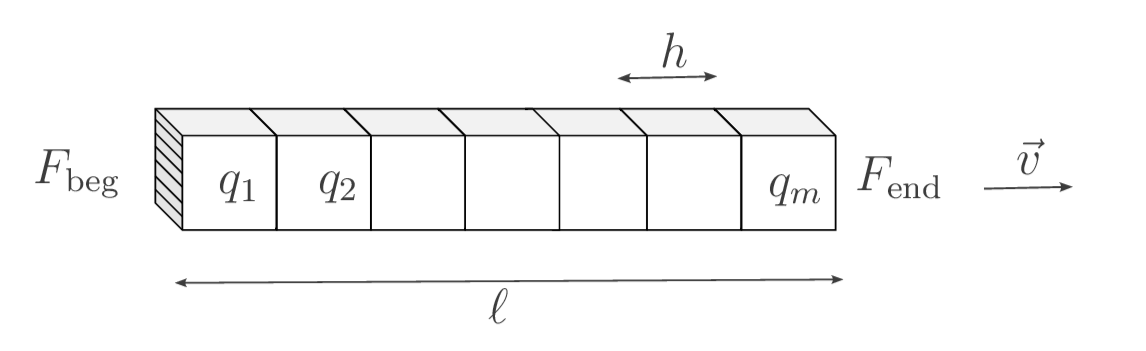}
	\centering
	\caption{A pictorial representation of tunnel $\tun$ described above.} 
\label{fig:fig_tun}
\end{figure}

We now recall the notion of  the {doubling index} of a cube introduced in 
\eqref{defn_di_cube}: given a cube $Q$ with diameter $r>0$, and a harmonic function on $\sqrt{n}Q$,  we write
\begin{align*}
\mathcal{N}^*(Q)= \sup_{x \in Q, \hskip .05cm \rho \in (0,r) } \N(x,\rho).
\end{align*}
We recall the notion of $\ell$-neighbourhood of a set $S \subset \real^n$:
\begin{align*}
S_{+\ell} := \{x \in \real^n : x= s+b\text{, where }s\in S\text{ and }b\in B(0,\ell) \}. 
\end{align*}

In the following lemma, we show that if a harmonic function $u$ has  a large enough multiplicative increment $Z$ in a tunnel $\tun$, and the doubling indicies in all of the subcubes $q_k$ are much smaller than $Z$, then the zero set of $u$ in a neighbourhood of $\tun$ is non-empty and is well spread out. That is, there are many disjoint balls (of radius $\asymp h$) where $u$ vanishes. 

\begin{lem}
	\label{lem: general lemma}
	Let $B \subset \real^n$ be a unit ball and let  $u$ be a harmonic function in $4B $. Let $\tun \subset {1.6B}$ be a tunnel of length $\ell$ and width $h$.  Suppose that  $u$ has l.m.i. $Z>1$ in $\tun$, and that  $m= \ell/ h\ll_n Z$. Moreover, suppose that
	$$\max_k \mathcal{N}^*(100\sqrt{n}q_k)\leq Z/V,$$
	for some positive number $V \ll_n Z$. Then there exist $\gtrsim_n V$ many disjoint balls $B_{k} = B(y_k,r_k)$ such that $B_{k} \subset \tun_{+50\sqrt{n}h}$  and the following holds.
	\begin{enumerate}
		\item \label{cond1} 
		Vanishing and non-intersection properties:
		$$u(y_i) = 0\text{, } \text{ } 4B_i \cap 4 B_j=\emptyset \text{, and }\text{ } 4B_i\subset 2B,$$ 
		for all $i \neq j$.
		\item \label{cond2} 
		A lower bound for the sum of the doubling indices: 
		$$ Z \lesssim_n \sum_k \mathcal{N}(\tfrac{1}{2}B_k).$$
	\end{enumerate}
	
\end{lem}

\begin{proof}[Proof of Lemma \ref{lem: general lemma}]
	Let $z_k$ be the center of $q_k$. We start with the following claim. 
	

	\begin{claim}
		\label{claim: there are zeros}
		Let $S \subset \{1,2,\ldots,m-1\}$ be the collection of of $k$'s such that $(2q_k \cup 2q_{k+1})$ contains a zero of $u$. Then we have
		\begin{align*}
		|S| \gtrsim_n V \text{ }\text{ and } \text{ }
		\sum_{k \in S} \N(z_k, 4\sqrt{n}h) \gtrsim_n Z.
		\end{align*}
	\end{claim}
	\begin{proof}[Proof of Claim \ref{claim: there are zeros}]
		We first observe that 
		\begin{align*}
		Z = \log\frac{\sup_{q_m}|u|}{\sup_{q_1}|u|}= \sum_{k=1}^{m-1}  \log\frac{\sup_{q_{k+1}}|u|}{\sup_{q_k}|u|}.
		\end{align*}
		If $u$ does not vanish on  $(2q_{k} \cup 2q_{k+1})$, then by Harnack's inequality,  there exists some constant $C_0=C_0(n)>0$ such that 
		\begin{align*}
		\log \frac{\sup_{q_{k+1}}|u|}{\sup_{q_{k}}|u|} \leq C_0.
		\end{align*}
		Therefore, writing $S^{c}=\{1,2,\ldots,m-1\} \backslash S$, we have
		\begin{align*}
		Z &= \sum_{k \in S^{c}}  \log\frac{\sup_{q_{k+1}}|u|}{\sup_{q_k}|u|} + \sum_{k \in S}  \log\frac{\sup_{q_{k+1}}|u|}{\sup_{q_k}|u|},\\
		& \leq C_0 m + \sum_{k \in S}  \log\frac{\sup_{q_{k+1}}|u|}{\sup_{q_k}|u|}.
		\end{align*}
		Since  $m\ll Z$, we have $C_0 m \leq Z/2$ and hence
		\begin{align}
		\label{ineqq1}
		\frac{Z}{2} & \leq  \sum_{k \in S} \log  \frac{\sup_{q_{k+1}}|u|}{\sup_{q_{k}}|u|}.
		\end{align}
		Since $B(z_k, \tfrac{h}{2}) \subset q_k$, and $q_{k+1} \subset B(z_k, 2\sqrt{n}h)$, we conclude from Lemma \ref{lem: monotonicity doubling index}  that 
		\begin{align}\label{ineqq2new}
		\log   \frac{\sup_{q_{k+1}}|u|}{\sup_{q_{k}}|u|} &\leq  \log \frac{\sup_{B(z_k, 2\sqrt{n}h)}|u|}{\sup_{B(z_k, \tfrac{h}{2})}|u|}, \nonumber \\
		& \leq \log [(4\sqrt{n})^{2\N(z_k, 4\sqrt{n}h) + C_1}], \nonumber \\ 
		&\leq C_2 \N(z_k, 4\sqrt{n}h) +C_2,
		\end{align}
		where $C_1, C_2>0$ are dimensional constants. 
		 From \eqref{ineqq1} and  \eqref{ineqq2new}, we have
			\begin{gather*} 
		\frac{Z}{2} \leq \sum_{k \in S} \left[C_2 \N(z_k, 4\sqrt{n}h) + C_2 \right] \leq C_2 \sum_{k \in S} \N(z_k, 4\sqrt{n}h) + C_2 m. 
		\end{gather*}
		 Since  $m\ll Z$, we have
		\begin{gather*} 
		{Z}\lesssim  \sum_{k \in S} \N(z_k, 4\sqrt{n}h). 
		\end{gather*}
		By our assumption, each $\N^{*}(4\sqrt{n} q_k) \leq Z/V$, and hence $\N(z_k,4\rootn h) \leq Z/V$.  Hence	we have $|S| \gtrsim_n V$. This completes the proof of Claim \ref{claim: there are zeros}. 
	\end{proof} 
	
	
		The collection of balls obtained in Claim \ref{claim: there are zeros} satisfy \eqref{cond2}, but the balls  might intersect each other and the zeros are not necessarily at their centers. It is easy to modify this collection to also satisfy \eqref{cond1}. 
	
	For $k \in S$, let $y_k \in (2q_k \cup 2q_{k+1})$ be a point where $u$ vanishes.  The distance between $y_k$ and $z_k$ is at most $3\sqrt{n}h$. So
	\begin{align*}
	B(z_k, 8\sqrt{n}h)  \subset B(y_k, 16\sqrt{n}h).
	\end{align*}
	By almost monotonicity of the doubling index for non-concentric balls (Lemma \ref{lem: almost monotonicity})
	\begin{align*}
	\N(z_k, 4\sqrt{n}h) \lesssim_n \N(y_k, 16\sqrt{n}h). 
	\end{align*}
	Thus by Claim \ref{claim: there are zeros} 
	\begin{align*}
	\sum_{k \in S} \N(y_k, 16\sqrt{n}h)  \gtrsim_{n} \sum_{k \in S} \N(z_k, 4\sqrt{n}h)  \gtrsim_{n} Z. 
	\end{align*}
	Define $B_k := B(y_k, 32\sqrt{n}h)$.  
	Then $$\sum_{k \in  S} \N(\tfrac{1}{2} B_k) \gtrsim_{n} Z .$$ 
	Since each of the balls $4B_k$ intersects at most $C=C(n)$ of other balls $4B_l$,
	we can split the set $\{ 4B_k\}_{k \in S}$ into at most $C$ collections of balls such the balls in each collection do not intersect.
	Let  $\{4B_k\}_{k  \in \tilde S}$ be the collection of disjoint balls with the maximal sum $\sum_{k \in  \tilde S} \N(\tfrac{1}{2} B_k)$. Then
	$$\sum_{k \in  \tilde S} \N(\tfrac{1}{2} B_k) \geq \frac{1}{C} \sum_{k \in S} \N(\tfrac{1}{2} B_k) \gtrsim_{n} Z.$$ 
	The collection of balls $\{B_k\}_{k \in \tilde S}$ satisfies \eqref{cond1} and \eqref{cond2}. This completes the proof of Lemma \ref{lem: general lemma}.

\end{proof}

\section{Connecting the dots: Proof of Lemma \ref{lem: extra assumption lemma 2}}
\label{sec: conencting the dots}

Throughout this section, we will use the notation for tunnels introduced in \S \ref{sec: general lemma}. 


The proof of Lemma \ref{lem: extra assumption lemma 2} is split into three parts.
In \S \ref{subsec9.1} we  show how to construct tunnels in $2B= B(0,2)$, which have l.m.i. at least $N=\mathcal{N}(\frac{1}{2}B)$. In \S \ref{subsec9.2} we use information about the distribution of doubling index from Proposition \ref{prop: lemma 1} to prove that many of these  tunnels  are \textit{good}. Good tunnels are those for which the l.m.i. is at least $N$, and the doubling index in each of its subcubes is much smaller than $N$. 
Finally in \S \ref{subsec9.3} we use Lemma \ref{lem: general lemma} in good tunnels to conclude the proof of Lemma \ref{lem: extra assumption lemma 2}. 

\subsection{Step 1: Constructing tunnels with large multiplicative increment}\label{subsec9.1}








We first start with a lemma about the growth of harmonic functions. 
\begin{claim}\label{claim: off-setting formula} Let $B=B(0,1) \subset \real^n$ and let $u$ be a harmonic function in $B$. There exists $C=C(n)>0$ such that the following holds. For every $r_1, r_2$ satisfying $0< 2r_1 < r_2 < 4^{-1}$, and $y \in \partial(r_2B)$, we have
	\begin{align*} 
	\sup_{B(0,r_1)}|u| \geq |u(y)|\left(\frac{r_2}{r_1}\right)^{-2\mathcal{N}(0,r_2) - C}.
	\end{align*}
\end{claim}
\begin{proof}[Proof of Lemma \ref{claim: off-setting formula}]
	We use monotonicity of doubling index (Lemma \ref{lem: monotonicity doubling index}) to see that 
	$$\frac{|u(y)|}{\sup_{B(0,r_1)}|u|}\leq \frac{\sup_{B(0,r_2)}|u|}{\sup_{B(0,r_1)}|u|}\leq \left(\frac{r_2}{r_1}\right)^{2\mathcal{N}(0, r_2) +C},$$
	as required. 
\end{proof}

Consider a  harmonic function $u$ with stable growth in the unit ball $B=B(0,1)$, with $\N(\frac{1}{2}B) =N$. In the following claim, we consider a tunnel $\tun$ whose end is near a point where $|u|$ attains its maximum on $\partial(1.5 B)$, and whose beginning is contained well within $1.5B$ (for example in $1.3B$). We show that if $\mathcal{T}$ is not too narrow, then  $u$ has l.m.i. $Z \gtrsim N$ in $\tun$. 

\begin{lem} 	\label{lem: multiplicative increment} Let $B=B(0,1) \subset \real^n$. Assume that  a harmonic function $u$ in $4B$ has stable growth in $B$. Let $x_0 \in \partial (1.5B)$ be a point where $|u|$ attains its maximum in $1.5B$. 	Let $\mathcal{T}$ be any tunnel of length ${1}/{4}$ and width $h$ such that the beginning $q_1$ satisfies
	\begin{align*}
	q_1 \subset 1.3B.
	\end{align*}
	Denote by $z_m$ the center of the ending $q_m$ of $\tun$ and put $r=|z_m -x_0|$. Assume that
	\begin{align*}
	h \ll_{n} 1\text{ }\text{ and }\text{ }\frac{h}{2}\leq r \ll_{n} (\log\tfrac{1}{h})^{-1}
	\end{align*}
	and  
	$$\N(\thalf B) = N \gg_n \tfrac{1}{h} \log \tfrac{1}{h}.$$ Then $u$ has multiplicative increment $Z \gtrsim_n N$ in $\mathcal{T}$.
\end{lem}

\begin{proof}[Proof of Lemma\ref{lem: multiplicative increment}]
	
	Let $M=|u({x_0})|$. We can apply Lemma \ref{lemma: behaviour_around_max 2.0} to $u$ in the spherical layer $\{x\in \mathbb{R}^n: 1.1\leq |x|\leq 1.9\}$ to conclude that 
	\begin{align}\label{eq_q1}
	\sup_{q_{1}} |u|\leq \sup_{B(0,1.3)} |u| \leq \sup_{B(0,1.5)} |u|\exp(-c_1 N) = M \exp(-c_1 N)
	\end{align}
	for some  constant $c_1=c_1(n)>0$.  We conclude from Claim \ref{claim: off-setting formula}  that
	\begin{align} \label{eq: tunnelend}
	\sup_{q_m} |u| \geq \sup_{B(z_m, {h}/{4})} |u| \geq |u(x_0)| \lb \frac{4r}{h}\rb^{-2\N(z_m,r)-C},
	\end{align}
	for some $C=C(n)>0$. Since  $$Nr(\log\tfrac{1}{r})^{-1} \geq Nh ( \log \tfrac{1}{h})^{-1} \gg_n 1$$ 
	we can use  Lemma \ref{lem: scaling around the maximum} to get $$\N(z_m, r) \lesssim_{n} rN.$$ Let's plug the latter inequality in \eqref{eq: tunnelend} and use $rN \gg 1$ :
	\begin{align}\label{eq_qm}
	\sup_{q_m} |u| \geq M (2r/h)^{-C_1rN} = M\exp({-C_1rN \log(4r/h)}) \nonumber \\
	\geq M\exp({-C_1rN \log (1/h)})
	\end{align}
	for some  $C_1=C_1(n)>1$.
	Hence, we have the following bound for the multiplicative increment of $u$ in $\tun$:
	\begin{align*}
	Z = \log \frac{\sup_{q_m}|u|}{\sup_{q_1}|u|} \mathop{\gtrsim}\limits_{\eqref{eq_q1},\eqref{eq_qm}} N(c_1 - C_1r\log (1/h)) \gtrsim N.
	\end{align*}
	In the last inequality we use the assumption $r \ll_n (\log \tfrac{1}{h})^{-1}$. The proof of Lemma \ref{lem: multiplicative increment} is finished.
\end{proof}

\subsection{Step 2: Seeking for many {good} tunnels}\label{subsec9.2}
A randomly chosen tunnel from Lemma \ref{lem: multiplicative increment} may have a subcube with doubling index comparable to $N$. In order to apply Lemma \ref{lem: general lemma}, we need to ensure that all the subcubes in the tunnel have doubling index $\ll_{n} N$. Such a tunnel will be called a good tunnel. The goal of this section is to construct many good tunnels in Lemma \ref{lem: good tunnels}. 

\vspace{2mm}
We are proving Lemma \ref{lem: extra assumption lemma 2} and we keep the notation introduced in its statement. Namely, $u$ is harmonic on $4B$ and has stable growth in $B=B(0,1)$, with $\N(\tfrac{1}{2}B) =N$. 

Let $x_0 \in \partial (1.5B)$ be such that $|u(x_0)| = \sup_{1.5B}|u|$.
Denote by $\mathcal{R}$ a tunnel of length $1/4$ and width $w \ll 1$, whose longer side is parallel to the vector $x_0$, and $x_0$ is the center of its end face. We will assume that  $(4w)^{-1} \in \nat$ so that we could chop $\mathcal{R}$ into equal subcubes. 
Then $\mathcal{R} \subset 1.6B$ and its beginning is contained in $1.3B$. See Figure \ref{multiple_tunnels_location} 

\begin{figure}[H]
	\includegraphics[scale=0.7]{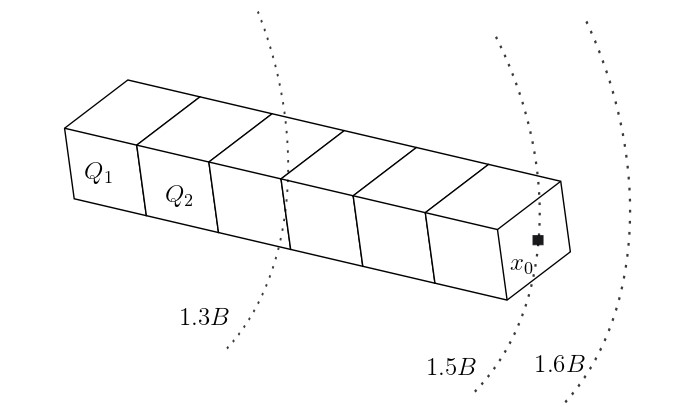}
	\centering
	\caption{Location of the tunnel $\mathcal{R}$}\label{multiple_tunnels_location}
\end{figure}

We will start chopping $\mathcal{R}$ into subcubes and we will need an integer (chopping) parameter $K \gg 1$. 
$K$ and $w$ will be chosen later, depending on $A$ in the statement of Lemma \ref{lem: extra assumption lemma 2}. 
Partition $\mathcal{R}$ into equal subcubes $\{Q_j: 1\leq j \leq (4w)^{-1}\}$ of side length $w$. Also partition $\mathcal{R}$ into equal subtunnels $\{\tun_i: 1 \leq i \leq K^{n-1}\}$ of length $1/4$ and width $w/K$. 

The tunnels $\tun_i$ are such that all their beginning faces and end faces  are contained in the beginning face and the end face of $\mathcal{R}$ respectively. We denote by $\{q_{i,t}: 1\leq t \leq K(4w)^{-1}\}$ the subcubes of tunnel $\tun_i$ with side length $w/K$. See Figure \ref{fig:fig_mult_tun} for a pictorial explanation for this partitioning. 

\begin{figure}[H]
	\includegraphics[scale=0.8]{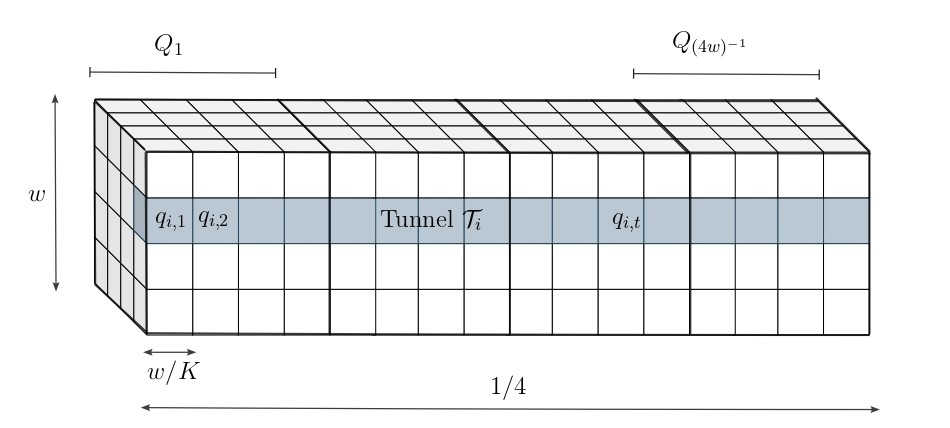}
	\centering
	\caption{A pictorial representation of $\mathcal{R}$ and its partition into tunnels $\{\mathcal{T}_i\}_i$.}\label{fig:fig_mult_tun}
\end{figure}

\noindent It is easy to see that $Q_j$ consists of $K^n$ equal subcubes $q_{i,t}$:
\begin{align*}
Q_j = \cup \{q_{i,t}: 1+ (j-1)K \leq t \leq jK\text{ and }1 \leq i \leq K^{n-1}\}.
\end{align*}

We will show  that in at least half of the tunnels $\tun_i$, the doubling index of all the scaled subcubes $100\sqrt{n}q_{i,t}$ is much smaller than $N$, by using a modified version of Proposition \ref{prop: lemma 1}, which is proved in Appendix \ref{sec: appendix}. 
\begin{cor}[A modified version of Proposition \ref{prop: lemma 1}]\label{cor_of_prop4.1} Let a unit cube $Q \subset \real^n$ be  partitioned into $K^n$ equal subcubes $\{q_{\alpha}: 1\leq \alpha \leq K^n\}$. Let $u$ be a harmonic function in $4\sqrt{n} Q$.  Let $\N^{*}(2Q) = N$. If $K \gg_n 1$, then we have
	\begin{align*}
	\#\{\alpha :\N^{*}(100\sqrt{n}q_{\alpha}) \geq \max(N\cdot 2^{-c\log K/\log \log K},C)\} \leq K^{n-1-c},
	\end{align*}
	where $c,C >0$ are dimensional constants. 
\end{cor}

The next claim is a straightforward consequence of the almost monotonicity of doubling index (Lemma \ref{lem: almost monotonicity}).
\begin{claim}
	\label{claim: upper bound doubling index cube}
	Let $B=B(0,1) \subset \real^n$. Then for every harmonic function $u$ defined in $4B$, and every cube $Q\subset B(0, 1.6)$ of side length at most $(10\sqrt{n})^{-1}$, we have 
	\begin{align*}
	\mathcal{N}^{*}(Q)\lesssim_n \mathcal{N}(2B). 
	\end{align*}
\end{claim}
We are now ready to show the existence of many good tunnels in $\mathcal{R}$.
\begin{lem}
	\label{lem: good tunnels} With the  notation as above, we define $$ N^*(\mathcal{T}_i) = \max_{t=1,2,\ldots, K(4w)^{-1}} \mathcal{N}^{*}(100\rootn q_{i,t}).$$
	If  {$K \gg_n 1$} and $w\asymp_{n} (\log K)^{-1}$, then 
	\begin{align*}
	\# \{i: 	N^*(\mathcal{T}_i) \leq  N/ C_K \} \geq \frac{K^{n-1}}{2},
	\end{align*}
	where $C_K = \exp({c' \log K/\log \log K})$ for some constant $c' = c'(n)>0$.
\end{lem}
In total there are $K^{n-1}$  tunnels $\tun_i$. The above lemma implies that  in a neighborhood  of at least  half of these tunnels the doubling index significantly drops down.

\begin{proof}[Proof of Lemma \ref{lem: good tunnels}]

	If $w \ll 1$, we can apply Claim \ref{claim: upper bound doubling index cube} to get
	$$\N^*(2Q_j) \leq C N,$$ for some constant $C=C(n)>1$.	
	
	By Corollary \ref{cor_of_prop4.1} applied to $Q_j$ we have
	\begin{align*}
	\# \{q_{i,t} \subset Q_j : \N^{*}(100\sqrt{n}q_{i,t}) \geq NC2^{-c\log K/\log \log K}\} \leq K^{n-1-c}, 
	\end{align*} 
	Put $c':=c/2$ and $C_K := \exp({c' \log K/\log \log K})$.
	If $K\gg 1$,then 	$$ N C 2^{-c\log K/\log \log K} \leq N/C_K.$$  
	There are $(4w)^{-1}$ subcubes $Q_j$. Using the assumptions $w^{-1} \asymp_n \log K$  and $K \gg 1$ we have
	\begin{align*}
	\# \{q_{i,t} \subset \mathcal{R} : \N^{*}(100\sqrt{n}q_{i,t}) \geq N/C_K\} \lesssim w^{-1} K^{n-1-c}  \ll \frac{K^{n-1}}{2}, 
	\end{align*}
	Since the total number of tunnels $\tun_i$ is $K^{n-1}$, we conclude that at least half of them satisfy $N^*(\mathcal{T}_i) \leq N/C_K$.
	
\end{proof}

\subsection{Step 3: Concluding the proof of Lemma \ref{lem: extra assumption lemma 2}}\label{subsec9.3}
In order to finish the proof we need to carefully choose the parameters $K,w$ in the construction and chopping of the tunnel $\mathcal{R}$.
Recall that $\mathcal{R}$ has  length $1/4$ and width $w$.  Its subtunnels $\tun_i$ have length $1/4$ and width $h=w/K$.

\begin{claim}\label{all_tunnels_have_m.i._N}  We can choose a dimensional constant  $c_0=c_0(n)>0$ and put  $w=c_0/ \log K$ so that the the following holds.   If $K\gg_n 1$ and  $N \gg_{n} K(\log K)^2$, then every tunnel $\tun_i \subset \mathcal{R}$ has a multiplicative increment $Z \gtrsim_n N$.  
\end{claim}
\begin{proof}[Proof of Claim \ref{all_tunnels_have_m.i._N}]
	Formally,  we just need to check that the assumptions of  Lemma \ref{lem: multiplicative increment} hold to conclude that the multiplicative increment $ Z \gtrsim_n N$.
	
	For that we need to consider the end cube $q$ of $\tun_i$ and its center $z$. 
	Recall that $x_0$ is the maximum point of $u$ on $\partial(1.5 B)$ and $ r:=|z - x_0|$.  
	
	The assumptions needed for  Lemma \ref{lem: multiplicative increment} are the following:

	\begin{align} \label{3 inequalities}
		h/2 \leq r,  \quad r \ll_n \left (\log \frac{1}{h}\right )^{-1} \quad \text{ and }  \quad \frac{1}{h} \log \frac{1}{h}\ll_n N .
	\end{align}
	
	The first inequality holds because $x_0$ is on the end face of $\mathcal{R}$ and $z$ is the center of the end cube $q$, which is of side-length $h$. 
Moreover,  the diameter of the end face of $\mathcal{R}$ is $\asymp_n w$.   See figure \ref{fig:fig_mult_tun}.
So $r \lesssim_{n} w $ and for the second inequality in \eqref{3 inequalities} it is sufficeint to have 
$$ w \ll_n \left( \log \frac{1}{h}\right )^{-1}.$$
In view of  $w=c_0 / \log  K$ and $h=K/w$, it is equivalent to
\begin{align*}
	\frac{c_0}{\log K} \cdot \log \frac{K \log K}{c_0} \ll_n 1. 
\end{align*}
This is the place where we choose the dimensional constant $c_0>0$, keeping in mind that $K\gg_n 1$.

The third inequality in \eqref{3 inequalities}  holds because we assume $N \gg_{n} K(\log K)^2$ and
$$ \frac{1}{h} \log \frac{1}{h} = \frac{K \log K}{c_0} \log \frac{K \log K}{c_0} \lesssim_n K (\log K)^2 \ll_n N.$$

\end{proof}

\textbf{Connecting the dots and choosing the parameters.}

We have defined $w=c_0 / \log K$.  The constant $K\gg_n 1$ is not chosen yet and we  assume $N \gg_{n} K(\log K)^2$ .   By the previous Claim \ref{all_tunnels_have_m.i._N} every tunnel $\tun_i \subset \mathcal{R}$ has a multiplicative increment $Z \gtrsim_n N$.  

Lemma \ref{lem: multiplicative increment} ensures that half of $\{\tun_i\}_i$  satisfy $N^*(\tun_i) \leq N/ C_K$.  
We can find a  positive portion of the good tunnels $\{\tun_i\}_i$,  which are well-separated.  That is, there is $S \subset \{1,\ldots, K^{n-1}\}$ such that $|S|\gtrsim_n K^{n-1}$,  $N^*(\tun_i) \leq N/ C_K$ for each $i \in S$, and for distinct  $i,j \in S$:
\begin{align*}
(\tun_i)_{+50\rootn w/K} \cap (\tun_j)_{+50\rootn w/K} = \emptyset.
\end{align*}

Finally for every $i \in S$, we can use Lemma \ref{lem: general lemma} in $\tun_i$ to obtain a collection of disjoint balls $\{B_{i,k}\}_k$ with center $y_{i,k}$ and radius $r_{i,k} \asymp w/K$ satisfying the following properties.
\begin{enumerate}[label=(\roman*)]
\item \label{co1} The doubling index drops down:
$$\max_{k}\mathcal{N} (2B_{i,k})\leq N/C_K.$$
\item\label{co2} From condition \eqref{cond2} of Lemma \ref{lem: general lemma}, we have the following lower bound on the sum of the doubling indices:
\begin{align} 
	\label{eq: 5.1}
	N \lesssim_n \sum_k \mathcal{N}( \tfrac{1}{2}B_{i,k}).
\end{align}
\item \label{co3} From condition \eqref{cond1} of Lemma \ref{lem: general lemma},  we have
$$u(x_{i,j}) = 0,\text{ } 4B_{i,j} \cap 4 B_{i,k}=\emptyset,   \text{ and } 4B_{i,k} \subset (\tun_i)_{+50\rootn w/K} \subset 2B,$$
for all $i\in S$ and all $j\neq k.$
\end{enumerate}

Finally, we will show that the collection $\bigcup_{i\in S}\{B_{i,k}\}$ satisfies conditions $(1)$, $(2)$ and $(3)$ of Lemma \ref{lem: extra assumption lemma 2}. 

Condition $(3)$ of Lemma \ref{lem: extra assumption lemma 2}.  almost follows  from \ref{co3}. We need also to show that for two different tunnels $\mathcal{T}_i$ and $\mathcal{T}_j$ with $i,j \in S$  any balls $B_{i,k}$ and $ B_{j,l}$ are disjoint.  This is why we constructed the tunnels to be well-separated.  Namely, $(\tun_i)_{+50\rootn w/K}, i 
\in S$ are disjoint while the balls $4B_{i,k}$ and $ 4B_{j,l}$ must lie in $50\rootn w/K$ neighbourhood of the corresponding tunnels. 

To get $(1)$ of Lemma \ref{lem: extra assumption lemma 2} we will choose $K$ depending on $A$ and $n$.  Namely,  we choose $K$ so that 
\begin{align*}
C_K = e^{c'\log K/\log \log K} > A
\end{align*}
Then  condition \ref{co1} translates to condition (1) in Lemma \ref{lem: extra assumption lemma 2}.    We also assumed many times that $K\gg_n 1$ which follows from the assumption $A \gg_n 1$, in Lemma \ref{lem: extra assumption lemma 2}. Finally we make our choice of $K$ so that

$$\log A \asymp_n \log K/\log \log K \hspace{5mm} \text{and} \hspace{5mm} 	C_K = e^{c'\log K/\log \log K} > A .$$

Now the constants $w$ and $K$ are chosen, and we can specify the constant $C(A)>1$ in Lemma \ref{lem: extra assumption lemma 2}. The condition $$N\gg K (\log K)^2$$ translates to 
$$N\gg C(A) \hspace{5mm} \text{with} \hspace{5mm} \log C(A) \asymp \log A \log\log A.$$

We now need to check $(2)$ of Lemma \ref{lem: extra assumption lemma 2} and we do this below.  Recall that for every $i\in S$, each $r_{i,k} \asymp w/K$ with $w\asymp (\log K)^{-1}$. Hence from \eqref{eq: 5.1} and the fact that $|S| \gtrsim_n K^{n-1}$, we have
\begin{align}\label{eq: 9.5}
\sum_{i\in S} \sum_k \N ( \tfrac{1}{2}B_{i,k} )~ r_{i,k}^{n-1} \gtrsim_n |S| \lb \frac{w}{K}\rb^{n-1} N \gtrsim_n w^{n-1} N \asymp_n  \frac{N}{(\log K)^{n-1}}.
\end{align} 
Since $\log A \asymp_n \log K/\log \log K$, we have $\log K \lesssim_n  \log A \cdot \log \log A$, for all large enough $K$. Hence we get the following lower bound from \eqref{eq: 9.5}:
\begin{align*}
\sum_{i\in S} \sum_k \N (\tfrac{1}{2}B_{i,k})~ r_{i,k}^{n-1} \gtrsim_n \frac{N}{(\log A \cdot \log \log A)^{n-1}}.
\end{align*}
as required. This completes the proof of Lemma \ref{lem: extra assumption lemma 2}.

\appendix
\section{proof of classical results}
\label{sec: appendix}
In this first part of this section, we prove several known results  about the frequency function which was presented in \S\ref{sec3}. 
Recall that the frequency  of a harmonic function $u$ in $\mathbb{R}^n$ is defined by
$$
\beta(x,r)= \frac{r G(x,r)}{H(x,r)},
$$
where $x\in \mathbb{R}^n$, $r>0$ and 
$$
H(x,r)= \int_{\partial B(x,r)} u^2 dS, \qquad G(x,r)= \int_{B(x,r)} |\nabla u|^2 d\vol.
$$ 
The following well-known identity will be useful in proving that $\beta(r):=\beta(0,r)$ is a non-decreasing function of $r$, see \cite{GL86}.
\begin{lem}
	\label{lem: proof of monotonicity}
	With the above notation, we have 
	\[
	\int_{\partial (rB)} |\nabla u|^2 = \frac{n-2}{r} \int_{rB} |\nabla u|^2d \vol + 2\int_{\partial (rB)}(\nabla u\cdot \hat{n})^2 dS.
	\]
	where  $\hat{n}$ is then unit normal vector to the sphere $\partial (rB)$ and $B=B(0,1)$.
\end{lem}
\begin{proof}[Proof of Lemma \ref{lem: proof of monotonicity}]
	Writing $1= \hat{r} \cdot \hat{n} / r$, where  $\hat{n}= \hat{r}r^{-1}$ is the unit normal to the sphere $\partial(rB)$, we have
	\[
	\int_{\partial (rB)} |\nabla u|^2= r^{-1}\int_{\partial (rB)} |\nabla u|^2 \hat{r}\cdot \hat{n} dS= r^{-1}\int_{rB} \Div \left( |\nabla u|^2 \cdot \hat{r}\right) d\vol,
	\]
	where the second equality follows by the divergence theorem. Expanding the divergence operator and using the fact that $\Div (\hat{r})=n$, we obtain 
	\begin{align}
		\label{eq: appendix 1} \int_{rB} \Div \left( |\nabla u|^2 \cdot \hat{r}\right) d\vol&= n\int_{rB} |\nabla u|^2 d\vol + 2\int_{rB} (\textbf{H} (u)  \hat{r})\cdot \nabla u  \hspace{2mm}d\vol \nonumber \\
		&= n G(r) +  2\int_{rB} (\textbf{H} (u)  \hat{r})\cdot \nabla u  \hspace{2mm}d\vol,
	\end{align}
	where $\textbf{H}(u)$ is the Hessian of $u$. Using the identity 
	\[
	\nabla(\nabla u \cdot \hat{r}) = \textbf{H}(u) \hat{r} + \nabla u,
	\]
	the second term in \eqref{eq: appendix 1} can be re-written as 
	\begin{align}
		\label{eq: appendix 2}	\int_{rB} (\textbf{H} (u)  \hat{r})\cdot \nabla u  \hspace{2mm}d\vol &= \int_{rB} \nabla (\nabla u\cdot \hat{r}) \cdot   \nabla u  \hspace{2mm}d\vol - \int_{rB} |\nabla u|^2 d\vol  \nonumber \\
		&= \int_{\partial (rB)} (\nabla u \cdot \hat{r}) (\nabla u \cdot \hat{n}) dS - \int_{rB} |\nabla u|^2 d\vol.
	\end{align}
	Therefore, inserting \eqref{eq: appendix 2} into \eqref{eq: appendix 1} and using the fact that $\hat{n}= \hat{r} r^{-1}$, we conclude 
	\[
	\int_{\partial (rB)} |\nabla u|^2=\frac{n-2}{r} G(r) + 2\int_{\partial (rB)}(\nabla u\cdot \hat{n})^2 dS,
	\]
	as required. 
\end{proof}

\begin{repthm}{thm: monotonicity frequency} [Monotonicity of the frequency]
	For any harmonic function in $\mathbb{R}^n$, its frequency function $\beta(r):=\beta(0,r)$ is a non-decreasing function of $r$. 
	
\end{repthm}

\begin{proof}[Proof of Theorem \ref{thm: monotonicity frequency}]  
	Taking the logarithmic derivative in the definition of the frequency function \eqref{def: frequency function} we have 
	\begin{align}
		\label{eq: derivative} \beta'(r)= \beta(r)\left(\frac{1}{r} + \frac{G'(r)}{G(r)} - \frac{H'(r)}{H(r)}\right).
	\end{align}
	We first  consider the term $H'/H$ on the r.h.s. of \eqref{eq: derivative}. Observe that 
	$$H(r)= \int_{\partial (rB)} |u|^2 dS= r^{n-1}\int_{\partial B} |u(r\cdot)|^2 dS.$$
	Thus, we have
	\begin{align}
		H'(r)&= (n-1)r^{n-2} \int_{\partial B} |u(r\cdot)|^2 dS + r^{n-1} \int_{\partial B}\partial_r |u(r\cdot)|^2 dS, \nonumber \\
		&= \frac{n-1}{r}H(r) + 2 \int_{\partial (rB)} u (  \nabla u \cdot \hat{n} )dS, \nonumber 
	\end{align}
	which gives
	\begin{align}\label{eq: expression H'/H}
		\frac{	H'(r)}{H(r)}=  \frac{n-1}{r} + 2 \frac{\int_{\partial (rB)} u (  \nabla u \cdot \hat{n} )dS}{\int_{\partial (rB)} u^2dS}.
	\end{align}
	
	\vspace{2mm}
	
	We now compute the term $G'/G$ on the r.h.s. of \eqref{eq: derivative}.  Differentiating $G(r)$, we find 
	$$G'(r)=\int_{\partial (rB)} |\nabla u|^2 dS.$$
	Therefore, Lemma \ref{lem: proof of monotonicity} gives
	\[
	G'(r)= \frac{n-2}{r}G(r) + \int_{\partial (rB)}( \nabla u \cdot \hat{n} )^2dS.
	\] 
	Since, for a non-constant harmonic function $u$,
	$$G(r)= \int_{B(x,r)} |\nabla u|^2 d\vol = \int_{\partial (rB)} u( \nabla u \cdot \hat{n} )dS>0,$$
	we deduce
	\begin{align}
		\label{eq: expression for D'/D}
		\frac{G'(r)}{G(r)}= \frac{n-2}{r} + 2 \frac{\int_{\partial (rB)} ( \nabla u \cdot \hat{n} )^2dS}{\int_{\partial (rB)} u( \nabla u \cdot \hat{n} )dS}.
	\end{align} 	
	Finally,	inserting \eqref{eq: expression for D'/D} and \eqref{eq: expression H'/H} into \eqref{eq: derivative}, we obtain 
	
	$$ \beta'(r)= 2\beta(r)\left(\frac{\int_{\partial (rB)} (  u \cdot \hat{n} )^2dS}{\int_{\partial (rB)} u( \nabla u \cdot \hat{n} )dS}- \frac{\int_{\partial (rB)} u ( \nabla u \cdot \hat{n} )dS}{\int_{\partial (rB)} u^2dS}\right),$$
	which is non-negative by the Cauchy-Schwartz inequality. 
\end{proof}

\begin{repcor}{cor: formulas for the growth} For any harmonic function in $\mathbb{R}^n$, let $\beta(r)= \beta(0,r)$ be its frequency function. For any $0<r_1<r_2$,  we have 
	$$\frac{H(r_2)}{r_2^{n-1}}= \frac{H(r_1)}{r_1^{n-1}}\exp \left( 2 \int_{r_1}^{r_2}\frac{\beta(r)}{r}dr\right).$$
	Moreover, we also have 
	$$\left(\frac{r_2}{r_1}\right)^{2\beta(r_1) + n-1}\leq \frac{H(r_2)}{H(r_1)}\leq \left(\frac{r_2}{r_1}\right)^{2\beta(r_2) + n-1}.$$
	
\end{repcor}

\begin{proof}[Proof of Corollary \ref{cor: formulas for the growth}]
	Since we may write
	$$G(r)= \int_{\partial (rB)} u (\nabla u\cdot \hat{n})dS,$$ formula \eqref{eq: expression H'/H} implies the following expression for the frequency function 
	$$ 2\frac{\beta(r)}{r} = \frac{2G(r)}{H(r)}= \frac{H'(r)}{H(r)}- \frac{n-1}{r}$$
	Integrating both sides, we obtain the first claim of Corollary \ref{cor: formulas for the growth}. Since, by Theorem \ref{thm: monotonicity frequency}, we have $\beta(r_1)\leq \beta(r)\leq \beta(r_2)$ for all $r\leq r_2$, the second claim in Theorem \ref{thm: monotonicity frequency} follows from the first. 
\end{proof}

Thanks to the Poisson formula for harmonic functions in $\mathbb{R}^n$, it is possible to bound the $L^{\infty}$ norm of a harmonic function in a ball by its $L^2$ norm on the boundary of a slightly larger ball.
\begin{lem}[Elliptic estimates]
	 Let $\delta>0$ and let $B\subset \mathbb{R}^n$ be a unit ball. There exists some constant $C=C(n,\delta)>1$ such that 
	$$\sup_{B}|u|^2 \leq C\int_{\partial((1+\delta)B)} |u|^2 dS,$$
	for any harmonic function $u$ defined in a neighbourhood of $(1+\delta)B$.  
\end{lem}

\begin{replem}{lem: comparability di and frequency}  Let $\delta>0$ be a sufficiently small parameter and let $\mathcal{N}(r):= \mathcal{N}(B(0,r))$, for $r>0$. There exists a constant $C=C(\delta,n)\geq 1$ such that 
	$$\beta (r(1+\delta))(1-100\delta) -C\leq \mathcal{N}(r) \leq \beta (2r(1+\delta))(1+100\delta) +C$$
	uniformly	for all $0<r\leq 4^{-1}$ 	and all harmonic functions $u$ defined in $B(0,2)$. 
\end{replem}

\begin{proof}[Proof of Lemma \ref{lem: comparability di and frequency}]
	We first establish the lower bound in Lemma \ref{lem: comparability di and frequency}. Let us write $B=B(0,1)\subset \mathbb{R}^n$ and observe that, by  {elliptic estimates}, we have 
	$$\sup_{rB}|u|^2 \lesssim_{\delta} \frac{1}{r^{n-1}} \int_{\partial (r(1+\delta)B)} |u|^2 dS \lesssim_{\delta}  \frac{H((1+\delta)r)}{r^{n-1}}.$$
	{We also have} 
	$$\sup_{2rB}|u|^2 \gtrsim \frac{1}{r^{n-1}} \int_{\partial (2rB)} |u|^2 dS\gtrsim \frac{H(2r)}{r^{n-1}}.$$
	Therefore, 
	$$2\mathcal{N}(rB) = \log \frac{\sup_{2rB}|u|^2}{\sup_{rB}|u|^2}\geq \log\frac{H(2r)}{H((1+\delta)r)} + C_0,$$
	for some $C_0=C_0(n,\delta)\geq 1$. Using the second part of Corollary \ref{cor: formulas for the growth} and taking $\delta>0$ sufficiently small (so that $\log_2(2/(1+\delta))\geq 1-100\delta$), we obtain 
	$$\log\frac{H(2r)}{H((1+\delta)r)}\geq 2\beta (r(1+\delta))(1-100\delta) + C_1,$$
	for some $C_1=C_1(n,\delta)\geq 1$, and the lower bound in Lemma \ref{lem: comparability di and frequency} follows. 
	
	\vspace{2mm}
	
	We are now going to establish the upper bound in Lemma \ref{lem: comparability di and frequency}. By elliptic estimates, we have 
	$$\sup_{2rB}|u|^2 \lesssim_{\delta} \frac{1}{r^{n-1}} \int_{\partial (2r(1+\delta)B)} |u|^2 dS \lesssim_{\delta}  \frac{H(2(1+\delta)r)}{r^{n-1}},$$
	and we also have
	$$\sup_{rB}|u|^2 \geq \frac{1}{r^{n-1}} \int_{\partial (rB)} |u|^2 dS= \frac{H(r)}{r^{n-1}}.$$
	Therefore, 
	$$2\mathcal{N}(rB) = \log \frac{\sup_{2rB}|u|^2}{\sup_{rB}|u|^2}\leq \log\frac{H(2r(1+\delta))}{H(r)} + C_2,$$
	for some $C_2=C_2(n,\delta)>0$. By Corollary \ref{cor: formulas for the growth} and taking $\delta>0$ sufficiently small (so that $\log_2(2/(1+\delta))\leq 1+100\delta$), we have 
	$$\log\frac{H(2r)}{H((1+\delta)r)}\leq 2\beta (2r(1+\delta))(1+100\delta) + C_3,$$
	for some $C_3=C_3(n,\delta)>0$, and the upper bound in Lemma \ref{lem: comparability di and frequency} follows.
\end{proof}

\begin{replem}{lem: monotonicity doubling index}[Almost monotonicity of the doubling index]
	Let $\delta>0$ be a sufficiently small parameter and let $u$ be a harmonic function in $B(0,2)$. {Let $\mathcal{N}(r):= \mathcal{N}(B(0,r))$, for $r>0$.} There exists a constant $C=C(\delta,n)>0$ such that 
$$\left(\frac{r_2}{r_1}\right)^{\mathcal{N}(r_1)(1-\delta)-C}\leq \frac{\sup \limits_{B(0,r_2)}|u|}{\sup \limits_{B(0,r_1)}|u|}\leq \left(\frac{r_2}{r_1}\right)^{\mathcal{N}(r_2)(1+\delta)+C},$$
for all $r_1\leq r_2/2\leq 4^{-1}$. In particular, we have 
$$\mathcal{N}(r_1)(1-\delta) - C\leq \mathcal{N}(r_2)(1+\delta) + C.$$
\end{replem}

\begin{proof}[Proof of Lemma \ref{lem: monotonicity doubling index}]
	Let $B=B(0,1)\subset \mathbb{R}^n$.	We will first prove the upper bound in Lemma \ref{lem: monotonicity doubling index}. Let $\delta_1>0$ be some, sufficiently small, parameter to be chosen in terms of $\delta>0$. By elliptic estimates, we have 
	$$\sup_{r_2B}|u|^2\lesssim_{\delta_1} \frac{H(r_2(1+\delta_1))}{r_2^{n-1}}.$$
	By Corollary \ref{cor: formulas for the growth}, we have 
	$$\frac{H(r_2(1+\delta_1))}{r_2^{n-1}} \lesssim_{\delta_1} \frac{H(r_1)}{r_1^{n-1}} \left(\frac{r_2(1+\delta_1)}{r_1}\right)^{2\beta(r_2(1+\delta_1)) +n-1}.$$
	By Lemma \ref{lem: comparability di and frequency}, we have 
	$$\beta(r_2(1+\delta_1)) \leq \frac{\mathcal{N}(r_2)}{(1+100\delta_1)} +C_1\leq \mathcal{N}(r_2)(1+200\delta_1) + C_1 ,$$
	for some $C_1=C(\delta_1,n)\geq 1$. Therefore, have shown that 
	$$\sup_{r_2B}|u|^2\lesssim_{\delta_1} \frac{H(r_1)}{r_1^{n-1}} \left(\frac{r_2(1+\delta_1)}{r_1}\right)^{2\mathcal{N}(r_2)(1+200\delta_1) + C_1}.$$
	Choosing $\delta_1=\delta/1000$ (so that $(1+200\delta_1)\log(1+\delta_1)\leq \delta/10$, say), we deduce
	$$\sup_{r_2B}|u|^2\leq \frac{H(r_1)}{r_1^{n-1}}\left(\frac{r_2}{r_1}\right)^{2\mathcal{N}(r_2)(1+\delta) + C_2},$$
where we have absorbed the constant in the $\lesssim$ notation into $C_2$. Hence, the upper bound in Lemma \ref{lem: monotonicity doubling index} follows upon noticing that 
	$$\frac{H(r_1)}{r_1^{n-1}}\leq \sup_{r_1B}|u|^2,$$
	and taking the square-root. 
	
	\vspace{2mm}
	
	We are now going to prove the lower bound in Lemma \ref{lem: monotonicity doubling index}. First, we observe that we may assume $r_2\geq 2(1+\delta)r_1$. Indeed, if $2r_1\leq r_2\leq 2(1+\delta)r_1$, then 
	$$ \left(\frac{r_2}{r_1}\right)^{\mathcal{N}(r_1)(1-\delta)}\leq (2(1+\delta))^{\mathcal{N}(r_1)(1-\delta)}\leq 2^{\mathcal{N}(r_1)},$$
	for all $\delta>0$ sufficiently small so that $(2(1+\delta))^{1-\delta}\leq 2$. Therefore, in this case, we have 
	$$\sup_{r_2B}|u|\geq \sup_{2r_1B}|u|= 2^{\mathcal{N}(r_1)}\sup_{r_1B}|u|\geq \left(\frac{r_2}{r_1}\right)^{\mathcal{N}(r_1)(1-\delta)} \sup_{r_1B}|u|. $$

	From now we will assume that $r_2\geq 2(1+\delta)r_1$. Let $\delta_1 \in (0,\delta/2)$ be a parameter to be chosen later in terms of $\delta>0$. By Corollary \ref{cor: formulas for the growth}, we have 
	$$\sup_{r_2B}|u|^2\geq \frac{H(r_2)}{r_2^{n-1}} \geq \frac{H(2r_1(1+\delta_1))}{(2r_1(1+\delta_1))^{n-1}} \left(\frac{r_2}{2r_1(1+\delta_1)}\right)^{2\beta(2r_1(1+\delta_1))},$$
	and, by Lemma  \ref{lem: comparability di and frequency}, we also have 
	$$ \beta(2r_1(1+\delta_1))\geq \frac{\mathcal{N}(r_1)}{1+100\delta_1} - C_3,$$
	for some $C_3=C_3(\delta_1,n)\geq 1$. Moreover, by elliptic estimates, we have 
	$$\frac{H(2r_1(1+\delta_1))}{(2r_1(1+\delta_1))^{n-1}} \gtrsim_{\delta_1} \sup_{2r_1B}|u|^2\gtrsim_{\delta_1} 2^{2\mathcal{N}(r_1)} \sup_{r_1B}|u|^2.$$
	All in all, we have shown that 
	$$ \sup_{r_2B}|u|^2 \gtrsim_{\delta_1} \sup_{r_1B} |u|^2 \left(\frac{r_2}{2r_1(1+\delta_1)}\right)^{2\mathcal{N}(r_1)(1+100\delta_1)^{-1} -C_3}2^{2\mathcal{N}(r_1)} .$$
	Taking the square root and choosing $\delta_1=\delta/10000>0$ (so that $(1+100\delta_1)^{-1}\geq 1-\delta/10$), we conclude that 
	$$ \sup_{r_2B}|u| \geq  \sup_{r_1B} |u| \left(\frac{r_2}{r_1}\right)^{\mathcal{N}(r_1)(1-\delta) -C_4},$$
	where we have absorbed the constant implied in the $\lesssim$ notation into $C_4$,	as required. 
\end{proof}

\begin{claim}[Gradient estimates for harmonic functions] \label{claim:grad_estimates} Let $B(p,r)\subset \real^n$ be any ball, and let $u$ be a harmonic function defined in a neighbourhood of its closure. Then
	\begin{align}
		\label{der_esti}
		| \nabla u (p)| \lesssim_n \frac{1}{r} \sup_{B(p,r)}|u|.
	\end{align}
\end{claim}
\begin{proof}[Proof of Claim \ref{claim:grad_estimates}]
	Since the partial derivatives of $u$ are also harmonic functions, the mean value theorem and the divergence theorem imply that 
	$$\nabla u (p) = \frac{1}{\vol B(p, r)}\int_{B(p,r)} \nabla u \hspace{1mm} d\vol = \frac{1}{\vol B(p, r)} \int_{\partial B(p,r)} u(x)\widehat{n}_x dS(x),\nonumber $$
	where $\widehat{n}_x$ is the exterior normal vector to $\partial B(p,r)$ at $x \in \partial B(p,r)$; and $dS$ denotes the surface measure on $\partial B(p,r)$. Hence
	\begin{align*}
		| \nabla u (p)| \lesssim_n \frac{1}{r} \sup_{B(p,r)}|u|,
	\end{align*}
	as claimed.
\end{proof}	

\begin{claim}
	\label{claim: lower bound doubling index}
	Let $B\subset \mathbb{R}^n$ be the unit ball, and let $u$ be a non-zero harmonic function on $2B$ such that $u(0)=0$. Then, there exists some constant $c=c(n)>0$, independent of $u$, such that 
	$$\mathcal{N}(B)\geq c.$$ 
\end{claim}
\begin{proof}
	Let $m_1^+:=\sup_B u$, $m_2^+:= \sup_{2B} u$ and   $m_1^-:=-\inf_B u$, $m_2^-:= -\inf_{2B} u$ and observe that, since $u(0)=0$, the Mean Value Theorem implies 
	$$ m_1^+,m_2^+,m_1^-,m_2^->0.$$
	Now, assume that $m_1^+=\sup_{B}|u|$ and consider the harmonic function $h:= m_2^+-u$. Then $h\geq 0$ in $2B$ and Harnack's inequality implies
	$$\sup_B  h\leq C \inf_B h.$$
	Let $x\in B$ be the point such that $u(x)= m_1^+=\sup_{B}|u|$. We obtain 
	$$\frac{1}{C}h(0)\leq h(x) \leq C h(0),$$
	that is 
	$$C^{-1} m_2^+\leq m_2^{+}- m_1^+\leq C m_2^+,$$
	which, upon rearranging, gives 
	$$m_1^+\leq m_2^+ \left(1- \frac{1}{C}\right).$$
	All in all, we have shown that 
	$$\left(1- \frac{1}{C}\right)^{-1}\leq \frac{m_2^+}{m_1^+}\leq \frac{\sup_{2B}|u|}{\sup_B|u|},$$
	as required. If  $-m_1^-=\sup_{B}|u|$, we apply the above argument to the harmonic function $-u$. This concludes the proof of Claim \ref{claim: lower bound doubling index}. 
\end{proof}

\begin{claim}
	\label{lem: monotonicity for cubes}
	Let $Q\subset \mathbb{R}^n$ be a unit cube. There exists some  $\Gamma=\Gamma(n)\geq 1$ such that for any  harmonic function $u$ in $\Gamma Q$, we have 
	$$\mathcal{N}^{*}\left(Q\right)\lesssim_{n} \log \frac{\int_{\Gamma Q}|u|^2 d\vol}{\int_{Q}|u|^2 d \vol}.$$
\end{claim}

\begin{proof}
	Up to a translation, we may assume that $Q$ is centered at zero. By the monotonicity of the doubling index for non-concentric balls, Lemma \ref{lem: almost monotonicity}, it is enough to prove that 
	$$\mathcal{N}(4\sqrt{n}B)\lesssim_n  \log \frac{\int_{\Gamma Q}|u|^2 d\vol}{\int_{Q}|u|^2 d \vol},$$
	where $B=B(0,1)$. To see this, the elliptic estimates imply 
	$$\sup_{8\sqrt{n}B}|u| \lesssim \frac{1}{\Gamma^n} \int_{\Gamma Q}|u|^2 d\vol,$$
	for all sufficiently large $\Gamma=\Gamma(n)$. On the other hand, we have 
	$$\sup_{4\sqrt{n}B}|u|\gtrsim_n \int_{Q}|u|^2 d \vol,$$
	and Claim \ref{lem: monotonicity for cubes} follows.  
\end{proof}

Most of the remaining section is devoted to proving the following lemma about almost monotonicity of the doubling index for non-concentric balls.

\begin{replem}{lem: almost monotonicity} [Almost monotonicity for non-concentric balls]
	Let $B\subset \mathbb{R}^n$ be any ball and let $u$ be a harmonic function in a neighbourhood of $\overline{2B}$. There exists a constant $C=C(n)>1$ such that 
	$$\mathcal{N}(b) \leq C \mathcal{N}(B)$$
	for all balls $b$ with  $2b\subset B$.  
\end{replem}
We will prove Lemma \ref{lem: almost monotonicity} in several steps. The main step is to prove the following special case of the lemma when $b$ is concentric to $B$. 
\begin{claim}
	\label{claim: lower bound small di}
	Let $B=B(0,1) \subset \real^n$ be the unit ball and let $u$ be a harmonic function in a neighbourhood of  $\overline{2B}$. There exists a constant $C=C(n)>1$ such that  the following holds for every $r \in (0,\thalf]$:
	$$\mathcal{N}\left(rB\right)\leq C\mathcal{N}(B).$$
\end{claim}

To establish Claim \ref{claim: lower bound small di}, we need several auxiliary results and we prove them now.  The following claim proves Claim \ref{claim: lower bound small di} in two cases: when either $\N(B)$ or $r$ is bounded away from $0$.

\begin{claim}\label{claim:part_one}	Let $B=B(0,1) \subset \real^n$ be the unit ball and let $u$ be a harmonic function in a neighbourhood of  $\overline{2B}$. Then the following hold.
	\begin{enumerate}
		\item If $\N(B) \geq c >0$, then 
		\begin{align*}
			\mathcal{N}\left(rB\right) \lesssim_{n,c} \mathcal{N}(B), \qquad~\forall r \in (0,\thalf].
		\end{align*} 
		\item For every $t \in (0,\thalf)$, there is a dimensional constant $C_{t} >0$ such that 
		\begin{align}\label{eq:c01}
			\N(rB) \leq C_{t} \cdot \N(B),\qquad ~\forall r \in [t, \thalf]. 
		\end{align}
	\end{enumerate}
\end{claim}
\begin{proof}[Proof of Claim \ref{claim:part_one}] We first show the first part of the claim. 	Recall from Lemma \ref{lem: monotonicity doubling index} that
	\begin{align}\label{eq:c02}
		\mathcal{N}\left(rB\right) \leq 2\mathcal{N}(B) + C,~\forall r \in (0,\thalf]
	\end{align}
	for some constant $C= C(n)>1$. If $\N(B) \geq c >0$, then it follows from \eqref{eq:c02} that
	\begin{align*}
		\mathcal{N}\left(rB\right) \leq (2 + \tfrac{C}{c}) \cdot \mathcal{N}(B),~\forall r \in (0,\thalf],
	\end{align*}
	and we get the desired conclusion. 
	\vskip 2mm
	We now prove the second part of the claim. If $\N(B) \geq C$, then  \eqref{eq:c02} implies
	\begin{align*}
		\N(rB) \leq 3 \N(B),~\forall r \in (0,\thalf].
	\end{align*}
	Hence it suffices to prove \eqref{eq:c01} in the case when $\N( B) <C$, and we do this below. We fix $t \in (0,\thalf)$, and let $r \in  [t, \thalf]$. Again it follows from \eqref{eq:c02} that if $\N( B) <C$, then $\N(rB) \leq 3C$. 
	Multiplying $u$ by a constant, we may assume that 
	$$\sup_{rB}|u|=1, ~\text{ and }~ \sup_{B}|u|=\sup_B u =: M.$$
	Let $p\in \partial B$ be a point such that $u(p)=M$, and put $M'= \sup_{2B}|u|$. With the above notation, since $2r\leq 1$, one has
	\begin{align*}
		\N(rB) \leq \log M	,~\text{ and }~\N(B) = \log \frac{M'}{M}. 
	\end{align*}
	Finally, define $$D:= B(p,1),~ D_1:= D\cap rB,~ D_2:= D\cap (B\backslash rB), \text{ and } D_3:= D\cap (2B\backslash B).$$ By the mean value theorem for harmonic functions, we have 
	\begin{align}
		M=u(p)= \frac{1}{\vol (D)}\int_{D} u \hspace{1mm} d\vol \leq \frac{1}{\vol(D)}\left(\vol(D_1) + M \vol(D_2) +  M'\vol(D_3) \right). \nonumber
	\end{align} 
	Since $\vol(D) = \vol(D_1)  + \vol(D_2)+ \vol(D_3)$, we obtain  
	$$ (M-1)\vol (D_1)\leq (M'-M)\vol(D_3),$$
	that is 
	\begin{align*}
		M +(M-1)\frac{\vol(D_1)}{\vol(D_3)}\leq M'.
	\end{align*}
	Let $c= c(t)= \vol(D \cap t B)/\vol(D_3)$, since $(D \cap t B) \subseteq D_1$, we have 
	\begin{align}
		\label{eq: bound}	1 + c\frac{M-1}{M}\leq \frac{M'}{M}.
	\end{align}
	To conclude \eqref{eq:c01}, we consider two cases: the first being when $M$ is very close to $1$, and when $M$ is not very close to $1$. 
	In the first case $M= 1+\varepsilon$, for some small $\varepsilon\leq \varepsilon_0$. Then
	$$\mathcal{N}\left(rB\right) \leq \log M =  \log (1+\varepsilon)\asymp \varepsilon.$$
	And from \eqref{eq: bound} we get 
	$$\mathcal{N}\left(B\right)\geq \log \left(1+\frac{c\varepsilon}{1+\varepsilon}\right)\asymp c\varepsilon.$$
	Thus \eqref{eq:c01} follows. Let us consider the second case: $M\geq 1+\varepsilon_0$. Then \eqref{eq: bound} implies that there exists some constant $c_1=c_1(n,\varepsilon_0)>0$ such that 
	$$\mathcal{N}\left(B\right)\geq c_1,$$
	which implies \eqref{eq:c01}, in light of the fact that $\mathcal{N}(rB)\leq 3C$. This completes the proof of Claim \ref{claim:part_one}. 
\end{proof}

The following claim gives a comparison of the doubling indices in two balls when their radii are much larger than the distance between their centers.  
\begin{claim} \label{claim:di_comparison} Let $B=B(0,1) \subset \real^n$ be the unit ball, and let $u$ be a harmonic function in a neighbourhood of  $\overline{2B}$. 
	Let $r >0$, and $p,q \in B$ such that $|p-q|< r$. Also assume that $8r <1$, then 
	\begin{align*}
		\N(p,2r) \lesssim_n \N(q,4r).
	\end{align*}
\end{claim}

\begin{proof}[Proof of Claim \ref{claim:di_comparison}]Observe that 
	\begin{align*}
		B(q,r) \subset B(p,2r) \subset B(p,4r) \subset B(q,8r). 
	\end{align*}
	Thus one has
	\begin{align*}
		\frac{\sup_{B(p,4r)}|u|}{\sup_{B(p,2r)}|u|} \leq \frac{\sup_{B(q,8r)}|u|}{\sup_{B(q,r)}|u|} = \frac{\sup_{B(q,8r)}|u|}{\sup_{B(q,4r)}|u|} \cdot \frac{\sup_{B(q,4r)}|u|}{\sup_{B(q,2r)}|u|} \cdot \frac{\sup_{B(q,2r)}|u|}{\sup_{B(q,r)}|u|},
	\end{align*}
	from which one gets
	\begin{align*}
		\N(p,2r) \leq \N(q,4r) + \N(q,2r) + \N(q,r). 
	\end{align*}
	It follows from part two of Claim \ref{claim:part_one} that both $\N(q,r)$, $\N(q,2r) \lesssim_n \N(q,4r)$. So we get the desired conclusion $\N(p,2r) \lesssim_n \N(q,4r)$. 
\end{proof}

With  Claim \ref{claim: lower bound doubling index} and Claim \ref{claim:di_comparison} in hand, we conclude in the following claim that if the doubling index of a harmonic function in a ball is small enough, then it is non-vanishing in a smaller concentric ball.

\begin{claim}\label{newest_claim} Let $B=B(0,1) \subset \real^n$ be the unit ball, and let $u$ be a harmonic function in a neighbourhood of  $\overline{2B}$. If $\N(\tfrac{1}{4}B) \ll_n 1$, then $u$ is non-vanishing in $\tfrac{1}{16}B$. 
\end{claim}
\begin{proof}[Proof of Claim \ref{newest_claim}] Assume that there is a point $p \in \tfrac{1}{16}B$ where $u$ vanishes. Then it follows from  Claim \ref{claim: lower bound doubling index} that $\N(p,\tfrac{1}{8}) \geq c$, where $c>0$ is as  in Claim \ref{claim: lower bound doubling index}. Using Claim \ref{claim:di_comparison} for $u$ with $q=0$, $p$ as above, and $r= 1/16$, one can conclude that 
	$$c \leq \N(p,\tfrac{1}{8}) \lesssim_n  \N(0,\tfrac{1}{4})=\N(\tfrac{1}{4}B).$$
	Hence if $u$ vanishes somewhere in $\tfrac{1}{16}B$, then $\N(\tfrac{1}{4}B) \gtrsim_n 1$. This proves Claim \ref{newest_claim}. 
\end{proof}

We now present one final auxiliary  result before proving Claim \ref{claim:part_one}. The following claim will be helpful for comparing doubling indices in two sub-balls when  the harmonic function has very small doubling index in a ball. 

\begin{claim}\label{claim:last_di}
	Let $B=B(0,1) \subset \real^n$ be the unit ball, and let $u$ be a positive harmonic function in a neighbourhood of $\overline{2B}$ with $u(0)=1$.  There exist small dimensional constants $\ell, c_0 >0$ such that {if $\N(B) \leq c_0$,  then $\N(B) \gtrsim_n \sup_{({\ell}/{2}) B} |u-1| $.} 
\end{claim}

\begin{proof}[Proof of Claim \ref{claim:last_di}] 
	We start by choosing $\ell$. Claim \ref{claim: lower bound doubling index} and Claim \ref{claim:harnack_consequence} guarantee the existence of constants $C>1$ and $c>0$ respectively such that the following holds. For any non-constant harmonic function $f$ in a neighbourhood of $\overline{2B}$ vanishing at its center, one has
	\begin{align}\label{eq_a_1}
		\N_{f}(B) \geq \log C, ~\text{ and }~ \sup_{2B}f \geq c \sup_{B}|f|. 
	\end{align}	
	{Fix a large $k \in \nat$ such that}  $c\cdot C^{k} \geq 4$. For this  $k$, let $\ell \in (0,1)$ be  such that $\tfrac{1}{8} \leq 2^k \ell \leq \tfrac{1}{4}$. The choice of  $c_0 \in (0,1)$ will be made later.

	For $u$  as in the statement of the claim, define $v:= u-u(0) = u-1$. Then $v$ is harmonic in $2B$ and  vanishes at its center. Using the first conclusion of \eqref{eq_a_1} repeatedly for $v$, one gets
	\begin{align*}
		\sup_{2^{k}\ell B} |v| \geq C^{k} \sup_{\ell B} |v|.
	\end{align*}
	This together with the second conclusion of \eqref{eq_a_1} gives
	\begin{align*}
		\sup_{2^{k+1}\ell B}v \geq c \cdot  \sup_{2^{k}\ell B} |v| \geq c\cdot C^{k} \sup_{\ell B} |v| \geq c\cdot C^{k} \sup_{\ell B} v. 
	\end{align*}
	By the choice of $k$, we have $c \cdot C^{k} \geq 4$, and hence
	\begin{align*}
		\sup_{2^{k+1}\ell B}v \geq 4 \sup_{\ell B} v.
	\end{align*}
	Hence
	\begin{align}\label{eq_a_2}
		\sup_{2^{k+1}\ell B}u \geq 1+ 4 \sup_{\ell B} v.
	\end{align}
	We now get a lower bound for the doubling index of $u$ in $B$. We have
	\begin{align*}
		\frac{1+ 4 \sup_{\ell B} v }{1 + \sup_{\ell B} v} \mathop{\leq}\limits_{\eqref{eq_a_2}}	\frac{\sup_{2^{k+1}\ell B} u }{\sup_{\ell B} u} & = \frac{\sup_{2^{k+1}\ell B} u }{\sup_{2^k\ell B} u} \cdot \frac{\sup_{2^{k}\ell B} u }{\sup_{2^{k-1}\ell B} u} \cdots  \frac{\sup_{2\ell B} u }{\sup_{\ell B} u}.
	\end{align*}
	Thus
	\begin{align}\label{eq_a_5}
		\log \frac{1+ 4 \sup_{\ell B} v }{1 + \sup_{\ell B} v} \leq \sum_{j=0}^{k} \N(2^j \ell B) \lesssim_{n}  \N(B),
	\end{align}
	where the last inequality holds because each term in the sum is $\lesssim \N(B)$ by the second part of Claim \ref{claim:part_one}. 
	
	Since $u$ is positive {in $2B$} and $u(0) = 1$, by Harnack's inequality one has $u \asymp_n 1$ in {$B$}. Hence there is a dimensional constant $c_1 >0$ such that 
	\begin{align}\label{eq_a_6}
		\frac{1+ 4 \sup_{\ell B} v }{1 + \sup_{\ell B} v}  = 1+ \frac{3\sup_{\ell B} v }{\sup_{\ell B} u}\geq 1 + c_1 \sup_{\ell B} v. 
	\end{align}
	Thus from \eqref{eq_a_5} and \eqref{eq_a_6}, one has
	\begin{align}\label{eq_a_7}
		\log (1 + c_1 \sup_{\ell B} v) \lesssim_n \N(B). 
	\end{align}
	Note that if $\N(B) \ll 1$, then 
	\begin{align*}
		1\gg  \N(B) \gtrsim_n \log (1 + c_1 \sup_{\ell B} v) \asymp \sup_{\ell B} v \mathop{\geq} \limits_{\eqref{eq_a_1}} c \sup_{(\ell/2) B}|v|.
	\end{align*}
	 Hence if $c_0 >0$ is small enough, then $\N(B) \leq c_0$ implies $\N(B) \gtrsim_{n} \sup_{({\ell}/{2}) B} |u-1|$.

\end{proof}

We are finally ready to prove Claim \ref{claim: lower bound small di}. 
\begin{proof}[Proof of Claim \ref{claim: lower bound small di}] 
	
	Recall that in Claim \ref{claim:part_one}, it was shown that $\N(rB) \lesssim \N(B)$ holds whenever either $r$ or $\N(B)$ is {separated} from $0$. 
	Hence it suffices to prove that there exist dimensional constants $r_0 \in (0,1)$ and $C_0 >1$ such that 
	\begin{align}\label{eq:c03}
		\N(rB) \leq C_0 \cdot \N(B),~\forall r \in (0,r_0).
	\end{align}
	Let $c_0,\ell >0$ be as in Claim \ref{claim:last_di}. Let $c_1 >0$ be a small dimensional constant, to be chosen later. 
	We now prove \eqref{eq:c03} for {$r_0 := \ell/2^{10}$}. If {$\N(\tfrac{1}{4}B) > c_1$}, then it follows from part two of Claim \ref{claim:part_one} that $\N(B) \gtrsim {\N(\tfrac{1}{4}B}) \geq  c_1$, and from part one of Claim \ref{claim:part_one} we get that 
	\begin{align*}
		\N(rB) \lesssim_{n} \N(B),~\forall r\in (0,\tfrac{1}{2}].
	\end{align*}
	Hence it suffices to consider the case when ${\N(\tfrac{1}{4}B)} \leq c_1$. Claim \ref{newest_claim} guarantees that if $c_1$ is small enough,  then $u$ is non-vanishing in ${\tfrac{1}{16}B}$. We will choose $c_1$ satisfying  this and an additional condition to be specified soon.
	\vskip 2mm
	Let $\ep >0$ and $2r\in (0,r_0)$ be such that $\N(rB) = \ep$. We will now show that { with the above assumption of $u$ being non-vanishing in $\tfrac{1}{16}B$, one has  $\N(\tfrac{1}{16}B) \gtrsim_{n} \ep/r$} and this will prove \eqref{eq:c03}. Let $M = \sup_{rB} u$, then $\sup_{2rB} u = 2^{\ep}M$. Hence there is a point $p \in 2rB {\subset \tfrac{\ell}{2^{10}}B}$ such that 
	\begin{align*}
		|\nabla u(p)| \gtrsim \frac{M(2^{\ep}-1)}{r} \gtrsim \frac{M\ep}{r}.
	\end{align*}
	Multiplying $u$ by a constant, we may assume that $u(p) = 1$, in which case by Harnack's inequality $M \asymp_n 1$. Hence we have the following bound for the gradient at $p$: 
	\begin{align*}
		|\nabla u(p)| \gtrsim \frac{\ep}{r}.
	\end{align*}
	From the standard gradient estimates (Claim \ref{claim:grad_estimates}), we get 
	\begin{align}\label{eq:condition1}
		\sup_{B(p,{\ell}/{2^6})} |u-1| \gtrsim |\nabla (u-1)(p)| = |\nabla u(p)| \gtrsim \frac{\ep}{r}.
	\end{align}
	Recalling that $\ell <1$, we have the following comparison for the doubling indices  from Claim \ref{claim:di_comparison} (applied twice):
	\begin{align*}
		\N(p,\tfrac{1}{32}) \lesssim \N(0,\tfrac{1}{4})=\N(\tfrac{1}{4}B) \leq c_1.
	\end{align*}
	We will choose $c_1$ small enough so that the above inequality implies $\N(p,\tfrac{1}{32}) \leq c_0$. This together with \eqref{eq:condition1} implies that we can use Claim \ref{claim:last_di} to conclude that $\N(p,\tfrac{1}{32}) \gtrsim \tfrac{\ep}{r}$. Once again using Claim \ref{claim:di_comparison}, one can conclude that 
	\begin{align*}
		\N(\tfrac{1}{16}B) = \N(0,\tfrac{1}{16}) \gtrsim \N(p,\tfrac{1}{32}) \gtrsim \frac{\ep}{r} \geq \ep = \N(rB).
	\end{align*}
	This completes the proof of Claim \ref{claim: lower bound small di}. 
\end{proof}

The following claim  gives a comparison between the  doubling indices of a sub-ball and a ball, when the center of the sub-ball is close to the center of the ball.  
\begin{claim}\label{cl_ap_1} Let $B=B(0,1) \subset \real^n$ be the unit ball and let $u$ be a harmonic function in $2B$. Then there is a constant $C=C(n)>1$ such that for every $z \in \tfrac{1}{2}\overline{B}$, the following holds:
	\begin{align}\label{eq_ap_part01}
		\N(z,1-\tfrac{|z|}{2}) \leq C \N(B). 
	\end{align}
	Also, for any ball $b$ which is centered at $z$ and  satisfies $2b \subset B$, we have 
	\begin{align}\label{eq_ap_part02}
		\N(b) \leq C \N(B). 
	\end{align}
\end{claim}

\begin{proof}[Proof of Claim \ref{cl_ap_1}]
	We first observe that $D = B(z,1-\tfrac{|z|}{2})$ is the largest ball centered at $z$ such that $2D \subset 2B$. Using $|z|<\tfrac{1}{2}$, one has
	\begin{align*}
		B(0,\tfrac{1}{4}) \subseteq B(z,1-\tfrac{|z|}{2}) \subseteq B(z,2-|z|) \subseteq B(0,2).
	\end{align*}
	Hence 
	\begin{align*}
		\frac{\sup_{B(z,2-|z|)}|u|}{\sup_{B  (z,1-({|z|}/{2})) }|u|} \leq \frac{\sup_{B(0,2)}|u|}{\sup_{B(0,\tfrac{1}{4})}|u|} = \frac{\sup_{B(0,2)}|u|}{\sup_{B(0,1)}|u|} \cdot \frac{\sup_{B(0,1)}|u|}{\sup_{B(0,\tfrac{1}{2})}|u|} \cdot  \frac{\sup_{B(0,\tfrac{1}{2})}|u|}{\sup_{B(0,\tfrac{1}{4})}|u|}.
	\end{align*}
	Therefore
	\begin{align*}
		\N(z,1-\tfrac{|z|}{2}) \leq \N(0,1) + \N(0,\tfrac{1}{2}) + \N(0,\tfrac{1}{4}) \lesssim \N(0,1),
	\end{align*}
	where the last inequality follows from Claim \ref{claim:part_one}, and this proves the first part \eqref{eq_ap_part01} of the claim.  
	
	Let $b=B(z,\ell)$. Since $2b \subset B$, one has 
	$$2\ell \leq 1-|z| \leq 1-\tfrac{|z|}{2},$$
	and hence $2b \subset D$.  Now using Claim \ref{claim: lower bound small di} and \eqref{eq_ap_part01}, one can conclude
	\begin{align*}
		\N(z,\ell) \lesssim \N(z,1-\tfrac{|z|}{2}) \lesssim \N(0,1),
	\end{align*}
	and this proves the last  part \eqref{eq_ap_part02} of the claim.

\end{proof}

We are finally ready to prove Lemma \ref{lem: almost monotonicity}. 

\begin{proof}[Proof of Lemma \ref{lem: almost monotonicity}]
Once again, we may assume w.l.o.g. that $B=B(0,1)$. Let $b = B(x,r)$ be such that $2b \subset B$. If $|x| \leq \thalf$, then \eqref{eq_ap_part02} in Claim \ref{cl_ap_1} immediately implies the conclusion of the lemma.

We now prove the lemma when $|x| \in (\thalf,1)$. Let $x = |x| \theta$, where $\theta \in \partial B(0,1)$. Let $C$ be as in Claim \ref{cl_ap_1}. By \eqref{eq_ap_part01} of Claim \ref{cl_ap_1} one has 
\begin{align}\label{eq:app102}
	\N(\tfrac{1}{2} \theta, \tfrac{3}{4}) \leq C \N(B).
\end{align}
and 
$$\mathcal{N}\left(\tfrac{7}{8} \theta, \tfrac{9}{16}\right)\leq 	C\N(\tfrac{1}{2} \theta, \tfrac{3}{4}) \leq C^2 \N(B).$$
Recall that $2b\subset B(x,1-|x|)$. We observe that, if $|x|\in [\tfrac{1}{2},\tfrac{7}{8}]$, then 
$$2b \subset  B(x,1-|x|)\subset B\left(\tfrac{1}{2} \theta, \tfrac{1}{2}\right)\subset B\left(\tfrac{1}{2} \theta, \tfrac{3}{4}\right),$$
and $x\in \tfrac{1}{2}B(\tfrac{1}{2} \theta, \tfrac{3}{4})$. Using \eqref{eq_ap_part02} of Claim \ref{cl_ap_1} for $B(\tfrac{1}{2} \theta, \tfrac{3}{4})$ gives 
$$\mathcal{N}(b)\leq  C	\N(\tfrac{1}{2} \theta, \tfrac{3}{4}) \leq C^2 \N(B).$$
If $x\in [\tfrac{7}{8},1]$,  as illustrated in Figure \ref{fig: proof Lemma 3.6}, we have 
$$2b \subset B(x,1-|x|)\subset B\left(\tfrac{7}{8} \theta, \tfrac{1}{8}\right)\subset B\left(\tfrac{7}{8} \theta, \tfrac{9}{16}\right),$$
and $x\in \tfrac{1}{2}B(\tfrac{7}{8} \theta, \tfrac{9}{16})$. Using \eqref{eq_ap_part02} of Claim \ref{cl_ap_1} for the ball $B\left(\tfrac{7}{8} \theta, \tfrac{9}{16}\right)$  gives 
$$\mathcal{N}(b)\leq  C	\N(\tfrac{7}{8} \theta, \tfrac{9}{16})\leq C^3 \mathcal{N}(B),$$
as required. 
\end{proof}
\begin{figure}[H]
	\includegraphics[scale=.9]{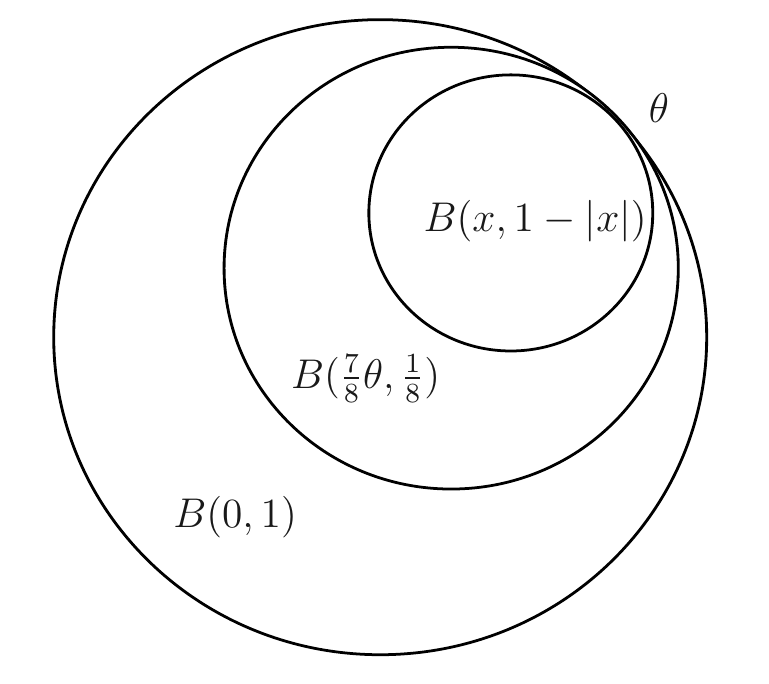}
	\centering
	\caption{Relative position of the balls in the proof of Lemma \ref{lem: almost monotonicity}}	\label{fig: proof Lemma 3.6}
\end{figure}

We end this section by showing how Proposition  \ref{prop: lemma 1} implies the following modified version of it. 
\begin{repcor}{cor_of_prop4.1} [A modified version of Proposition \ref{prop: lemma 1}] Let a unit cube $Q \subset \real^n$ be  partitioned into $K^n$ equal subcubes $\{q_{\alpha}: 1\leq \alpha \leq K^n\}$. Let $u$ be a harmonic function in $4\sqrt{n} Q$.  Let $\N^{*}(2Q) = N$. If $K \gg_n 1$, then we have
\begin{align*}
\#\{\alpha :\N^{*}(100\sqrt{n}q_{\alpha}) \geq \max(N\cdot 2^{-c\log K/\log \log K},C)\} \leq K^{n-1-c},
\end{align*}
where $c,C >0$ are dimensional constants. 
\end{repcor}

\begin{proof}[Proof of Corollary \ref{cor_of_prop4.1}] Partition $2Q$ into $L^n$ many equal subcubes $\{s_{j} : 1\leq j \leq L^n\}$, where 
\begin{align*}
L := \left \lfloor \frac{K}{100\sqrt{n}} \right\rfloor.
\end{align*}
Hence the side length of each $s_j$ is at least $200\sqrt{n}/K$. 
Let $c_0$ and $C_0$ denote the constants given by Proposition \ref{prop: lemma 1}. Using Proposition \ref{prop: lemma 1} for the cube $2Q$ with the above partition, one can conclude that if $L \gg_{n} 1$ (which is equivalent to $K \gg_{n} 1$, since $K \asymp_n L$), then
\begin{align}\label{concl}
\# \{1 \leq j \leq L^n: \N^*(s_j) > \max(N\exp(-c_0 \log L/\log \log L), C_0)\} \leq L^{n-1-c_0}. 
\end{align}
For  cubes $100\sqrt{n}q_{\alpha}$ and $s_j$, one has
\begin{align*}
\N^{*}(100\sqrt{n}q_{\alpha}) &= \sup \{\N(x,r): x\in 100\sqrt{n}q_{\alpha}\text{ and }r\leq 100n/K\},\\
\N^{*}(s_j) & \geq  \sup \{\N(x,r): x\in s_j \text{ and }r\leq 200n/K\}.
\end{align*}
It is easy to see that  every cube $100\sqrt{n}q_{\alpha}$ can be covered by cubes from $\{s_j\}_j$; and there is $M \in \nat$ such that for every $s_j$, there are at most $M$ many cubes in $\{100\sqrt{n}q_\alpha\}_{\alpha}$ which have non-trivial intersection with $s_j$. From the above discussion, it follows that $\N^*(100\sqrt{n}q_{\alpha})$ can be large only when $100\sqrt{n}q_{\alpha}$ non-trivially intersects $s_j$ for which $\N^{*}(s_j)$ is large. Hence we have
\begin{align*}
\# \{\alpha: \N^{*}(100\sqrt{n}&q_{\alpha}) \geq \max(N\exp(-c_0 \log L/\log \log L), C_0)\},\\
& \leq M \cdot  \# \{j : \N^*(s_j) \geq \max(N\exp(-c_0 \log L/\log \log L), C_0)\}.
\end{align*}
The desired conclusion now follows from \eqref{concl} since $K \asymp_n L$.

\end{proof}

\section{Few words about the proof of Proposition \ref{prop: lemma 1}}\label{sec: proof upper bound distribution}

\textbf{Maximal Doubling index for cubes.} For a given cube $Q$ in $\mathbb{R}^n$, we denote by $2Q$ the cube, which is the homothetic copy of $Q$ with same center and with homothety coefficient 2.  

In some of the multiscale arguments it is convenient to work with a maximal version $N^*_u(Q)$ of the doubling index:

$$ N^*_u(Q) = \sup\limits_{x \in Q, r \leq { \textup{diam}}(Q) } \log \frac{\max_{B(x, 2 r)}|u|}{\max_{B(x,r)}|u|} $$
because $N^*_u(Q)$ is a monotone function of a set.

\subsection{Informal guide to the proof of Proposition \ref{prop: lemma 1}}
In the proof of the lower effective bound in Nadriashvili's conjecture we used once Proposition \ref{prop: lemma 1} as a black box. Let us remind the formulation of the Proposition.

\begin{repprop}{prop: lemma 1}.
Let a cube $Q$ in $\mathbb{R}^n$ be partitioned into $A^n$ equal subcubes $Q_i$. 
 Let $u$ be a harmonic function in $2Q$. There exist numerical (depending only on the dimension $n$) constants $A_0,C > 1$ and $c>0$  such that  if $A>A_0$, then the number of $Q_i$ with $$N^*_u(Q_i)>\max(N^*_u(Q)/\exp(- c \log A/ \log \log A ), C)$$
 is smaller than $A^{n-1-c}$
 \end{repprop}

The goal of this section is to provide an informal guide for those who would like to understand the proof of the Proposition \ref{prop: lemma 1}. In order to understand the full details the reader may read pages 221-231 in \cite{Llemmas} and pages 249-254 in \cite{Lproof}, but there are only two facts about harmonic functions, which are used in the proof.
The first fact is the monotonicity property of the frequency, which is a powerful tool.
The second fact is quantitative Cauchy uniqueness property.

 The monotonicity property of the frequency and some considerations from Euclidean geometry imply the following  quasi-additive property of the  doubling index.
 
 \vspace{2mm}

\begin{lem}[Simplex lemma (simplified version)]
There exist a small number $c>0$ and and a large number $C>1$, depending on the dimension $n$, such that the following holds for any harmonic function $u$ in $\mathbb{R}^n$. Suppose that  $n+1$ points $x_1,x_2, ..., x_{n+1}$ in $\mathbb{R}^n$ form an equilateral simplex $S$ with side length $r$. 
Assume that $$\max_{1\leq i \leq n+1 }N_u(B(x_i,\rho)) \geq A,$$ where $A$ is any positive number and $\rho \leq r/2$.
Let $x$ be the center of $S$. 
Then $$ N_u(B(x,C r)) \geq A(1+c) - C.$$
\end{lem}
 
\begin{rem}
In the statement above one can replace $n+1$ balls with centers at the vertices of the equilateral simplex by any other $n+1$ balls, whose centers form a non-degenerate simplex and the radii are bounded by the diameter of the simplex. But the constants $c>0$ and $C>1$ depend on how degenerate the simplex is. We refer to the formal statement of simplex lemma in \cite{Llemmas}, where the constants depend on the ratio of the width of the simplex and its diameter.
\end{rem}

 The second fact about harmonic functions, which is used in the proof of Proposition \ref{prop: lemma 1}, is quantitative Cauchy uniqueness property.
 
 \vspace{2mm}

\begin{lem}[Quantitative Cauchy uniqueness property]
Consider a half-ball 
$$ B_+ = \{(x,y)\in\mathbb{R}^{n-1}\times\mathbb{R}: |x|^2+y^2< 1, y> 0 \} .$$

There exist  $\alpha \in(0,1)$ and $C>0$ such that if $h$ is harmonic in $B_+$, $h\in C^1(\overline{B}_+)$ and satisfies the inequalities $|h|\le 1, |\nabla h|\le 1$ in $B_+$ and $|h|\le\varepsilon$, $|\partial_y h|\le \varepsilon$ on $\{(x,y)\in \overline{B}_+, y=0\}$, $\varepsilon\le 1$, then 
$$ |h(x)|\le C\varepsilon^\alpha \quad{\text{when}}\quad x\in\frac{1}{3}B_+.$$
\end{lem}

 We refer to Appendix of \cite{LMNN21} for elementary proof. Quantitative Cauchy uniqueness property is used to prove another quasi-additive property of the frequency.
 
 \vspace{2mm}

\begin{lem}[Hyperplane lemma (simplified version)]
 There exists a large integer  $K>100$ and $C>1$ such that the following holds. Let $u$ be a harmonic function in $\mathbb{R}^3$. 
Consider a finite lattice of points $$L_K:=\{ (i,j,0): i=- K, \dots, K, j=-K, \dots, K \}.$$
If  $N(B(x,1)) \geq A$ for each $x\in L_K$ and some number $A>C$, then $$N(B(0,K))\geq 2A.$$
\end{lem}

\begin{rem}
In the statement of the hyperplane lemma 
 one can replace the finite lattice of balls $B(x,1)$, $x \in L_K $  by any collection of balls $B(p_x,r)$,$x \in L_K $, with the centers $p_x \in B(x,1)$ and any radius $r\leq 1$, and the statement will remain true.
 In the $n$-dimensional version of hyperplane lemma one should take $(n-1)$-dimensional lattice of balls.
 \end{rem}
 
 Hyperplane lemma and simplex lemma can be combined together to get the following useful fact about distribution of doubling indices.
 
\begin{thm}[Simplified version of Theorem 5.1 from \cite{Llemmas}]
  Let $u$ be a harmonic function in $\mathbb{R}^n$.   There exist numerical  (depending only on the dimension $n$) constants $A_0 \in \mathbb{N}$, $C > 1$ and $c_1>0$  such that  if we partition a cube $Q$ in $\mathbb{R}^n$  into $A_0^n$ equal subcubes $Q_i$, then the number of $Q_i$ with $$N^*_u(Q_i)>\max(N^*_u(Q)/(1+c_1), C)$$
 is smaller than $\frac{1}{2} A_0^{n-1}$.
\end{thm}

 In the proof Theorem 5.1 from \cite{Llemmas} the following intuitive principle is used. The simplex lemma implies that ``bad" $Q_i$ with $N^*_u(Q_i) > N^*_u(Q)/(1+c_1)$ should be close to some $(n-1)$-dimensional hyperplane, and then  one can find a finite ``lattice" of ``bad" cubes  and arrive to contradiction with the hyperplane lemma. We refer to \cite{Llemmas} for details.
 
 \vspace{2mm}

Proposition \ref{prop: lemma 1} follows from Theorem 5.1 from \cite{Llemmas} by iterative subdivision of $Q$ into smaller and smaller cubes. One can think that $A=A_0^k$,  and we subsequently chop cubes into $1/A_0$ smaller subcubes and each time we apply Theorem 5.1 to get more information on the distribution of doubling indices.
The formal iterative argument is given on pages 249-254 in \cite{Lproof}.

\section{Donnelly and Fefferman's bound for the doubling index}
\label{sec: Appendix DF}
The aim of this section is to show how Proposition \ref{lem: half of the doubling index} follows from the work of Donnelly and Fefferman on complex extensions of real-analytic functions, more precisely from \cite[Proposition 5.11]{DF88}. We now recall the statement of Proposition \ref{lem: half of the doubling index}:

\begin{repprop}{lem: half of the doubling index}
\textit{Let $Q$ be a unit cube in $\real^n$ and let $u$ be a harmonic function in $4\sqrt n \cdot Q$. Let $\varepsilon>0$ and 	let  $Q$  be partitioned into  equal subcubes $\{Q_i\}_i$ of side length  $\asymp_n\tfrac{c}{\N_u^*(Q)}$ for some sufficiently small $c=c(n,\varepsilon)>0$.  There exists a constant $C = C(n,\varepsilon) > 1$ such that 
	$$\N^*(Q_i)\leq C$$
for at least $(1-\varepsilon)$ portion of the $Q_i$. }
\end{repprop}

Proposition \ref{lem: half of the doubling index} follows with the help of the following fact about holomorphic function \cite[Proposition 5.11]{DF88}:
\begin{prop}
	\label{prop: DF}
	Let $k$ be a sufficiently large integer depending on $n$. Let $G(z)$ be holomorphic in $|z|\leq 3^k$, $z\in \mathbb{C}^n$, and satisfying 
	$$\max_{|z|\leq 2^k} |G(z)|\leq |G(0)| \exp (D),$$
	for some $D\geq 1$ sufficiently large depending on $n$. Assume that $G(z)$ is real and non-negative for real $x\in Q_0:=\{(x_1,...,x_n)\in \mathbb{R}^n: |x_i|\leq 1, \hspace{2mm} 1\leq i\leq n \}$. Suppose that $R\subset \tfrac{1}{2}Q_0$ is a sub-cube and subdivide $R$ into equal sub-cubes $R_i$ of side length $\asymp_n \frac{1}{D}$. Let $\varepsilon>0$ be given. Then there exists a subset $E\subset Q_0$ of measure $\vol(E)\leq \varepsilon$ and some constant $C=C(n,\varepsilon)\geq 1$ such that 
	\begin{align}
		\left|\log G(x)- \log \left(\frac{1}{\vol (R_i)}\int_{R_i} G \hspace{2mm}d\vol \right)\right|\leq C \hspace{5mm} \text{for all} \hspace{3mm} x\in R_i\backslash E. \nonumber
	\end{align}
\end{prop}

\begin{cor}
		\label{cor: DF}
		Let $k$ be a sufficiently large integer depending on $n$. Let $G(z)$ be holomorphic in $|z|\leq 3^k$, $z\in \mathbb{C}^n$, and satisfying 
		$$\max_{|z|\leq 2^k} |G(z)|\leq |G(0)| \exp (D),$$
		for some $D\geq 1$ sufficiently large depending on $n$. Assume that $G(z)$ is real and non-negative for real $x\in Q_0:=\{(x_1,...,x_n)\in \mathbb{R}^n: |x_i|\leq 1, \hspace{2mm} 1\leq i\leq n \}$. Let $\delta>0$, suppose that $R\subset \tfrac{1}{2}Q_0$ is a sub-cube and subdivide $R$ into equal sub-cubes $R_i$ of side length $\lesssim_n \frac{c}{D}$ for some sufficiently small $c=c(n,\delta)>0$. Then we have  
		\begin{align}
			\int_{R_i} G \hspace{2mm}d\vol \asymp_{\delta} \int_{2R_i} G \hspace{2mm}d\vol \label{eq: c2.1}
		\end{align}
		for $(1-\delta)$ portion of the $R_i$.
\end{cor}

\begin{proof}

We apply Proposition \ref{prop: DF} twice, for $D$ and $D'\gg D$ instead of $D$, to obtain two sets $E$ and $E'$ of volume $\leq \varepsilon$ (to be chosen later) and two collections of sub-rectangles $R_i$ and $R_i'$   of side length $\asymp\tfrac{1}{D}$ and $\asymp\tfrac{a}{D}$, correspondingly, for some $a$ to be chosen later  subject to  $0<a\ll 1$.   For all $x \in R \backslash (E \cup E')$, $x\in R_i \cap R_j'$ for some $i,j$, and  
	$$\int_{R_i} G \hspace{2mm}d\vol \asymp G(x)\vol(R_i)\asymp  G(x)\vol(R_j')\asymp \int_{R_j'} G \hspace{2mm}d\vol,$$
	where the constant implied in the $\asymp$ notation depend on $n,\varepsilon,a$. We will call $R_j'$ \emph{good} if $2R_j'\subset R_i$ for some $R_i$ and 	$R_j'$ has at least one point in $R\backslash (E\cup E')$. The portion of good $R_j'$ is at least $(1-\delta)$ if $\varepsilon$ and $a$ are small enough (depending on $\delta$). For good $R_j'$ we have 
	$$\int_{2R_j'} G \hspace{2mm}d\vol \leq \int_{R_i} G \hspace{2mm}d\vol \asymp \int_{R_j'} G \hspace{2mm}d\vol,$$
	as required. 
\end{proof}

\begin{rem}
\label{rem: cor DF}
One may change conclusion \eqref{eq: c2.1} to
$$	\int_{R_i} G \hspace{2mm}d\vol \asymp \int_{\Gamma R_i} G \hspace{2mm}d\vol,$$
fo any parameter $\Gamma\geq 2$, and the statement of Corollary \ref{cor: DF} will remain true with $c$ and the constant implied in the $\asymp$ notation also depending on $\Gamma$ (and $\delta>0$). 
\end{rem}

In order to use Proposition \ref{prop: DF} in the proof of Lemma \ref{lem: half of the doubling index}, we will need a claim about the complex extension of a harmonic function $u$, denote by $u^{\mathbb{C}}$, to a subset of $\mathbb{C}^n$.
\begin{claim}
	\label{claim: bound on complexification}
 There exists some constant $C=C(n)\geq 1$ such that for any  harmonic function $u$ on a ball $B\subset \mathbb{R}^n$ centered at $0$ the following holds. There exists a holomorphic function $u^{\mathbb{C}}$ defined in
 $$\Omega:=\left\{x+ iy \in \mathbb{C}^n: x\in \tfrac{1}{5}B \hspace{2mm}, y\in \tfrac{1}{5}B\right\}$$
 such that $u^{\mathbb{C}}(x)= u(x)$ for all $x\in \tfrac{1}{5}B$ and 
 $$\sup_{\Omega} |u^{\mathbb{C}}| \leq C \sup_{B} |u|.$$ 
\end{claim}
\begin{proof}
By considering a slightly smaller ball and rescaling, we may assume that $B=B(0,1)\subset \mathbb{R}^n$ and $u$ is continuous on $\overline{B}$. Then $u$ has a Poisson representation in terms of its boundary values given by 
\begin{align}\nonumber
u(x) = \int_{\partial B} P(x,\zeta) u(\zeta)dS(\zeta), \qquad \text{where} \qquad P(x,\zeta)=  c_n \frac{1 - |x|^2 }{|x-\zeta|^n}.
\end{align}
The holomorphic extension of the Poisson kernel is 
\begin{align}  \label{eq: C1.1}
	P(x+iy,\zeta)=  c_n \frac{1 - \sum_k (x_k +iy_k)^2 }{\left(\sqrt{\sum_k (x_k +iy_k-\zeta_k)^2}\right)^{n}},
\end{align}
where $x=(x_1,...,x_n)$, $y=(y_1,...,y_n)$ and $\zeta=(\zeta_1,...,\zeta_n).$ Indeed, since $\sqrt{z}$ is a well-defined holomorphic function for $|1-z|<1$, it is enough to check
$$\left| 1-\sum_k  (x_k +iy_k-\zeta_k)^2\right|<1.$$
For all $x+iy \in \Omega$, we have 
$$\left|\sum_k (x_k+iy_k)^2\right|\leq 2 \left(\sum_k x_k^2 +y_k^2\right)\leq \frac{4}{25},$$ 
and, since $\zeta \in \partial B$, 
\begin{align}
	\left|1-\sum_k (x_k +iy_k-\zeta_k)^2\right| = \left| \sum_k(x_k +i y_k)^2 + 2\zeta_k(x_k +iy_k)\right| \nonumber \\
	\leq   \left| \sum_k(x_k +i y_k)^2\right| +2 \sqrt{\sum_k \zeta_k^2} \cdot  \sqrt{\sum_k (x_k+iy_k)^2} \leq  \frac{4}{25} + \frac{4}{5}=1-\frac{1}{25}.
\end{align} 
Therefore, $P(x+iy,\zeta)$ is a well-defined holomorphic function of $x+iy$ in  $\Omega$ and $\zeta \in \partial B$. Moreover, 
$$\left|\sum_k (x_k +iy_k-\zeta_k)^2\right| \geq \frac{1}{25}.$$
Thus, the absolute value of the denominator in \eqref{eq: C1.1} is uniformly bounded from below, and 
$$	\sup_{x+iy \in \Omega}|P(x+iy,\zeta)|\leq C_n.$$
Hence, the Poisson representation gives a holomorphic extension of $u$ onto $\Omega$ 
$$u^{\mathbb{C}}(x+iy) = \int_{\partial B} P(x+iy ,\zeta) u(\zeta)dS(\zeta),$$
satisfying 
 $$\sup_{\Omega} |u^{\mathbb{C}}| \leq C \sup_{B} |u|.$$ 
\end{proof}

\begin{lem}
	\label{lem: DF}
	Let $k$ be a sufficiently large integer depending on $n$. Let $u(x)$ be a harmonic function in $|x|\leq 5\cdot 3^k$, $x\in \mathbb{R}^n$, and satisfying 
	$$\max_{|x|\leq 5\cdot 2^k} |u(x)|\leq |u(0)| \exp (D),$$
	for some $D\gg_{n} 1$. Let $Q_0:=\{(x_1,...,x_n)\in \mathbb{R}^n: |x_i|\leq 1, \hspace{2mm} 1\leq i\leq n \}$ and suppose that $R\subset \tfrac{1}{2}Q_0$ is a sub-cube. Let $\varepsilon>0$. Subdivide $R$ into equal sub-cubes $R_i$ of side length $\lesssim_n \frac{c}{D}$ for some sufficiently small $c=c(n,\varepsilon)>0$. Then there exists some constant $C=C(n,\varepsilon)>1$ such that 
	$$\mathcal{N}^{*}(R_i)\leq C,$$
	for at least $(1-\varepsilon)$ portion of the $R_i.$
\end{lem}
\begin{proof}
Thanks to Claim \ref{claim: bound on complexification}, we can extend $u$ holomorphically onto $|z|\leq 3^k$. Therefore, the function $G(z)=(u^{\mathbb{C}}(z))^2$ is holomorphic for $|z|\leq 3^k$, non-negative and real on $Q_0$. Moreover, $G$ satisfies   
	$$\max_{|z|\leq 2^k} |G|\leq |G(0)| \exp (2D).$$
Therefore, we may apply Corollary  \ref{cor: DF} and Remark \ref{rem: cor DF} to see that 
$$\frac{\int_{\Gamma R_i}|u|^2 d\vol}{\int_{R_i}|u|^2 d \vol}\leq C_1,$$
for at least $(1-\varepsilon)$ portion of the $R_i$. Hence, Claim \ref{lem: monotonicity for cubes} finishes the proof of Lemma ~\ref{lem: DF}. 
\end{proof}

We are finally ready to present the proof of Proposition \ref{lem: half of the doubling index}.

\begin{proof}[Proof of Proposition \ref{lem: half of the doubling index}]
	 Let $N=\mathcal{N}^*(Q)$ and partition $Q$ into equal sub-cubes $Q_j$ of side length $\asymp_n\tfrac{c}{N}$ for some $c=c(n,\varepsilon)>0$ to be chosen later. We wish to cover the unit cube $Q$ with at most $C_n$ balls $B_i=B(x_i,r)$ of small radius $0<r=r(n)\ll_n 1$ such that 
	$$\sup_{2^{10k} B_i}|u|\leq u(x_i)\exp(C N),$$
	where $k=k(n)$ is as in Lemma \ref{lem: DF}. We start with any covering of the unit cube by $C_n$ balls $B_i'$ of radius $r/2$. Put $B_i=B(x_i,r)$, and let $x_i$ be the point in $\partial B_i$ such that $\sup_{B_i'}|u| = |u(x_i)|$. For each $B_i'$,  we have  $\mathcal{N}(2^{10k+2} B_i')\leq N$ (by the definition of the doubling index of the cube and assuming that $r$ is sufficiently small). Therefore, the almost monotonicity of the doubling index, Lemma \ref{lem: monotonicity doubling index}, implies  
	$$\sup_{2^{10k} B_i} |u|\leq \sup_{2^{10k+2} B_i'}|u|\leq \sup_{B_i'}|u|\exp(C k N)= |u(x_i)| \exp(2k N).$$
	 We may assume that $k$ is sufficiently large. Consider a cube $R_i$ with the following properties. The faces of $R_i$ are parallel to the faces of $Q$, $R_i$ is a union of a finite number of sub-cubes $Q_j$, and  
	 $$B_i\subset R_i \subset 2^k B_i .$$ 
	Since the $B_i$ cover $Q$, every cube $Q_j$ is contained in at least one of the $R_i$.  We apply rescaled Lemma \ref{lem: DF} to the cube $R_i$ partitioned into the $Q_j$ which lie inside $R_i$. We conclude that, if $c=c(n,\varepsilon)>0$ (the parameter in the side length of the $Q_j$) is sufficiently small, there exists $C=C(n,\varepsilon)>1$ such that the following holds. For at least $(1-\varepsilon)$ portion of the $Q_j$ in $R_i$, we have  
	$$\mathcal{N}^*(Q_j) \leq C.$$
	Finally, considering all the $C_n$ choices of $R_i$, we find that, for	at least $(1-C_n\varepsilon)$ portion of all the $Q_j$, we have  
	$$\mathcal{N}^*(Q_j) \leq C.$$
\end{proof}

\bibliographystyle{siam}
\bibliography{lowerbound.bib}

\end{document}